\documentclass[12pt, reqno]{amsart}
\usepackage{amsmath,amsthm,amsfonts,amscd,amssymb,eucal,latexsym,mathrsfs, appendix, textcomp}
\usepackage{mathtools,mathdots}
\usepackage[numbers,sort&compress]{natbib}
\usepackage{enumitem}

\makeatletter
\@namedef{subjclassname@2020}{%
  \textup{2020} Mathematics Subject Classification}
\makeatother

\newtheorem{theorem}{Theorem}[section]
\newtheorem{corollary}[theorem]{Corollary}
\newtheorem{lemma}[theorem]{Lemma}
\newtheorem{proposition}[theorem]{Proposition}

\theoremstyle{definition}
\newtheorem{definition}[theorem]{Definition}
\newtheorem{remark}[theorem]{Remark}
\newtheorem{notation}[theorem]{Notation}
\newtheorem{example}[theorem]{Example}

\newcommand{\mdim}{{\rm mdim}}

\newcommand{\cA}{{\mathcal A}}
\newcommand{\cB}{{\mathcal B}}

\newcommand{\cF}{{\mathcal F}}

\newcommand{\cK}{{\mathcal K}}

\newcommand{\cM}{{\mathcal M}}

\newcommand{\cP}{{\mathcal P}}

\newcommand{\cU}{{\mathcal U}}
\newcommand{\cV}{{\mathcal V}}

\newcommand{\Nb}{{\mathbb N}}

\newcommand{\Rb}{{\mathbb R}}

\newcommand{\Zb}{{\mathbb Z}}

\newcommand{\sB}{{\mathscr B}}

\newcommand{\sR}{{\mathscr R}}

\newcommand{\sU}{{\mathscr U}}

\newcommand{\oA}{{\boldsymbol{A}}}

\newcommand{\oI}{{\boldsymbol{I}}}

\newcommand{\oU}{{\boldsymbol{U}}}

\newcommand{\of}{{\boldsymbol{f}}}
\newcommand{\ot}{{\boldsymbol{t}}}
\newcommand{\ov}{{\boldsymbol{v}}}
\newcommand{\ow}{{\boldsymbol{w}}}
\newcommand{\ox}{{\boldsymbol{x}}}
\newcommand{\oy}{{\boldsymbol{y}}}

\newcommand{\supp}{{\rm supp}}

\newcommand{\IE}{{\rm IE}}

\newcommand{\IN}{{\rm IN}}
\newcommand{\IT}{{\rm IT}}

\newcommand{\omu}{{\boldsymbol{\mu}}}
\newcommand{\onu}{{\boldsymbol{\nu}}}
\newcommand{\olambda}{{\boldsymbol{\lambda}}}
\newcommand{\orP}{{\boldsymbol{\mathrm{P}}}}

\newcommand{\upind}{{\overline{\rm I}}}

\newcommand{\Sym}{{\rm Sym}}

\newcommand{\rmm}{{\rm m}}

\newcommand{\ext}{{\rm ext}}

\allowdisplaybreaks

\begin{document}

\title[Local entropy theory for induced actions]{Local entropy theory, combinatorics, and local theory of Banach spaces}

\author{Hanfeng Li}
\address{Hanfeng Li,
College of Mathematics and Statistics, Center of Mathematics, Chongqing University, Chongqing 401331, P.R. China
}
\email{hfli@cqu.edu.cn}

\author{Kairan Liu}
\address{Kairan Liu,
College of Mathematics and Statistics, Chongqing University, Chongqing 401331, P.R. China}
\email{lkr111@cqu.edu.cn}

\date{July 4, 2025}


\begin{abstract}
Each continuous action of a countably infinite discrete group $\Gamma$ on a compact metrizable space $X$ induces a continuous  action of $\Gamma$ on the space $\cM(X)$ of Borel probability measures on $X$. We compare the local entropy theory for these two actions, and describe the relation between  their IE-tuples.  Several other types of tuples are also studied. Our main tool is a new combinatorial lemma. We also give an application of the combinatorial lemma to the local theory of Banach spaces. 
\end{abstract}

\maketitle

\tableofcontents

%%%%%%%%%%%%%%%%%%%%%%%%%%%%%%%%%%%%%%%%%%%%%%%%%%%%%%%%%%%%%%%%%%%%%%%%%%%%%%%%%%%%%%%%%%%%%%%%%%%%
\section{Introduction} \label{S-introduction}

Let $\Gamma$ be a countably infinite discrete group and let $\Gamma$ act on a compact metrizable space $X$ continuously. The space of all Borel probability measures on $X$, denoted by $\cM(X)$,   is convex and compact metrizable under the weak$^*$-topology. Then we have the induced continuous affine action of $\Gamma$ on $\cM(X)$. The relation between the dynamical properties of the two actions $\Gamma\curvearrowright X$ and $\Gamma\curvearrowright \cM(X)$ was studied first by Bauer and Sigmund \cite{BS}, and has attracted much attention since then \cite{BDV22, BS22, GW95, KL05, LYY, LW, SZ23, Vermersch}. 

When $\Gamma$ is amenable, one has the topological entropy $h_{\rm top}(X)$ and $h_{\rm top}(\cM(X))$ \cite{AKM} and the mean topological  dimension $\mdim(X)$ and $\mdim(\cM(X))$ \cite{Gromov99, LW00} in $[0, +\infty]$ defined for the actions $\Gamma\curvearrowright X$ and $\Gamma\curvearrowright \cM(X)$ respectively. 
It is trivially true that $h_{\rm top}(X)\le h_{\rm top}(\cM(X))$.
Bauer and Sigmund showed that, when $\Gamma=\Zb$, if $h_{\rm top}(X)>0$, then $h_{\rm top}(\cM(X))=+\infty$ \cite{BS}. Recently Burguet and Shi showed that, when $\Gamma=\Zb$, if $h_{\rm top}(X)>0$, then $\mdim(\cM(X))=+\infty$ \cite{BS22}, which strengthens the Bauer-Sigmund result since actions with positive mean topological dimension always have infinite topological entropy \cite{LW00}. Shi and Zhang further extended the result to all amenable groups $\Gamma$ \cite{SZ23}.

The above results indicate that $h_{\rm top}(\cM(X))$ tends to be much larger than $h_{\rm top}(X)$, which intuitively means that the induced action $\Gamma\curvearrowright \cM(X)$ is much more complicated than the original action $\Gamma\curvearrowright X$. Despite these results, Glasner and Weiss showed the striking result that, for all amenable groups $\Gamma$, if $h_{\rm top}(X)=0$, then $h_{\rm top}(\cM(X))=0$ \cite{GW95}. 

In this work we study the relation between the dynamical properties of $\Gamma\curvearrowright X$ and $\Gamma\curvearrowright \cM(X)$ from the viewpoint of the local entropy theory. The local entropy theory was initiated by Blanchard in \cite{Blanchard92, Blanchard93}, when he sought an analogue of the completely positive entropy systems for topological dynamical systems. It has drawn much attention and found various applications \cite{BDV22, BGH, BHMMR, BS22, Glasner97, Glasner03, HLSY03, HMRY, HMY, HY06, HY09, HYZ, KL09, KL16, LW, Romagnoli, SZ23, Vermersch}, notably in the proofs for that positive entropy systems have plenty of proper one-sided asymptotic pairs \cite{BHR, CL, BGL} and are Li-Yorke chaotic \cite{BGKM, KL07, KL13b, LR}, and in the work of Burguet-Shi and Shi-Zhang mentioned above regarding $\mdim(\cM(X))$ \cite{BS22, SZ23}. See  \cite{GY09, GL} for surveys and references therein. 

One of the key notions in the local entropy theory is that of entropy pairs, introduced by Blanchard \cite{Blanchard93}. A pair of distinct points $(x_1, x_2)$ in $X$ is called an {\it entropy pair} if for any disjoint closed neighborhoods $U_1$ and $U_2$ of $x_1$ and $x_2$ respectively, the open cover $\{X\setminus U_1, X\setminus U_2\}$ of $X$ has positive topological entropy. It is a result of Blanchard that $h_{\rm top}(X)>0$ if and only if there exist entropy pairs \cite{Blanchard93}. The notion of entropy pairs has natural extension to entropy tuples of length $k$ for all integers $k\ge 2$ \cite{GW952}. Though the definition of entropy tuples involves entropy itself, it turns out that they admit an equivalent combinatorial description: one has the notion of IE-tuples of length $k$ defined using combinatorial independence for every integer $k\ge 1$ (see Section~\ref{SS-IE} below),  and for $k\ge 2$ the entropy tuples of length $k$ are exactly the non-diagonal IE-tuples of length $k$, proved first by Huang and Ye using ergodic theoretic method for the case $\Gamma=\Zb$ \cite{HY06} and later by Kerr and Li using combinatorial method for all amenable groups $\Gamma$ \cite{KL07}. We denote the set of all IE-tuples of $X$ with length $k$ by $\IE_k(X)$, which is a closed subset of $X^k$. It behaves well under taking factors or extensions (see Theorem~\ref{T-IE basic}). 

Concerning the actions $\Gamma\curvearrowright X$ and $\Gamma\curvearrowright \cM(X)$, it is tempting to inquire about the relation between the IE-tuples of these two actions. 
Note that $\cM(X)^k$ is naturally a compact convex subset of the locally convex topological vector space $(C(X)^*)^{\oplus k}$, where $C(X)$ is the space of all continuous $\Rb$-valued functions on $X$ and $C(X)^*$ is the dual of $C(X)$ equipped with the weak$^*$-topology. 
We treat $X$ as a closed $\Gamma$-invariant subset of $\cM(X)$ via identifying $x\in X$ with the Dirac measure $\delta_x$ at $x$.
Then $\IE_k(X)\subseteq X^k\subseteq \cM(X)^k$ and $\IE_k(\cM(X))\subseteq \cM(X)^k$. 
For each $\mu\in \cM(X^k)$ and $1\le j\le k$, we denote by $\mu^{(j)}$ the push-forward of $\mu$ in $\cM(X)$ under the $j$-th coordinate map $X^k\rightarrow X$.
Our first main result describes the relation between $\IE_k(X)$ and $\IE_k(\cM(X))$ completely. 

\begin{theorem} \label{T-IE main}
Assume that $\Gamma$ is amenable. For each  $k\in \Nb$, the following hold:
\begin{enumerate}
\item $\IE_k(\cM(X))$ is the closed convex hull of $\IE_k(X)$ in $\cM(X)^k$. 
\item $\IE_k(X)=\IE_k(\cM(X))\cap X^k$ is the set of extreme points of $\IE_k(\cM(X))$. 
\item $\IE_k(\cM(X))=\{(\mu^{(1)}, \dots, \mu^{(k)}): \mu\in \cM(\IE_k(X))\}$.
\end{enumerate}
\end{theorem}

The result of Glasner and Weiss mentioned above is equivalent to saying that there exist non-diagonal IE-pairs for $\Gamma\curvearrowright \cM(X)$ only if there are such pairs for $\Gamma\curvearrowright X$, thus is a direct consequence of the case $k=2$ of Theorem~\ref{T-IE main}. 

Part (3) of Theorem~\ref{T-IE main} gives us an explicit description of elements of $\IE_k(\cM(X))$ in terms of $\IE_k(X)$, and has some interesting applications. For each positive integer $N$, we denote by $\cM_N(X)$ the set of $\mu\in \cM(X)$ whose support has at most $N$ points. This is a closed $\Gamma$-invariant subset of $\cM(X)$, and one has $\IE_k(\cM_N(X))\subseteq \IE_k(\cM(X))$ trivially for every integer $k\ge 1$. Using Theorem~\ref{T-IE main} we can describe more precisely the relation between IE-tuples of $\cM_N(X)$ and $\cM(X)$. A naive guess is that $\cM_N(X)^k\cap \IE_k(\cM(X))=\IE_k(\cM_N(X))$. But this is not true in general, as we give a counterexample in Example~\ref{E-not IE} below. The subtle issue here is that for a tuple $(\mu_1, \dots, \mu_k)$ in $\cM_N(X)^k\cap \IE_k(\cM(X))$, when one finds a $\nu\in \cM(X)$ witnessing the condition that $(\mu_1, \dots, \mu_k)\in \IE_k(\cM(X))$, though one can perturb $\nu$ to have finite support, there is a priori no control on the cardinality of the support of $\nu$. Nevertheless, using the explicit description in part (3) of Theorem~\ref{T-IE main} we are able to show $\cM_N(X)^k\cap\IE_k(\cM(X))\subseteq \IE_k(\cM_{N^k}(X))$. More generally, we have the following consequence. 

\begin{corollary} \label{C-IE finite support}
Assume that $\Gamma$ is amenable. Let $k\in \Nb$. For any $N_1, \dots, N_k\in \Nb$, setting $N=\prod_{j=1}^k N_j$, one has
$$ (\cM_{N_1}(X)\times \cdots \times \cM_{N_k}(X))\cap \IE_k(\cM(X))\subseteq \IE_k(\cM_N(X)).$$
\end{corollary}

Several other types of tuples are studied in the local entropy theory, related to various dynamical properties. When $\Gamma$ is amenable, given a $\Gamma$-invariant Borel probability measure $\mu$ on $X$, the notion of $\mu$-entropy pairs was introduced by Blanchard, Host, Maass, Martinez and Rudolph \cite{BHMMR}. A pair of distinct points $(x_1, x_2)$ in $X$ is called a {\it $\mu$-entropy pair} if for any Borel partition $\{A_1, A_2\}$ of $X$ such that $A_j$ is a neighborhood of $x_j$ for $j=1, 2$, the measure entropy of $\{A_1, A_2\}$ with respect to $\mu$ is positive. A result of Blanchard et al. says that the measure entropy $h_\mu(X)$ of the action $\Gamma\curvearrowright (X, \mu)$ is positive if and only if there exist $\mu$-entropy pairs \cite{BHMMR}. The notion of $\mu$-entropy pairs also has natural extension to $\mu$-entropy tuples of length $k$ for all integers $k\ge 2$ \cite{HY06}. The measure entropy tuples and the entropy tuples are linked as one expects from the variational principle: the set of entropy tuples is the union of $\mu$-entropy tuples for $\mu$ ranging over all $\Gamma$-invariant Borel probability measures on $X$ \cite{BGH, BHMMR}. It turns out that $\mu$-entropy tuples also admit a combinatorial description: Kerr and Li introduced the notion of $\mu$-IE-tuples of length $k$, defined using combinatorial independence like that for IE-tuples but in a much more sophisticated way, for all integers $k\ge 1$ (see Section~\ref{SS-measure IE} below), and showed that for each $k\ge 2$ the $\mu$-entropy tuples of length $k$ are exactly the non-diagonal $\mu$-IE-tuples of length $k$ \cite{KL09}. We denote by $\IE_k^\mu(X)$ the set of all $\mu$-IE-tuples of $X$ with length $k$. 

In the context of the actions $\Gamma\curvearrowright X$ and $\Gamma\curvearrowright \cM(X)$, one may start with a $\Gamma$-invariant Borel probability measure $\rmm$ on $\cM(X)$, and consider its barycenter $\mu$ (see Section~\ref{SS-barycenter} below), which is a $\Gamma$-invariant Borel probability measure on $X$. In the literature the measure-preserving action $\Gamma\curvearrowright (\cM(X), \rmm)$ is called a {\it quasi-factor} of the measure-preserving action $\Gamma\curvearrowright (X, \mu)$ \cite{Glasner83, Glasner03, GTW, GW95, GW03}, since every factor of $\Gamma\curvearrowright (X, \mu)$ gives rise to one $\rmm$ through measure disintegration. In \cite{GW95} Glasner and Weiss gave two proofs for their result above regarding $h_{\rm top}(X)$ and $h_{\rm top}(\cM(X))$. The first proof is ergodic theoretic, and at the heart of this proof Glasner and Weiss actually showed that, when $\Gamma=\Zb$, if $h_\mu(X)=0$, then $h_{\rmm}(\cM(X))=0$.
%the measure entropy of $\Gamma\curvearrowright (X, \mu)$ is $0$, then so is that of $\Gamma\curvearrowright (\cM(X), \rmm)$. 
This was later extended to all amenable groups $\Gamma$ by Glasner, Thouvenot, and Weiss in \cite{GTW} (it also follows from the explicit description of all quasi-factors in \cite{GW03}). 

Given the Glasner-Thouvenot-Weiss result, it is natural to explore an analogue of Theorem~\ref{T-IE main} for measure IE-tuples by investigating the relationship between $\IE_k^{\rmm}(\cM(X))$ and $\IE_k^{\mu}(X)$. However, it is possible that $h_\mu(X)>0$ while $h_{\rmm}(\cM(X))=0$,
because if one starts with any $\Gamma$-invariant Borel probability measure $\mu$ on $X$ and take $\rmm$ to the Dirac measure $\delta_\mu$ on $\cM(X)$ at $\mu$, which corresponds to the trivial factor of $\Gamma\curvearrowright (X,\mu)$, then $h_{\rmm}(\cM(X))=0$.
Thus
one cannot expect the full analogue of Theorem~\ref{T-IE main} to hold for measure IE-tuples. 
Treating $X$ as a closed subset of $\cM(X)$ as earlier, we still have 
$\IE_k^\mu(X)\subseteq X^k\subseteq \cM(X)^k$ and $\IE_k^\rmm(\cM(X))\subseteq \cM(X)^k$. 
We denote by $\cM_\Gamma(X)$ the set of all $\Gamma$-invariant Borel probability measures on $X$. 
Our main result about measure IE-tuples is the following.

\begin{theorem} \label{T-measure IE main}
Assume that $\Gamma$ is amenable. Let $\rmm\in \cM_\Gamma(\cM(X))$ and let $\mu\in \cM_\Gamma(X)$ be the barycenter of $\rmm$. For each $k\in \Nb$, the following hold:
\begin{enumerate}
\item $\IE^\rmm_k(\cM(X))$ is contained in the closed convex hull of $\IE^\mu_k(X)$ in $\cM(X)^k$. 
\item $\IE^\rmm_k(\cM(X))\cap X^k\subseteq \IE^\mu_k(X)$.
\end{enumerate}
\end{theorem}

Again, the Glasner-Thouvenot-Weiss result mentioned above regarding $h_\mu(X)$ and $h_{\rmm}(\cM(X))$
is a direct consequence of the case $k=2$ of Theorem~\ref{T-measure IE main}. 

In the rest of this introduction, we shall let $\Gamma$ be an arbitrary countably infinite discrete group, not assuming amenability of $\Gamma$. For any sequence $\mathfrak{s}$ in $\Gamma$, one has the sequence topological entropy $h_{\rm top}(X; \mathfrak{s})$ of $\Gamma\curvearrowright X$ with respect to $\mathfrak{s}$ defined \cite{Goodman}, 
much like the definition of the topological entropy for actions of $\Zb$. We say the action $\Gamma\curvearrowright X$ is {\it null} if $h_{\rm top}(X; \mathfrak{s})=0$
for every sequence $\mathfrak{s}$ in $\Gamma$. Huang, Li, Shao and Ye introduced the notion of sequence entropy pairs \cite{HLSY03}, similar to the definition of entropy pairs. A pair of distinct points $(x_1, x_2)$ in $X$ is called a {\it sequence entropy pair} if  for any disjoint closed neighborhoods $U_1$ and $U_2$ of $x_1$ and $x_2$ respectively, the open cover $\{X\setminus U_1, X\setminus U_2\}$ of $X$ has positive sequence topological entropy for some sequence $\mathfrak{s}$ in $\Gamma$. Huang et al. showed that the action $\Gamma\curvearrowright X$ is nonnull if and only if there exist sequence entropy pairs \cite{HLSY03}. The notion of sequence entropy pairs also has natural extension to sequence entropy tuples of length $k$ for all integers $k\ge 2$ \cite{HMY}. It turns out that sequence entropy tuples also admit a combinatorial description: Kerr and Li introduced the notion of IN-tuples of length $k$ for all integers $k\ge 1$ using combinatorial independence (see Section~\ref{SS-IN} below), and showed that for each $k\ge 2$ the sequence entropy tuples of length $k$ are exactly the non-diagonal IN-tuples of length $k$ \cite{KL07}. Furthermore, Huang and Ye showed that the so-called maximal pattern entropy of $\Gamma\curvearrowright X$ can be calculated using the maximal length of sequence entropy tuples or IN-tuples with distinct coordinates \cite{HY09}. We denote by $\IN_k(X)$ the set of all IN-tuples of $X$ with length $k$.

It is natural to inquire about the analogue of Theorem~\ref{T-IE main} for IN-tuples. 
Treating $X$ as a closed subset of $\cM(X)$, we still have 
$\IN_k(X)\subseteq X^k\subseteq \cM(X)^k$ and $\IN_k(\cM(X))\subseteq \cM(X)^k$.
However, $\IN_k(\cM(X))$ could fail to be convex in general, even for minimal subshifts of $\Gamma=\Zb$ (see Section~\ref{S-example} for an example). Thus one cannot expect the direct analogue of Theorem~\ref{T-IE main} to hold for IN-tuples. In fact, 
the map $\cM(X^k)\rightarrow \cM(X)^k$ sending $\mu$ to $(\mu^{(1)}, \dots, \mu^{(k)})$ is affine, whence the set of $(\mu^{(1)}, \dots, \mu^{(k)})\in \cM(X)^k$ for all $\mu\in \cM(\IN_k(X))$ is always convex. Thus $\IN_k(\cM(X))$ cannot be equal to this set unless $\IN_k(\cM(X))$ is convex. This leaves little hope for an analogue of part (3) of Theorem~\ref{T-IE main} for IN-tuples. 
Nevertheless, it is somehow surprising that one can still obtain an explicit description of elements of $\IN_k(\cM(X))$ like in part (3) of Theorem~\ref{T-IE main}, though we have to replace $\cM(\IN_k(X))$ by $\cM_\IN(X^k)$, which is a closed subset of $\cM(\IN_k(X))$ and whose definition involves not only $\IN_k(X)$ but $\IN_k(X^n)$ for all $n\in \Nb$ (see Definition~\ref{D-IN measure}). Our main result about IN-tuples is the following. 

\begin{theorem} \label{T-IN main}
For each $k\in \Nb$, the following hold:
\begin{enumerate}
\item $\IN_k(\cM(X))$ is contained in the closed convex hull of $\IN_k(X)$ in $\cM(X)^k$. 
\item $\IN_k(\cM(X))\cap X^k=\IN_k(X)$.
\item $\IN_k(\cM(X))=\{(\mu^{(1)}, \dots, \mu^{(k)}):\mu\in \cM_\IN(X^k)\}$.
\end{enumerate}
\end{theorem}

Though $\cM_\IN(X^k)$ is more complicated than $\cM(\IN_k(X))$, part (3) of Theorem~\ref{T-IN main} still enables us to deduce the following analogue of Corollary~\ref{C-IE finite support}.

\begin{corollary} \label{C-IN finite support}
Let $k\in \Nb$.
For any $N_1, \dots, N_k\in \Nb$, setting $N=\prod_{j=1}^k N_j$, one has
$$ (\cM_{N_1}(X)\times \cdots \times \cM_{N_k}(X))\cap \IN_k(\cM(X))\subseteq \IN_k(\cM_N(X)).$$
\end{corollary}

Given a $\Gamma$-invariant Borel probability measure $\mu$ on $X$, one can also define the sequence measure entropy $h_\mu(X; \mathfrak{s})$ of $\Gamma\curvearrowright (X, \mu)$ with respect to any sequence $\mathfrak{s}$ in $\Gamma$ \cite{Kushnirenko}, much as one defines the measure entropy for measure-preserving actions of $\Zb$. We say the action $\Gamma\curvearrowright (X, \mu)$ is {\it null} if $h_\mu(X; \mathfrak{s})=0$
for every sequence $\mathfrak{s}$ in $\Gamma$. The action $\Gamma\curvearrowright (X, \mu)$ is null 
if and only if it 
has discrete spectrum in the sense that every $f\in L^2(X, \mu)$ has compact orbit closure \cite{Kushnirenko, KL09}. Huang, Maass and Ye introduced the notion of $\mu$-sequence entropy tuples of length $k$ for integers $k\ge 2$ \cite{HMY}, much like the definition of measure entropy tuples. In the case $k=2$, a pair of distinct points $(x_1, x_2)$ in $X$ is called a {\it $\mu$-sequence entropy pair} if for any Borel partition $\{A_1, A_2\}$ of $X$ such that $A_j$ is a neighborhood of $x_j$ for $j=1, 2$, the sequence entropy of $\{A_1, A_2\}$ with respect to $\mu$ and some sequence $\mathfrak{s}$ in $\Gamma$ is positive. Huang et al. showed that the 
action $\Gamma\curvearrowright (X, \mu)$ is nonnull 
if and only if there exist $\mu$-sequence entropy pairs \cite{HMY}. The measure sequence entropy tuples and the sequence entropy tuples satisfy half of the variational principle: the measure sequence entropy tuples are automatically sequence entropy tuples \cite{HMY}, but the union of the measure sequence entropy tuples for all invariant Borel probability measures may be a proper subset of the set of sequence entropy tuples. It turns out that the $\mu$-sequence entropy tuples also admit a combinatorial description: Kerr and Li introduced the notion of $\mu$-IN-tuples of length $k$ for every integer $k\ge 1$ using combinatorial independence (see Section~\ref{SS-measure IN} below), and showed that for each $k\ge 2$ the $\mu$-sequence entropy tuples of length $k$ are exactly the non-diagonal $\mu$-IN-tuples of length $k$ \cite{KL09}. We denote by $\IN^\mu_k(X)$ the set of all $\mu$-IN-tuples of $X$ with length $k$. Then we have the following analogue of Theorem~\ref{T-measure IE main} for $\mu$-IN-tuples. 

\begin{theorem} \label{T-measure IN main}
Let $\rmm\in \cM_\Gamma(\cM(X))$ and let $\mu\in \cM_\Gamma(X)$ be the barycenter of $\rmm$. For each $k\in \Nb$, the following hold:
\begin{enumerate}
\item $\IN^\rmm_k(\cM(X))$ is contained in the closed convex hull of $\IN^\mu_k(X)$ in $\cM(X)^k$.
\item $\IN^\rmm_k(\cM(X))\cap X^k\subseteq \IN^\mu_k(X)$.
\end{enumerate}
\end{theorem}

Rosenthal's characterization of Banach spaces containing $\ell_1$ isomorphically \cite{Rosenthal74, Rosenthal78} initiated a line of research based on Ramsey methods that led to the work of Bourgain, Fremlin and Talagrand on pointwise compact sets of Baire class one functions \cite{BFT} (see \cite{Gowers, Todorcevic} for general references). Transferring these ideas to dynamical systems, 
Köhler introduced the notion of tameness \cite{Kohler}. The action $\Gamma\curvearrowright X$ is said to be {\it untame} if there are an $f\in C(X)$
%, the space of continuous $\Rb$-valued functions on $X$, 
and an infinite subset $F$ of $\Gamma$ such that the map sending $s\in F$ to $sf$ extends to a Banach space isomorphism from $\ell_1(F)$ to its image in $C(X)$. 
The tameness turns out to be a highly interesting dynamical property and is closely related to the Ellis semigroup \cite{FGJO, FKY, Glasner06, Glasner07, Glasner18, GM06, GM18, GM22, Huang06, KL07}. For instance, the action $\Gamma\curvearrowright X$ is tame if and only if its Ellis semigroup has cardinality at most that of the real numbers \cite{Glasner06}. Inspired by the combinatorial independence in Rosenthal's work, Kerr and Li introduced the notion of IT-tuples of length $k$ for integers $k\ge 1$ (see Section~\ref{SS-IT} below), and showed that $\Gamma\curvearrowright X$ is untame if and only if there exist non-diagonal IT-pairs \cite{KL07}. We denote the set of all IT-tuples of $X$ with length $k$ by $\IT_k(X)$. By definition IT-tuples are IN-tuples. Very recently, Liu, Xu and Zhang showed the surprising result that measure sequence entropy tuples are IT-tuples \cite{LXZmeasure}. 

As the case of IN-tuples, in general the set $\IT_k(\cM(X))$ may fail to be convex (see Section~\ref{S-example} for an example). Though the combinatorics behind IN-tuples and IT-tuples are quite different, it turns out that the analogue of Theorem~\ref{T-IN main} still holds for IT-tuples. In the theorem below $\cM_\IT(X^k)$ is certain closed subset of $\cM(\IT_k(X))$ defined using $\IT_k(X^n)$ for all $n\in \Nb$ (see Definition~\ref{D-IT measure}).

\begin{theorem} \label{T-IT main}
For each $k\in \Nb$, the following hold:
\begin{enumerate}
\item $\IT_k(\cM(X))$ is contained in the closed convex hull of $\IT_k(X)$ in $\cM(X)^k$. 
\item $\IT_k(\cM(X))\cap X^k=\IT_k(X)$.
\item $\IT_k(\cM(X))=\{(\mu^{(1)}, \dots, \mu^{(k)}):\mu\in \cM_\IT(X^k)\}$.
\end{enumerate}
\end{theorem}

The analogue of Corollaries~\ref{C-IE finite support} and \ref{C-IN finite support} for IT-tuples also holds, see Corollary~\ref{C-IT finite support}.

The proof of Theorem~\ref{T-IE main} is mainly combinatorial, besides using the standard results in functional analysis. The case $k=1$ is very easy, as typical for results about IE-tuples. The case $k=2$ requires some combinatorial fact, Lemma~\ref{L-half to indep}, which is non-trivial but not too difficult to prove. Typically, proofs for results about IE-tuples in the case $k=2$ also work for the case $k\ge 3$. It is surprising that the case $k\ge 3$ of Theorem~\ref{T-IE main} is significantly more difficult to prove than the case $k=2$. The issue is that though Lemma~\ref{L-half to indep} holds for all integers $k\ge 1$, when $k\ge 3$, the dynamical condition does not imply the combinatorial condition in  Lemma~\ref{L-half to indep} (see Example~\ref{E-sum3}). The combinatorial results employed in the study of IE-tuples so far all originate from the Sauer-Perles-Shelah lemma \cite{Sauer, Shelah} and its generalization by Karpovsky and Milman \cite{KM} (see \cite{Bollobas, FT} for general references and  Lemma~\ref{L-KM} below for the asymptotic form used in the local entropy theory). Both Lemma~\ref{L-half to indep}  and the key combinatorial lemma in \cite{KL07} (see Lemma~\ref{L-07} below) depend on this generalization. In order to prove the case $k\ge 3$ for Theorem~\ref{T-IE main}, we establish a new combinatorial result. For a finite set $Z$ and $k\in \Nb$, we denote by $\sR(Z, k)$ the set of maps $\psi: Z\rightarrow [k]=\{1, \dots, k\}$ which are as equidistributed as possible (see Notation~\ref{N-regular}). For a function $f: Z\rightarrow \Rb$, the notation $\|f\|_p$ denotes the standard $\ell_p$-norm of $f$ (see \eqref{E-p norm1} and \eqref{E-p norm2} below). Our key combinatorial lemma is the following. 

\begin{lemma} \label{L-func to indep}
Let $k\in \Nb$ and $0<r<R\le C$. Also let $1<p\le \infty$ and $1\le q<\infty$ such that $1/p+1/q=1$. Then there exist $c, N>0$ and a finite set $T\subseteq \Rb^k$ depending only on $k, r, R, C$ and $p$ such that the following hold:
\begin{enumerate}
\item $\frac{1}{k}\sum_{j\in [k]} t_j\ge r$ for every $(t_1, \dots, t_k)\in T$, and
\item for any finite set $Z$ with $|Z|\ge N$, if for each $\psi\in \sR(Z, k)$  we take an $f_\psi: Z\rightarrow \Rb$  with $\|f_\psi\|_p\le C$ and $\frac{1}{|Z|^{1/q}}\sum_{z\in Z}f_\psi(z)\ge R$, then there are some $J\subseteq Z$ with $|J|\ge c|Z|$ and $(t_1, \dots, t_k)\in T$  such that every map $\sigma: J\rightarrow [k]$ extends to a $\psi\in \sR(Z, k)$ so that $f_\psi(z)\ge t_j|Z|^{-1/p}$ for all $j\in [k]$ and $z\in \sigma^{-1}(j)$.
\end{enumerate}
\end{lemma}

The proofs of Theorems~\ref{T-measure IE main}, \ref{T-IN main} and \ref{T-measure IN main} also depend on Lemma~\ref{L-func to indep}. One may also prove Theorem~\ref{T-IN main} using Ramsey's theorem in \cite{Ramsey}, but we choose to prove it using Lemma~\ref{L-func to indep} to keep the proofs streamlined. The proof of Theorem~\ref{T-IT main} uses a dichotomy of Rosenthal in \cite{Rosenthal74}, and does not require any new combinatorial facts. 

The formulation of Lemma~\ref{L-func to indep} is quite different and much more complicated than that of the Sauer-Perles-Shelah lemma and its generalization by Karpovsky and Milman, and the key combinatorial lemma in \cite{KL07}, as in application one needs to choose the function $f_\psi$ suitably. This complication also yields extra power to the lemma. In fact the proofs of  Theorems~\ref{T-IE main}, \ref{T-measure IE main}, \ref{T-IN main} and \ref{T-measure IN main} only need the case $p=\infty$ of Lemma~\ref{L-func to indep}. The full range $1<p\le \infty$ enables us to obtain certain application to the Banach space geometry to which we now turn. 

The local theory of Banach spaces, also called asymptotic geometric analysis, studies Banach spaces via their finite-dimensional subspaces or quotients \cite{AGM15, AGM21, MS, Pisier, TJ}. For two finite-dimensional 
Banach spaces $V$ and $W$ of the same dimension, 
we denote by $d_{\rm BM}(V, W)$ the (multiplicative) {\it Banach-Mazur distance} between them, defined as the infimum of $\|\Phi\|\cdot \|\Phi^{-1}\|$ for $\Phi$ ranging over linear isomorphisms between them. 
For a constant $C\ge 1$, we say $V$ and $W$ are {\it $C$-isomorphic} if $d_{\rm BM}(V, W)\le C$.
One of the important questions is, given $C>1$ and $1\le q, q'\le \infty$,  to find the asymptotic bounds for $m$ and $n$ such that 
$\ell_q^n$ is $C$-isomorphic to a linear subspace of $\ell_{q'}^m$, where $\ell_q^n$ denotes $\Rb^n$ equipped with the $\ell_q$-norm $\|\cdot \|_q$. 
This was solved by Figiel, Lindenstrauss and Milman for the case $q=2$ in their monumental work \cite{FLM} using the concentration of measure phenomenon on the Euclidean spheres. They showed that, when $q=2$, the optimal value of $n$ is $c\log m$ when $q'=\infty$, $cm^{2/q'}$ when $2\le q'<\infty$, and $cm$ when $1\le q'\le 2$.  Using Lemma~\ref{L-func to indep} we are able to establish the bounds in the case $q'=\infty$, much like the way the Sauer-Perles-Shelah lemma is used in the proof of the Elton-Pajor theorem \cite{Elton, Pajor83, Pajor85} \cite[Section 6.5]{AGM21}. 

\begin{theorem} \label{T-FLM}
Let $C>1$ and $1\le q<\infty$. Then there is some $c>0$ depending only on $C$ and $q$ such that if $\ell_q^n$ is $C$-isomorphic to a linear subspace of $\ell_\infty^m$ for some $m\ge 2$, 
then $n\le c\log m$. 
\end{theorem}

The case $1\le q\le 2$ follows from the above bounds of Figiel, Lindenstrauss and Milman. In fact, the second proof of Glasner and Weiss in \cite{GW95} for their result regarding $h_{\rm top}(X)$ and $h_{\rm top}(\cM(X))$ is geometric, and uses both the Sauer-Perles-Shelah lemma and the case $q=1$ of Theorem~\ref{T-FLM}. For $2<q<\infty$, the above bounds of Figiel, Lindenstrauss and Milman also yield a bound $n\le c (\log m)^{q/2}$, which is weaker than that in Theorem~\ref{T-FLM}. 

The bound in Theorem~\ref{T-FLM} is also optimal (see Lemma~\ref{L-embedding exist}). 

This paper is organized as follows. In Section~\ref{S-prelim} we set up some notations and recall some preliminary facts about barycenters, amenable groups, IE-tuples, IN-tuples, IT-tuples and their measure versions. We study the IE-tuples of $\cM(X)$ and prove Theorem~\ref{T-IE main} and Corollary~\ref{C-IE finite support} in Section~\ref{S-IE for prob}. Our key combinatorial tool Lemma~\ref{L-func to indep} is also proved there. Section~\ref{S-measure IE for prob} is devoted to the study of measure IE-tuples of $\cM(X)$ and the proof of Theorem~\ref{T-measure IE main}. We study IN-tuples of $\cM(X)$ and prove Theorem~\ref{T-IN main} and Corollary~\ref{C-IN finite support} in Section~\ref{S-IN for prob}. The case of measure IN-tuples is studied in Section~\ref{S-measure IN for prob},  where Theorem~\ref{T-measure IN main} is proved. We handle the IT-tuples and prove Theorem~\ref{T-IT main} in Section~\ref{S-IT for prob}. We construct a Toeplitz subshift with non-convex $\IN_2(\cM(X))$ and $\IT_2(\cM(X))$ in Section~\ref{S-example}. Theorem~\ref{T-FLM} is proved in Section~\ref{S-Banach}. 

\medskip

\noindent{\it Acknowledgements.}
K.~L. was partially supported by the China Postdoctoral Science Foundation (Grant No. 2022M710527) and the National Natural Science Foundation of China (Grant No. 12301224). We thank Sebastián Barbieri, Felipe García-Ramos and  Xiangdong Ye for helpful comments. 

%%%%%%%%%%%%%%%%%%%%%%%%%%%%%%%%%%%%%%%%%%%%%%%%%%%%%%%%%%%%%%%%%%%%%%%%%%%%%%%%%%%%%%%%%%%%%%%%%%%%%%%%%%%%%%%%%%%%%%%%%%%%%%%%%%%%%%%%%%%%%%%%%%%%%%%%%%
\section{Preliminaries} \label{S-prelim}

Throughout this article, all logarithms are taken with base $e$ and all vector spaces are over $\Rb$.  For $k\in \Nb$, we write $[k]$ for $\{1, \dots, k\}$. For a subset $E$ of an ambient set, we write ${\bf 1}_E$ for the characteristic function of $E$.

Let a countably infinite discrete group $\Gamma$ act on a compact metrizable space $X$ continuously.
We denote by $\cF(\Gamma)$ the set of all nonempty finite subsets of $\Gamma$. 
%For $k\in \Nb$, we write $\Delta_k(X)$ for the diagonal $\{(\underbrace{x, \dots, x}_{k}): x\in X\}$ in $X^k$.
We write $\Delta_2(X)$ for the diagonal $\{(x, x): x\in X\}$ in $X^2$.

We write $C(X)$ for the space of all continuous functions $f: X\rightarrow \Rb$, equipped with the supremum norm $\|\cdot \|$. For any $k\in \Nb$, we write $C(X)^{\oplus k}$ for the direct sum of $k$ copies of $C(X)$, equipped with the norm $\|\cdot \|$ given by
$$ \|\of\|=\max_{j\in [k]}\|f_j\|$$
for $\of=(f_1, \dots, f_k)\in C(X)^{\oplus k}$. We write $(C(X)^{\oplus k})^*$ for   the dual space   of $C(X)^{\oplus k}$, i.e., all continuous linear functionals $\varphi: C(X)^{\oplus k}\rightarrow \Rb$, and equip it with the weak$^*$-topology. For any $W\subseteq (C(X)^{\oplus k})^*$, we denote by $\overline{\rm co}(W)$ the closed convex hull of $W$. We denote by $\cM(X)$ the space of all Borel probability measures on $X$, and by $\cM_\Gamma(X)$ the space of all $\Gamma$-invariant Borel probability measures on $X$. Via the Riesz representation theorem, we may identity $\cM(X)$ with a compact metrizable convex subset of $C(X)^*$ naturally. In particular, for $\mu\in \cM(X)$ and $f\in C(X)$ we shall write $\mu(f)$ for the integral $\int_Xf(x)\, d\mu(x)$.
We treat $X$ as a closed $\Gamma$-invariant subset of $\cM(X)$ via identifying $x\in X$ with the Dirac measure $\delta_x$ on $X$ at $x$. 

For $k\in \Nb$, $\mu\in \cM(X^k)$ and $j\in [k]$, we write $\mu^{(j)}$ for the push-forward of $\mu$ in $\cM(X)$ under the $j$-th coordinate map $X^k\rightarrow X$.

Let $k\in \Nb$, and let $\oA=(A_1, \dots, A_k)$ be a tuple of subsets of $X$. We say a set $F\subseteq \Gamma$ is an {\it independence set} for $\oA$ if 
$$\bigcap_{s\in J}s^{-1}A_{\sigma(s)}\neq \emptyset$$ 
for every nonempty finite set $J\subseteq F$ and every map $\sigma: J\rightarrow [k]$.

For another continuous action $\Gamma\curvearrowright Y$ of $\Gamma$ on a compact metrizable space $Y$, a continuous surjective $\Gamma$-equivariant map $\pi: X\rightarrow Y$ is called a {\it factor map}. 

For finite covers $\cU$ and $\cV$ of $X$, we denote by $\cU\vee \cV$ the cover of $X$ consisting of $U\cap V$ for $U\in \cU$ and $V\in \cV$. For a finite open cover $\cU$ of $X$, we write $N(\cU)$ for the minimal number of elements of $\cU$ needed to cover $X$. For $\mu\in \cM(X)$ and a finite Borel partition $\cP$ of $X$, the {\it Shannon entropy} of $\cP$ with respect to $\mu$ is defined as 
$$H_\mu(\cP)=\sum_{P\in \cP}g(\mu(P)),$$
where $g: [0, 1]\rightarrow \Rb$ is the continuous function defined by
\begin{align} \label{E-Shannon}
g(t)=\begin{cases}
-t\log t & \quad \text{if } 0<t\le 1\\
0 & \quad \text{if } t=0.\\
\end{cases}
\end{align}

%%%%%%%%%%%%%%%%%%%%%%%%%%%%%%%%%%%%%%%%%%%%%%%%%%%%%%%%%%%%%%%%%%%%%%%%%%%%%%%%%%%%%%%%%%%%%%
\subsection{Barycenter} \label{SS-barycenter}

For a compact convex set $\cK$ in a Hausdorff locally convex topological vector space $V$, we denote by $\ext\, \cK$ the set of extreme points of $\cK$. When $\cK$ is metrizable, for every $\mu\in \cM(\cK)$, there is a unique $x\in \cK$ such that
$$\int_{\cK} \phi(v)\, d\mu(v)=\phi(x)$$
for every continuous linear functional $\phi: V\rightarrow \Rb$ \cite[Proposition 1.1]{Phelps}. This $x$ is called the {\it barycenter} of $\mu$. The map $\Psi: \cM(\cK)\rightarrow \cK$ sending $\mu$ to its barycenter is continuous and affine in the sense that $\Psi(\lambda \mu+(1-\lambda)\nu)=\lambda \Psi(\mu)+(1-\lambda)\Psi(\nu)$ for all $\lambda\in [0, 1]$ and $\mu, \nu\in \cM(\cK)$.

\subsection{Amenable groups} \label{SS-amenable}

The group $\Gamma$ is said to be {\it amenable} if for any $K\in \cF(\Gamma)$ 
and any $\varepsilon>0$ there is an $F\in \cF(\Gamma)$ 
such that
\begin{align} \label{E-Folner}
 |KF\Delta F|<\varepsilon |F|,
\end{align}
where $A\Delta B:=(A\setminus B)\cup (B\setminus A)$ for subsets $A, B$ of $\Gamma$.

Assume that $\Gamma$ is amenable. For a function $\varphi: \cF(\Gamma)\rightarrow \Rb$, we say that {\it $\varphi(F)$ converges to $L\in \Rb$ when $F\in \cF(\Gamma)$ becomes more and more left invariant} if for any $\delta>0$, there are some $K\in \cF(\Gamma)$ and $\varepsilon>0$ such that
$$|\varphi(F)-L|<\delta$$
for every $F\in \cF(\Gamma)$ satisfying \eqref{E-Folner}. In general, we say that {\it the limit supremum of $\varphi(F)$ as $F\in \cF(\Gamma)$ becomes more and more left invariant is $L\in \Rb$} if $L$ is the smallest $t\in \Rb$ satisfying that for any $\delta>0$, there are some $K\in \cF(\Gamma)$ and $\varepsilon>0$ such that
$$\varphi(F)<t+\delta$$
for every $F\in \cF(\Gamma)$ satisfying \eqref{E-Folner}.

%%%%%%%%%%%%%%%%%%%%%%%%%%%%%%%%%%%%%%%%%%%%%%%%%%%%%%%%%%%%%%%%%%%%%%%%%%%%%%%%%%%%
\subsection{IE-tuples} \label{SS-IE}

Assume that $\Gamma$ is amenable.
Let $k\in \Nb$, and let $\oA=(A_1, \dots, A_k)$ be a tuple of subsets of $X$. For each $F\in \cF(\Gamma)$, we set
$$ \varphi_{\oA}(F)=\max\{|J|: J\subseteq F \mbox{ is an independence set for } \oA\}.$$
The {\it independence density} of $\oA$ is defined as
$$ \oI(\oA)=\inf_{F\in \cF(\Gamma)}\frac{\varphi_\oA(F)}{|F|}.$$
When $F\in \cF(\Gamma)$ becomes more and more left invariant, $\frac{\varphi_\oA(F)}{|F|}$ converges to $\oI(\oA)$ \cite[page 287]{KL16}.
We call $\ox\in X^k$ an {\it IE-tuple} if for every product neighborhood $U_1\times \cdots \times U_k$ of $\ox$ in $X^k$, the tuple $\oU=(U_1, \dots, U_k)$ has positive independence density, i.e., $\oI(\oU)>0$. We denote by $\IE_k(X)$ the set of all IE-tuples of length $k$.

We refer the reader to \cite{MO, KL16} for general information about topological entropy. 
For each finite open cover $\cU$ of $X$, when $\cF\in \cF(\Gamma)$ becomes more and more left invariant, $\frac{1}{|F|}\log N\left(\bigvee_{s\in F}s^{-1}\cU\right)$ converges to a limit \cite[page 220]{KL16}, denoted by $h_{\rm top}(X, \cU)$. The {\it topological entropy} of $\Gamma\curvearrowright X$ is defined as $\sup_\cU  h_{\rm top}(X, \cU)$ for $\cU$ ranging over finite open covers of $X$, and is denoted by $h_{\rm top}(X)$. 

The following summarizes the basic properties of IE-tuples we need \cite[Theorem 12.19 and page 289]{KL16}.

\begin{theorem} \label{T-IE basic}
Let $k\in \Nb$. The following hold:
\begin{enumerate}
\item \label{i-IE invariant} $\IE_k(X)$ is a closed $\Gamma$-invariant subset of $X^k$.
\item \label{i-IE tuple} Let $(A_1, \dots, A_k)$ be a tuple of closed subsets of $X$ with positive independence density. Then there is an IE-tuple $\ox\in A_1\times \cdots \times A_k$.
\item \label{i-IE pair} $\IE_2(X)\setminus \Delta_2(X)$ is nonempty if and only if $h_{\rm top}(X)>0$.
\item \label{i-IE factor} Let $\Gamma\curvearrowright Y$ be another continuous action of $\Gamma$ on a compact metrizable space $Y$, and  let $\pi:X\rightarrow Y$ be a factor map. Then $(\pi\times \cdots \times \pi)(\IE_k(X))=\IE_k(Y)$.
\item \label{i-IE subset} Let $Z$ be a closed $\Gamma$-invariant subset of $X$. Then $\IE_k(Z)\subseteq \IE_k(X)$.
\item \label{i-IE product} If $\Gamma\curvearrowright Y$ is another continuous action of $\Gamma$ on a compact metrizable space $Y$, then
$$ \IE_k(X\times Y)=\IE_k(X)\times \IE_k(Y)$$
under the natural identification of $(X\times Y)^k$ and $X^k\times Y^k$.
\item \label{i-IE 1} $\IE_1(X)$ is the closure of $\bigcup_{\mu\in \cM_\Gamma(X)}\supp(\mu)$.
\end{enumerate}
\end{theorem}

%%%%%%%%%%%%%%%%%%%%%%%%%%%%%%%%%%%%%%%%%%%%%%%%%%%%%%%%%%%%%%%%%%%%%%%%%%%%%%%%%%%%%%%%%%%%%%
\subsection{Measure IE-tuples} \label{SS-measure IE}

Assume that $\Gamma$ is amenable.
Let $k\in \Nb$, let $\oA=(A_1, \dots, A_k)$ be a tuple of subsets of $X$, and let $D\subseteq X$. We say $F\in \cF(\Gamma)$ is an {\it independence set for $\oA$ relative to $D$} if 
$$D\cap \bigcap_{s\in F}s^{-1}A_{\sigma(s)}\neq \emptyset$$ 
for every map $\sigma: F\rightarrow [k]$ \cite[Definition 1.1]{KL09}.

Let $\mu\in \cM_\Gamma(X)$. For $\delta>0$, we denote by $\sB(\mu, \delta)$ the set of Borel subsets $D$ of $X$ satisfying $\mu(D)\ge 1-\delta$. For each $F\in \cF(\Gamma)$, we set
$$ \varphi_{\oA, \mu, \delta}(F)=\min_{D\in \sB(\mu, \delta)}\max\{|J|: J\subseteq F \mbox{ is an independence set for } \oA \mbox{ relative to } D\},$$
and define $\upind_\mu(\oA, \delta)$ as the limit supremum of $\frac{1}{|F|}\varphi_{\oA, \mu, \delta}(F)$ as $F\in \cF(\Gamma)$ becomes more and more left invariant.
Then we define the {\it upper $\mu$-independence density of $\oA$} \cite[Definition 2.1]{KL09} as
$$ \upind_\mu(\oA)=\sup_{\delta>0}\upind_\mu(\oA, \delta).$$
We call $\ox=(x_1, \dots, x_k)\in X^k$ a {\it $\mu$-IE-tuple} if for every product neighborhood $U_1\times \cdots \times U_k$ of $\ox$ in $X^k$, the tuple $(U_1, \dots, U_k)$ has positive upper $\mu$-independence density \cite[Definition 2.5]{KL09}. We denote by $\IE^\mu_k(X)$ the set of all $\mu$-IE-tuples of length $k$.

We refer the reader to \cite{MO, OW, KL16} for general information about measure entropy. 
For  each finite Borel partition $\cP$ of $X$, the entropy of $\Gamma\curvearrowright (X, \mu)$ with respect to $\cP$ is defined as
$$ h_\mu(X, \cP)=\inf_{F\in \cF(\Gamma)}\frac{1}{|F|}H_\mu \left(\bigvee_{s\in F}s^{-1}\cP\right).$$
When $\cF\in \cF(\Gamma)$ becomes more and more left invariant, $\frac{1}{|F|}H_\mu \left(\bigvee_{s\in F}s^{-1}\cP\right)$ converges to $h_\mu(X, \cP)$ \cite[page 198]{KL16}. 
The {\it measure entropy} of $\Gamma\curvearrowright (X, \mu)$ is defined as $\sup_\cP  h_\mu(X, \cP)$ for $\cP$ ranging over finite Borel partitions of $X$, and is denoted by $h_\mu(X)$. 

The following summarizes the basic properties of measure IE-tuples we need \cite[Proposition 2.16 and Theorem 2.30]{KL09}.

\begin{theorem} \label{T-measure IE basic}
Let $\mu\in \cM_\Gamma(X)$ and $k\in \Nb$. The following hold:
\begin{enumerate}
\item \label{i-measure IE invariant} $\IE^\mu_k(X)$ is a closed $\Gamma$-invariant subset of $X^k$.
\item \label{i-measure IE tuple} Let $(A_1, \dots, A_k)$ be a tuple of closed subsets of $X$ with positive upper $\mu$-independence density. Then there is a $\mu$-IE-tuple $\ox\in A_1\times \cdots \times A_k$.
\item \label{i-measure IE pair} $\IE^\mu_2(X)\setminus \Delta_2(X)$ is nonempty if and only if $h_\mu(X)>0$.
\item \label{i-measure IE factor} Let $\Gamma\curvearrowright Y$ be another continuous action of $\Gamma$ on a compact metrizable space $Y$, and  let $\pi:X\rightarrow Y$ be a factor map. Then $(\pi\times \cdots \times \pi)(\IE^\mu_k(X))=\IE^{\pi_*\mu}_k(Y)$, where $\pi_*\mu\in \cM_\Gamma(Y)$ is the push-forward of $\mu$ under $\pi$.
\item \label{i-measure IE product} If $\Gamma\curvearrowright Y$ is another continuous action of $\Gamma$ on a compact metrizable space $Y$ and $\nu\in \cM_\Gamma(Y)$, then
$$ \IE^{\mu\times \nu}_k(X\times Y)=\IE^\mu_k(X)\times \IE^\nu_k(Y)$$
under the natural identification of $(X\times Y)^k$ and $X^k\times Y^k$.
%\item $\IE_1(X)=\bigcup_{\mu\in \cM_\Gamma(X)}\supp(\mu)$.
\end{enumerate}
\end{theorem}

%%%%%%%%%%%%%%%%%%%%%%%%%%%%%%%%%%%%%%%%%%%%%%%%%%%%%%%%%%%%%%%%%%%%%%%%%%%%%%%%%%%%%%%%%%%%
\subsection{IN-tuples} \label{SS-IN}

Let $k\in \Nb$.
We call $\ox\in X^k$ an {\it IN-tuple} if for every product neighborhood $U_1\times \cdots \times U_k$ of $\ox$ in $X^k$, the tuple $\oU=(U_1, \dots, U_k)$ has arbitrarily large finite independence sets. We denote by $\IN_k(X)$ the set of all IN-tuples of length $k$.

For a sequence $\mathfrak{s}=\{s_n\}_{n\in \Nb}$ in $\Gamma$, the {\it sequence topological entropy} of $\Gamma\curvearrowright X$ with respect to $\mathfrak{s}$ is defined as 
$$h_{\rm top}(X; \mathfrak{s})=\sup_\cU \limsup_{n\to \infty}\frac{1}{n}\log N\left(\bigvee_{i=1}^ns_i^{-1}\cU\right)$$
for $\cU$ ranging over finite open covers of $X$ \cite{Goodman}. We say the action $\Gamma\curvearrowright X$ is {\it null} if $ h_{\rm top}(X; \mathfrak{s})=0$ for every sequence $\mathfrak{s}$ in $\Gamma$. 

The following summarizes the basic properties of IN-tuples we need \cite[Proposition 5.4]{KL07}.

\begin{theorem} \label{T-IN basic}
Let $k\in \Nb$. The following hold:
\begin{enumerate}
\item \label{i-IN invariant} $\IN_k(X)$ is a closed $\Gamma$-invariant subset of $X^k$.
\item \label{i-IN tuple} Let $(A_1, \dots, A_k)$ be a tuple of closed subsets of $X$ with arbitrarily large finite independence sets. Then there is an IN-tuple $\ox\in A_1\times \cdots \times A_k$.
\item \label{i-IN pair} $\IN_2(X)\setminus \Delta_2(X)$ is nonempty if and only if $\Gamma\curvearrowright X$ is nonnull.
\item \label{i-IN factor} Let $\Gamma\curvearrowright Y$ be another continuous action of $\Gamma$ on a compact metrizable space $Y$, and  let $\pi:X\rightarrow Y$ be a factor map. Then $(\pi\times \cdots \times \pi)(\IN_k(X))=\IN_k(Y)$.
\item \label{i-IN subset} Let $Z$ be a closed $\Gamma$-invariant subset of $X$. Then $\IN_k(Z)\subseteq \IN_k(X)$.
\end{enumerate}
\end{theorem}

%%%%%%%%%%%%%%%%%%%%%%%%%%%%%%%%%%%%%%%%%%%%%%%%%%%%%%%%%%%%%%%%%%%%%%%%%%%%%%%%%%%%%%%%%%%%%%%
\subsection{Measure IN-tuples} \label{SS-measure IN}

Let $k\in \Nb$.
Let $\oA=(A_1, \dots, A_k)$ be a tuple of subsets of $X$, and let $D: \Gamma\rightarrow 2^X$. We say $F\in \cF(\Gamma)$ is an {\it independence set for $\oA$ relative to $D$} if 
$$\bigcap_{s\in F}(D_s\cap s^{-1}A_{\sigma(s)})\neq \emptyset$$ 
for every map $\sigma: F\rightarrow [k]$ \cite[Definition 1.1]{KL09}. 

Let $\mu\in \cM_\Gamma(X)$.
% be a $\Gamma$-invariant Borel probability measure on $X$. 
We denote by $\sB(X)$ the set of Borel subsets of $X$. For $\delta>0$, we denote by $\sB'(\mu, \delta)$ the set of $D: \Gamma\rightarrow \sB(X)$ satisfying 
$$\inf_{s\in \Gamma}\mu(D_s)\ge 1-\delta.$$
We say $\oA$ has {\it positive sequential $\mu$-independence density}  if there are $\delta, c>0$ such that for every $M>0$ there is an $F\in \cF(\Gamma)$ with $|F|\ge M$ so that for every $D\in \sB'(\mu, \delta)$ there is an independence set $J\subseteq F$ for $\oA$ relative to $D$ with $|J|\ge c|F|$.
We call $\ox=(x_1, \dots, x_k)\in X^k$ a {\it $\mu$-IN-tuple} if for every product neighborhood $U_1\times \cdots \times U_k$ of $\ox$ in $X^k$, the tuple $(U_1, \dots, U_k)$ has positive sequential $\mu$-independence density \cite[Definition 4.2]{KL09}. We denote by $\IN^\mu_k(X)$ the set of all $\mu$-IN-tuples of length $k$.

For a sequence $\mathfrak{s}=\{s_n\}_{n\in \Nb}$ in $\Gamma$, the {\it sequence measure entropy} of $\Gamma\curvearrowright (X, \mu)$ with respect to $\mathfrak{s}$ is defined as 
$$h_\mu(X; \mathfrak{s})=\sup_\cP \limsup_{n\to \infty}\frac{1}{n}\log H_\mu\left(\bigvee_{i=1}^ns_i^{-1}\cP\right)$$
for $\cP$ ranging over finite Borel partitions of $X$ \cite{Kushnirenko}. We say the action $\Gamma\curvearrowright (X, \mu)$ is {\it null} if $h_\mu(X; \mathfrak{s})=0$ for every sequence $\mathfrak{s}$ in $\Gamma$.

The following summarizes the basic properties of measure IN-tuples we need \cite[Proposition 4.7 and Theorem 4.12]{KL09}.

\begin{theorem} \label{T-measure IN basic}
Let $\mu\in \cM_\Gamma(X)$ and $k\in \Nb$. The following hold:
\begin{enumerate}
\item \label{i-measure IN invariant} $\IN^\mu_k(X)$ is a closed $\Gamma$-invariant subset of $X^k$.
\item \label{i-measure IN tuple} Let $(A_1, \dots, A_k)$ be a tuple of closed subsets of $X$ with positive sequential $\mu$-independence density. Then there is a $\mu$-IN-tuple $\ox\in A_1\times \cdots \times A_k$.
\item \label{i-measure IN pair} $\IN^\mu_2(X)\setminus \Delta_2(X)$ is nonempty if and only if $\Gamma\curvearrowright (X,\mu)$ is nonnull.
\item \label{i-measure IN factor} Let $\Gamma\curvearrowright Y$ be another continuous action of $\Gamma$ on a compact metrizable space $Y$, and  let $\pi:X\rightarrow Y$ be a factor map. Then $(\pi\times \cdots \times \pi)(\IN^\mu_k(X))=\IN^{\pi_*\mu}_k(Y)$, where $\pi_*\mu\in \cM_\Gamma(Y)$ is the push-forward of $\mu$ under $\pi$. 
\item \label{i-measure IN product} If $\Gamma\curvearrowright Y$ is another continuous action of $\Gamma$ on a compact metrizable space $Y$ and $\nu\in \cM_\Gamma(Y)$, then
$$ \IN^{\mu\times \nu}_k(X\times Y)=\IN^\mu_k(X)\times \IN^\nu_k(Y)$$
under the natural identification of $(X\times Y)^k$ and $X^k\times Y^k$.
\end{enumerate}
\end{theorem}

%%%%%%%%%%%%%%%%%%%%%%%%%%%%%%%%%%%%%%%%%%%%%%%%%%%%%%%%%%%%%%%%%%%%%%%%%%%%%%%%%%%%%%%%%%%%
\subsection{IT-tuples} \label{SS-IT}

Let $k\in \Nb$.
We call $\ox\in X^k$ an {\it IT-tuple} if for every product neighborhood $U_1\times \cdots \times U_k$ of $\ox$ in $X^k$, the tuple $\oU=(U_1, \dots, U_k)$ has an infinite independence set. We denote by $\IT_k(X)$ the set of all IT-tuples of length $k$.

The action $\Gamma\curvearrowright X$ is said to be {\it untame} if there are an $f\in C(X)$, an infinite subset $F$ of $\Gamma$ and a constant $C>0$ such that 
$$\left\|\sum_{s\in F}\lambda_s sf\right\|\ge C\sum_{s\in F}|\lambda_s|$$
for all $\lambda \in \ell_1(F)$. 

The following summarizes the basic properties of IN-tuples we need \cite[Proposition 8.14]{KL16}.

\begin{theorem} \label{T-IT basic}
Let $k\in \Nb$. The following hold:
\begin{enumerate}
\item \label{i-IT invariant} $\IT_k(X)$ is a closed $\Gamma$-invariant subset of $X^k$.
\item \label{i-IT tuple} Let $(A_1, \dots, A_k)$ be a tuple of closed subsets of $X$ with an infinite independence set. Then there is an IT-tuple $\ox\in A_1\times \cdots \times A_k$.
\item \label{i-IT pair} $\IT_2(X)\setminus \Delta_2(X)$ is nonempty if and only if $\Gamma\curvearrowright X$ is untame.
\item \label{i-IT factor} Let $\Gamma\curvearrowright Y$ be another continuous action of $\Gamma$ on a compact metrizable space $Y$, and  let $\pi:X\rightarrow Y$ be a factor map. Then $(\pi\times \cdots \times \pi)(\IT_k(X))=\IT_k(Y)$.
\item \label{i-IT subset} Let $Z$ be a closed $\Gamma$-invariant subset of $X$. Then $\IT_k(Z)\subseteq \IT_k(X)$.
\end{enumerate}
\end{theorem}

%%%%%%%%%%%%%%%%%%%%%%%%%%%%%%%%%%%%%%%%%%%%%%%%%%%%%%%%%%%%%%%%%%%%%%%%%%%%%%%%%%%%%%%%%%%%%%%%%%%%%%%%%%%%%%%%%%%%%
\section{$\IE$-tuples of $\cM(X)$} \label{S-IE for prob}

Throughout this section, we consider a continuous action of a countably infinite amenable group $\Gamma$ on a compact metrizable space $X$.

%%%%%%%%%%%%%%%%%%%%%%%%%%%%%%%%%%%%%%%%%%%%%%%%%%%%%%%%%%%%%%%%%%%%%%%%%%%%%%%%
\subsection{Convexity of $\IE_k(\cM(X))$} \label{SS-convex}

In this subsection we establish the convexity of $\IE_k(\cM(X))$ and show that parts (2) and (3) of Theorem~\ref{T-IE main} follow from part (1).

\begin{notation} \label{N-finite support}
For $N\in \Nb$ we denote by $\cM_N(X)$ the set of $\mu\in \cM(X)$ with $|\supp(\mu)|\le N$.
\end{notation}

Clearly $\cM_N(X)$ is a  $\Gamma$-invariant closed subset of $\cM(X)$.

Let $k\in \Nb$. Then $\cM(X)^k$ is naturally a compact convex subset of $(C(X)^{\oplus k})^*=(C(X)^*)^{\oplus k}$. We show first that $\IE_k(\cM(X))$ is convex. As a contrast, later we shall see in Section~\ref{S-example} that $\IN_k(\cM(X))$ and $\IT_k(\cM(X))$ may fail to be convex. 
Part (1) of Lemma~\ref{L-IE convex} is not needed in this subsection, but will be used later in the proof of Lemma~\ref{L-support to support}. 

\begin{lemma} \label{L-IE convex}
Let $k, N_1, N_2\in \Nb$.
The following hold:
\begin{enumerate}
\item For any $\omu\in \IE_k(\cM_{N_1}(X))$, $\onu\in \IE_k(\cM_{N_2}(X))$ and $0\le \lambda \le 1$, one has
$$\lambda \omu+(1-\lambda)\onu\in \IE_k(\cM_{N_1+N_2}(X)).$$
\item $\IE_k(\cM(X))$ is a compact convex subset of $\cM(X)^k$.
\end{enumerate}
\end{lemma}
\begin{proof}
(1). Since $\omu\in \cM_{N_1}(X)^k$ and $\onu\in \cM_{N_2}(X)^k$,
we have $\lambda \omu+(1-\lambda)\onu\in \cM_{N_1+N_2}(X)^k$.
Say, $\omu=(\mu_1, \dots, \mu_k)$ and $\onu=(\nu_1, \dots, \nu_k)$.
Let $U_1\times \cdots \times U_k$ be a product neighborhood of
$$\lambda \omu+(1-\lambda)\onu=(\lambda \mu_1+(1-\lambda)\nu_1, \dots, \lambda \mu_k+(1-\lambda)\nu_k)$$
in $\cM_{N_1+N_2}(X)^k$.
Then we can find a product neighborhood $V_1\times \cdots \times V_k$ of $\omu$ in $\cM_{N_1}(X)^k$ and a product neighborhood $W_1\times \cdots \times W_k$ of $\onu$ in $\cM_{N_2}(X)^k$ such that 
$$\lambda \mu+(1-\lambda)\nu\in U_j$$ 
for all $j\in [k]$, $\mu\in V_j$ and $\nu\in W_j$. 

From Theorem~\ref{T-IE basic}.\eqref{i-IE product}  we know that $((\mu_1, \nu_1), \dots, (\mu_k, \nu_k))$ lies in $\IE_k(\cM_{N_1}(X)\times \cM_{N_2}(X))$. Since $(V_1\times W_1)\times \cdots \times (V_k\times W_k)$ is a product neighborhood of $((\mu_1, \nu_1), \dots, (\mu_k, \nu_k))$  in $(\cM_{N_1}(X)\times \cM_{N_2}(X))^k$, there is some $c>0$ such that for any $F\in \cF(\Gamma)$ one can find an independence set $J\subseteq F$ for $(V_1\times W_1, \dots,  V_k\times W_k)$ such that $|J|\ge c|F|$. For each map $\sigma: J\rightarrow [k]$, we can find some $(\mu, \nu)\in \cM_{N_1}(X)\times \cM_{N_2}(X)$ so that $s(\mu, \nu)\in V_{\sigma(s)}\times W_{\sigma(s)}$ for each $s\in J$, whence $\lambda\mu+(1-\lambda)\nu\in \cM_{N_1+N_2}(X)$ and
$$s(\lambda \mu+(1-\lambda)\nu)=\lambda s\mu+(1-\lambda) s\nu\in U_{\sigma(s)}$$
for each $s\in J$.  Thus $J$ is also an independence set for $\oU=(U_1, \dots, U_k)$. This shows that $\oI(\oU)\ge c>0$,
and hence $\lambda \omu+(1-\lambda)\onu\in \IE_k(\cM_{N_1+N_2}(X))$.

(2). The proof for the convexity of $\IE_k(\cM(X))$ is similar to that of part (1), replacing $\cM_{N_1}(X), \cM_{N_2}(X)$ and $\cM_{N_1+N_2}(X)$ everywhere by $\cM(X)$. From Theorem~\ref{T-IE basic}.\eqref{i-IE invariant} we know that $\IE_k(\cM(X))$ is closed in $\cM(X)^k$. Since $\cM(X)$ is compact, it follows that $\IE_k(\cM(X))$ is compact.
\end{proof}

Next we show that parts (2) and (3) of Theorem~\ref{T-IE main} follow from part (1). For this we need the following theorem of Milman \cite[Theorem V.7.8]{Conway}.

\begin{lemma} \label{L-ext}
For any compact convex set $\cK$ in a Hausdorff locally convex topological vector space and any closed $Y\subseteq \cK$ such that $\cK$ is the closed convex hull of $Y$, one has $\ext\, \cK\subseteq Y$.
\end{lemma}

\begin{lemma} \label{L-barycenter identity k}
Let $k\in \Nb$ and $\mu\in \cM(X^k)$. If we treat $\mu$ as a Borel probability measure $\rmm_\mu$ on $\cM(X)^k$ via the embedding $X^k\hookrightarrow \cM(X)^k$, then the barycenter of $\rmm_\mu$ in $\cM(X)^k$ is $(\mu^{(1)}, \dots, \mu^{(k)})$. 
\end{lemma}
\begin{proof} We denote by $\Psi$ the barycenter map $\cM(\cM(X)^k)\rightarrow \cM(X)^k$. We consider the map $\Phi_1: \cM(X^k)\rightarrow \cM(X)^k$ sending $\mu$ to $(\mu^{(1)}, \dots, \mu^{(k)})$, and the map $\Phi_2: \cM(X^k)\rightarrow \cM(X)^k$ sending $\mu$ to $\Psi(\rmm_\mu)$. Both $\Phi_1$ and $\Phi_2$ are affine and continuous. 

Let $(x_1, \dots, x_k)\in X^k$. We have $\Phi_1(\delta_{(x_1, \dots, x_k)})=(\delta_{x_1}, \dots, \delta_{x_k})$. Since $\rmm_{\delta_{(x_1, 
\dots, x_k)}}=\delta_{(\delta_{x_1}, \dots, \delta_{x_k})}$, we also have $\Phi_2(\delta_{(x_1, \dots, x_k)})=(\delta_{x_1}, \dots, \delta_{x_k})$. 
Thus 
$$\Phi_1(\delta_{(x_1, \dots, x_k)})=\Phi_2(\delta_{(x_1, \dots, x_k)}).$$

As $\cM(X^k)$ is the closed convex hull of the set of Dirac measures $\delta_{(x_1, \dots, x_k)}$ for $(x_1, \dots, x_k)\in X^k$, we conclude that 
$\Phi_1=\Phi_2$.
\end{proof}

Since we treat $X$ as a closed $\Gamma$-invariant subset of $\cM(X)$ via identifying $x\in X$ with the Dirac measure $\delta_x$ at $x$, by Theorem~\ref{T-IE basic}.\eqref{i-IE subset} we have $\IE_k(X)\subseteq \IE_k(\cM(X))$ for every $k\in \Nb$. 

\begin{lemma} \label{L-conditions for extreme of prob}
Let $k\in \Nb$. Assume that part (1) of Theorem~\ref{T-IE main} holds. Then parts (2) and (3) also hold.
\end{lemma}
\begin{proof}
Since $\IE_k(\cM(X))$ is compact convex by Lemma~\ref{L-IE convex}.(2), we can talk about $\ext\, \IE_k(\cM(X))$. Note that $\ext\, \cM(X)=X$, whence $\ext\, (\cM(X)^k)=(\ext\, \cM(X))^k=X^k$. Thus 
$$\IE_k(X)\subseteq \IE_k(\cM(X))\cap X^k=\IE_k(\cM(X))\cap \ext\, (\cM(X)^k)\subseteq \ext\, \IE_k(\cM(X)).$$
From part (1) of Theorem~\ref{T-IE main} and Lemma~\ref{L-ext} we have $\ext\, \IE_k(\cM(X))\subseteq \IE_k(X)$. Thus part (2) of Theorem~\ref{T-IE main} holds.

We denote by $\Psi$ the barycenter map $\cM(\cM(X)^k)\rightarrow \cM(X)^k$ sending $\rmm$ to its barycenter.

For each $\mu\in \cM(\IE_k(X))$,
we may treat $\mu$ as a Borel probability measure $\rmm$ on $\cM(X)^k$ via the embedding $\IE_k(X)\hookrightarrow X^k\hookrightarrow \cM(X)^k$.
%and then from Lemma~\ref{L-barycenter identity k} we have $\Psi(\rmm)=(\mu^{(1)}, \dots, \mu^{(k)})$. Thus the set of $(\mu^{(1)}, \dots, \mu^{(k)})$ for $\mu\in %\cM(\IE_k(X))$ is the same as $\Psi(\cM(\IE_k(X)))$.
From Lemma~\ref{L-barycenter identity k} we have
$$ \Psi(\cM(\IE_k(X)))=\{(\mu^{(1)}, \dots, \mu^{(k)}): \mu\in \cM(\IE_k(X))\}.$$

Since $\cM(\IE_k(X))$ is the closed convex hull of $\IE_k(X)$ in $\cM(\cM(X)^k)$ and $\Psi$ is affine and continuous, $\Psi(\cM(\IE_k(X)))$ is the closed convex hull  of $\Psi(\IE_k(X))=\IE_k(X)$ in $\cM(X)^k$. Thus part (3) of Theorem~\ref{T-IE main} follows from part (1).
\end{proof}

The case $k=1$ of Theorem~\ref{T-IE main}.(1) is easy to prove. It is mainly based on the following simple fact.

\begin{lemma} \label{L-supp for barycenter}
Let $\rmm\in \cM(\cM(X))$ and let $\mu\in \cM(X)$ be the barycenter of $\rmm$. Then
\begin{align} \label{E-supp for barycenter}
 \bigcup_{\nu\in \supp(\rmm)}\supp(\nu)\subseteq \supp(\mu).
 \end{align}
\end{lemma}
\begin{proof} We argue by contradiction. Assume that \eqref{E-supp for barycenter} fails. Then we can find some $\nu\in \supp(\rmm)$ and some $x\in \supp(\nu)$ such that $x\not\in \supp(\mu)$. By the Tietze extension theorem we can find an $f\in C(X)$ such that $f\ge 0$ on $X$, $f=0$ on $\supp(\mu)$, and $f(x)>0$. Then $\nu(f)>0$ and $\omega(f)\ge 0$ for every $\omega\in \cM(X)$, and hence
$$ \mu(f)=\int_{\cM(X)}\omega(f)\, d\rmm(\omega)>0,$$
which contradicts that $f=0$ on $\supp(\mu)$. Thus  \eqref{E-supp for barycenter} holds.
\end{proof}

We are ready to prove the case $k=1$ of Theorem~\ref{T-IE main}. 

\begin{proof}[Proof of the case $k=1$ of Theorem~\ref{T-IE main}] For any $k\in \Nb$, since $\IE_k(\cM(X))$ is a compact convex subset of $\cM(X)^k$ by Lemma~\ref{L-IE convex}.(2) and $\IE_k(X)\subseteq \IE_k(\cM(X))$ we immediately get $\overline{\rm co}(\IE_k(X))\subseteq \IE_k(\cM(X))$.

Let $\rmm\in \cM_\Gamma(\cM(X))$  and
let $\mu\in \cM_\Gamma(X)$ be the barycenter of $\rmm$. By Theorem~\ref{T-IE basic}.\eqref{i-IE 1} we have $\supp(\mu)\subseteq \IE_1(X)$. From Lemma~\ref{L-supp for barycenter} we have $\supp(\nu)\subseteq \supp(\mu)$ for every $\nu\in \supp(\rmm)$. Thus
$$ \nu\in \overline{\rm co}(\supp(\nu))\subseteq \overline{\rm co}(\supp(\mu))\subseteq \overline{\rm co}(\IE_1(X))$$
for every $\nu\in \supp(\rmm)$, whence
$$ \supp(\rmm)\subseteq \overline{\rm co}(\IE_1(X)).$$
Applying Theorem~\ref{T-IE basic}.\eqref{i-IE 1} again we get $\IE_1(\cM(X))\subseteq \overline{\rm co}(\IE_1(X))$. This proves part (1) of Theorem~\ref{T-IE main} in the case $k=1$. From Lemma~\ref{L-conditions for extreme of prob} we conclude that the case $k=1$ of Theorem~\ref{T-IE main} holds. 
\end{proof}

To prove the  $k\ge 2$ case of Theorem~\ref{T-IE main}, we shall use the following lemma and apply some combinatorial lemmas to show that the phenomenon described there cannot happen.

\begin{lemma} \label{L-Hahn Banach IE}
If Theorem~\ref{T-IE main} fails for some $k\in \Nb$, then there are some $(\mu_1, \dots, \mu_k)\in \IE_k(\cM(X))$ and some $(f_1, \dots, f_k)\in C(X)^{\oplus k}$ such that $\sum_{j\in [k]}f_j(x_j)\le 0$
for all $(x_1, \dots, x_k)\in \IE_k(X)$ and
$\sum_{j\in [k]}\mu_j(f_j)>0$.
\end{lemma}
\begin{proof}
In the proof of the case $k=1$ of Theorem~\ref{T-IE main}, we have shown $\overline{\rm co}(\IE_k(X))\subseteq \IE_k(\cM(X))$ for all $k\in \Nb$.

Assume that the case $k$ of Theorem~\ref{T-IE main} fails. Then by Lemma~\ref{L-conditions for extreme of prob} part (1) of  Theorem~\ref{T-IE main} fails for this $k$. Thus $\overline{\rm co}(\IE_k(X))\subsetneq \IE_k(\cM(X))$. Then we can find a $\omu=(\mu_1, \dots, \mu_k)\in \IE_k(\cM(X))\setminus \overline{\rm co}(\IE_k(X))$. Note that the dual of $(C(X)^{\oplus k})^*$ equipped with the weak$^*$-topology
is $C(X)^{\oplus k}$ \cite[Theorem V.1.3]{Conway}.
Since $\overline{\rm co}(\IE_k(X))$ and $\{\omu\}$ are disjoint compact convex sets in $(C(X)^{\oplus k})^*$, by the geometric consequence of the Hahn-Banach theorem \cite[Proposition IV.3.6.(b) and Theorem IV.3.9]{Conway}, there is some $\of=(f_1, \dots, f_k)\in C(X)^{\oplus k}$  such  that $\sum_{j\in [k]}\nu_j(f_j)\le 0$ for all $(\nu_1, \dots, \nu_k)\in \overline{\rm co}(\IE_k(X))$ and $\sum_{j\in [k]}\mu_j(f_j)>0$. Equivalently, $\sum_{j\in [k]}f_j(x_j)\le 0$ for all $(x_1, \dots, x_k)\in \IE_k(X)$ and $\sum_{j\in [k]}\mu_j(f_j)>0$.
\end{proof}

%%%%%%%%%%%%%%%%%%%%%%%%%%%%%%%%%%%%%%%%%%%%%%%%%%%%%%%%%%%%%%%%%%%%%%%%%%%%%%%%%%%%%%%%%
\subsection{Case $k=2$ of Theorem~\ref{T-IE main}} \label{SS-k=2}

In this subsection we establish the combinatorial Lemma~\ref{L-half to indep} and use it to prove the case $k=2$ of Theorem~\ref{T-IE main}.

We show first that, when $k=2$, the condition $\sum_{j\in [k]}\mu_j(f_j)>0$ in Lemma~\ref{L-Hahn Banach IE} leads to a combinatorial condition in Lemma~\ref{L-large density}. 

\begin{lemma} \label{L-sum}
Let $\mu_1, \mu_2\in \cM(X)$  and $f_1, f_2\in C(X)$ such that $\sum_{j\in [2]} \mu_j(f_j)>0$. Then there are open sets $U_1, U_2\subseteq X$ such that 
$$\sum_{j\in [2]} \min_{x\in \overline{U_j}}f_j(x)>0 \quad \text{ and } \quad \sum_{j\in [2]}\mu_j(U_j)>1.$$
\end{lemma}
\begin{proof} Let $C=\max(\|f_1||, \|f_2\|)>0$. We denote by $\nu$ the Lebesgue measure on $\Rb$. For $j\in [2]$, setting
$$A_j=\{(x, t)\in X\times [-C, C]: t<f_j(x)\}$$
and denoting by ${\bf 1}_{A_j}$ the characteristic function of $A_j$,
by Fubini's theorem  we have
\begin{align*}
\mu_j(f_j)
&=-C+\int_X (f_j(x)+C)\, d\mu_j(x)=-C+\int_X \int_{[-C, C]}{\bf 1}_{A_j}(x, t)\, d\nu(t)\, d\mu_j(x)\\
&=-C+\int_{[-C, C]}\int_X {\bf 1}_{A_j}(x, t)\, d\mu_j(x)\, d\nu(t)=-C+\int_{[-C, C]}\mu_j(f_j^{-1}((t, C]))\, d\nu(t).
\end{align*}
Thus
\begin{align*}
0&<\sum_{j\in [2]}\mu_j(f_j)\\
&=-2C+\int_{[-C, C]}\mu_1(f_1^{-1}((t, C]))\, d\nu(t)+\int_{[-C, C]}\mu_2(f_2^{-1}((t, C]))\, d\nu(t)\\
&=-2C+\int_{[-C, C]}\mu_1(f_1^{-1}((t, C]))\, d\nu(t)+\int_{[-C, C]}\mu_2(f_2^{-1}((-t, C]))\, d\nu(t)\\
&=-2C+\int_{[-C, C]}\left(\mu_1(f_1^{-1}((t, C]))+\mu_2(f_2^{-1}((-t, C]))\right)\, d\nu(t).
\end{align*}
Then there is some $t\in [-C, C]$ such that
$$ \mu_1(f_1^{-1}((t, C]))+\mu_2(f_2^{-1}((-t, C]))>1.$$
Since $\mu_1(f_1^{-1}((t+\varepsilon, C]))+\mu_2(f_2^{-1}((-t+\varepsilon, C]))$ converges to $\mu_1(f_1^{-1}((t, C]))+\mu_2(f_2^{-1}((-t, C]))$ as $\varepsilon\to 0^+$, we can find an $\varepsilon>0$ such that
$$\mu_1(f_1^{-1}((t+\varepsilon, C]))+\mu_2(f_2^{-1}((-t+\varepsilon, C]))>1.$$
We set $U_1=f_1^{-1}((t+\varepsilon, C])$ and $U_2=f_2^{-1}((-t+\varepsilon, C])$. Then $U_1$ and $U_2$ are open, and $\mu_1(U_1)+\mu_2(U_2)>1$. Clearly
$\sum_{j\in [2]} \min_{x\in \overline{U_j}}f_j(x)\ge 2\varepsilon>0$.
\end{proof}

One might hope to extend Lemma~\ref{L-sum} to the case $k\ge 3$ to yield 
$$\sum_{j\in [k]} \min_{x\in \overline{U_j}}f_j(x)>0 \quad \text{ and } \quad \sum_{j\in [k]}\mu_j(U_j)>k-1.$$ 
The following example shows that such an extension does not hold in general.

\begin{example} \label{E-sum3}
For each $j\in [3]$ let $\mu_j=\frac{1}{2}(\delta_{x_j}+\delta_{y_j})\in \cM(X)$ with distinct $x_j, y_j\in X$. By the Tietze extension theorem we can find an $f_j\in C(X)$ for each $j\in [3]$ such that the following hold:
\begin{enumerate}
\item $f_j(x_j)=-f_j(y_j)\in (1, 2)$ for $j=1, 2$, and
\item $f_3(x_3)\in (1, 2), f_3(y_3)\in (-2, -1)$ and $f_3(x_3)+f_3(y_3)>0$.
\end{enumerate}
Then $\sum_{j\in [3]}\mu_j(f_j)>0$. 

We claim that there are no closed sets $A_j\subseteq X$ for $j\in [3]$ satisfying 
$$\sum_{j\in [3]}\min_{x\in A_j}f_j(x)>0 \quad \text{ and } \quad \sum_{j\in [3]}\mu_j(A_j)>2.$$ 
Indeed, assume that such $A_j$ for $j\in [3]$ exist.
Note that $\mu_j(A_j)\in \{0, 1/2, 1\}$ for all $j\in [3]$. Thus there is some $F\subseteq [3]$ with $|F|=2$ such that $\mu_j(A_j)=1$ for all $j\in F$. Then $y_j\in A_j$ for all $j\in F$, whence $f_j(y_j)\ge \min_{x\in A_j}f_j(x)$ for all $j\in F$. Say, $\{i\}=[3]\setminus F$. Then $\mu_i(A_i)>0$. Thus $A_i\cap \{x_i, y_i\}\neq \emptyset$. Since $f_i(x_i)>0>f_i(y_i)$, we have 
$$ f_i(x_i)=\max\{f_i(x_i), f_i(y_i)\}\ge \min_{x\in A_i}f_i(x).$$
It follows that
$$f_i(x_i)+\sum_{j\in F}f_j(y_j)\ge \sum_{j\in [3]}\min_{x\in A_j}f_j(x)>0.$$
But $f_i(x_i)+\sum_{j\in F}f_j(y_j)<2+(-1)+(-1)=0$, a contradiction. This proves our claim. 
\end{example}

\begin{lemma} \label{L-large density}
Let $\mu_1, \mu_2\in \cM(X)$ and $f_1, f_2\in C(X)$ such that $\sum_{j\in [2]}\mu_j(f_j)>0$. Then there are $0<\delta<1$, nonempty closed sets $A_1, A_2\subseteq X$ and a product neighborhood $V_1\times V_2$ of $(\mu_1, \mu_2)$ in $\cM(X)^2$ such that the following hold:
\begin{enumerate}
\item $\sum_{j\in [2]}\min_{x\in A_j}f_j(x)>0$, and
\item for any independence set $F\in \cF(\Gamma)$ of $(V_1, V_2)$ and any surjective map $\psi: F\rightarrow [2]$, there is an $F_\psi\subseteq F$ so that 
    $$\sum_{j\in [2]} \frac{|F_\psi\cap \psi^{-1}(j)|}{|\psi^{-1}(j)|}>1+\delta \quad \text{ and } \quad \bigcap_{s\in F_\psi}s^{-1}A_{\psi(s)}\neq \emptyset.$$
\end{enumerate}
\end{lemma}
\begin{proof} Let $U_1, U_2$ be given by Lemma~\ref{L-sum}. We set $A_j=\overline{U_j}$ for $j\in [2]$. Then
$$\sum_{j\in [2]}\min_{x\in A_j}f_j(x)=\sum_{j\in [2]} \min_{x\in \overline{U_j}}f_j(x)>0. $$

By our choice of $U_1$ and $U_2$ we have $\sum_{j\in [2]}\mu_j(U_j)>1$.
Thus we can take a  $0<\delta<\big(\sum_{j\in [2]}\mu_j(U_j)-1\big)/3$. Then
$$V_j:=\{\nu\in \cM(X): \nu(U_j)>\mu_j(U_j)-\delta\}$$
is an open neighborhood of $\mu_j$ in $\cM(X)$ for $j\in [2]$.

Let $F\in \cF(\Gamma)$ be an independence set for $(V_1, V_2)$ and let $\psi: F\rightarrow [2]$ be surjective. Then we can find some $\nu\in \bigcap_{s\in F}s^{-1}V_{\psi(s)}$. For each $s\in F$, we have $s\nu\in V_{\psi(s)}$ and hence
$$\nu(s^{-1}A_{\psi(s)})=(s\nu)(A_{\psi(s)})\ge (s\nu)(U_{\psi(s)})>\mu_{\psi(s)}(U_{\psi(s)})-\delta.$$
Therefore
\begin{align*}
\int_X \sum_{j\in [2]}\frac{1}{|\psi^{-1}(j)|}\sum_{s\in \psi^{-1}(j)}{\bf 1}_{s^{-1}A_j}(x)\, d\nu(x)&=\sum_{j\in [2]}\frac{1}{|\psi^{-1}(j)|}\sum_{s\in \psi^{-1}(j)}\nu(s^{-1}A_j)\\
&\ge \sum_{j\in [2]}\frac{1}{|\psi^{-1}(j)|}\sum_{s\in \psi^{-1}(j)}(\mu_j(U_j)-\delta)\\
&=\sum_{j\in [2]}(\mu_j(U_j)-\delta)>1+\delta.
\end{align*}
Thus there is some $x\in X$ such that
$$\sum_{j\in [2]}\frac{1}{|\psi^{-1}(j)|}\sum_{s\in \psi^{-1}(j)}{\bf 1}_{s^{-1}A_j}(x)>1+\delta.$$
We denote by $F_\psi$ the set of $s\in F$ satisfying $sx\in A_{\psi(s)}$. Then
$$  \sum_{j\in [2]} \frac{|F_\psi\cap \psi^{-1}(j)|}{|\psi^{-1}(j)|}=\sum_{j\in [2]}\frac{1}{|\psi^{-1}(j)|}\sum_{s\in \psi^{-1}(j)}{\bf 1}_{s^{-1}A_j}(x)>1+\delta,$$
and $x\in \bigcap_{s\in F_\psi}s^{-1}A_{\psi(s)}$.
\end{proof}

We need the following weak version of Stirling's approximation \cite[Lemma 10.1]{KL16}.

\begin{lemma} \label{L-Stirling}
For every $m\in \Nb$ we have
\begin{align} \label{E-Stirling}
e(\frac{m}{e})^m\le m!\le em(\frac{m}{e})^m.
\end{align}
\end{lemma}

We also need the following consequence of Karpovsky and Milman's generalization of the Sauer-Perles-Shelah lemma \cite{KM, Sauer, Shelah} \cite[Lemma 12.14]{KL16}.

\begin{lemma} \label{L-KM}
Let $k\ge 2$ be an integer and $\lambda>1$. Then there is a $c>0$ depending only on $k$ and $\lambda$ such that for any nonempty finite set $Z$ and any $S\subseteq [k]^Z$ with $|S|\ge ((k-1)\lambda)^{|Z|}$ there exists a $J\subseteq Z$ such that $|J|\ge c|Z|$ and $S|_J=[k]^J$.
\end{lemma}

For a finite set $Z$ and $k\in \Nb$, we shall consider the set of maps $\psi: Z\rightarrow [k]$ which are as equidistributed as possible. When $|Z|$ is divisible by $k$, this means $|\psi^{-1}(j)|=|Z|/k$ for all $j\in [k]$. In general case, the precise meaning is given in the following. 

\begin{notation} \label{N-regular}
Let $k\in \Nb$ and let $Z$ be a finite set. We denote by $\sR(Z, k)$ the set of maps $\psi: Z\rightarrow [k]$ satisfying $\left||\psi^{-1}(j)|-|Z|/k\right|< 1$ for all $j\in [k]$.
\end{notation}

For $t\in \Rb$, we write $\lceil t\rceil$ for the smallest integer no less than  $t$. 
%The following is the key combinatorial lemma for the proof of the case $k=2$ of Theorem~\ref{T-IE main}. 
In the proof of the case $k=2$ of Theorem~\ref{T-IE main}, we shall obtain $F_\psi$ from Lemma~\ref{L-large density} and apply Lemma~\ref{L-half to indep} with $Z=F$ and $Z_\psi=F_\psi$. In order for the condition $\sum_{j\in [2]} \frac{|F_\psi\cap \psi^{-1}(j)|}{|\psi^{-1}(j)|}>1+\delta$ in Lemma~\ref{L-large density} to imply the condition $|Z_\psi|\ge \tau|Z|$ in Lemma~\ref{L-half to indep}, we need $|\psi^{-1}(1)|$  and $|\psi^{-1}(2)|$ to be almost equal, whence in Lemma~\ref{L-half to indep} we restrict our attention to $\psi\in \sR(Z, k)$.  

\begin{lemma} \label{L-half to indep}
Let $k\in \Nb$ and $\tau\in ((k-1)/k, 1]$. Then there exist $c, N>0$ depending only on $k$ and $\tau$ such that the following holds. For any finite set $Z$ with $|Z|\ge N$, if for each $\psi\in \sR(Z, k)$  we take a $Z_\psi\subseteq Z$  with $|Z_\psi|\ge \tau |Z|$, then there is some $J\subseteq Z$ with $|J|\ge c|Z|$ such that every map $\sigma: J\rightarrow [k]$ extends to a $\psi\in \sR(Z, k)$ so that $J\subseteq Z_\psi$.
\end{lemma}
\begin{proof} The case $k=1$ is obvious. Thus we may assume $k\ge 2$. 
Recall the function $g: [0, 1]\rightarrow \Rb$ defined by \eqref{E-Shannon}.
Note that
\begin{align*}
\lim_{\delta \to 0^+}&\frac{-g(\tau)-g(1-\delta)+g(\tau-\delta)+\delta \log k}{\delta}\\
& =-g'(\tau)+g'(1)+\log k=\log (k\tau)>\log (k-1).
\end{align*}
We take $\theta_1>\theta_2$ in $(\log (k-1), \log (k\tau))$. 
Then we can find a $\delta\in (0, \tau)$ small enough so that
\begin{align} \label{E-half to indep1}
 \frac{-g(\tau)-g(1-\delta)+g(\tau-\delta)+\delta \log k}{\delta}>\theta_1.
 \end{align}
Let $\lambda=e^{\theta_2}/(k-1)>1$. Then we have $c>0$ given by Lemma~\ref{L-KM} for $k$ and $\lambda$.
We take a large integer $N\ge k$ such that $(\tau-\delta)N>1$ and
\begin{align} \label{E-half to indep2}
\frac{ (1-\delta)^{-1}k^k}{e^{k+1}t^{2k} (\tau t-\delta t+1)^2} e^{\delta t \theta_1}\ge e^{\theta_2+\delta t\theta_2}
\end{align}
for all $t\ge N$. We shall show that $c\delta$ and $N$ have the desired property.

Let $Z$ be a finite set with $|Z|\ge N$, and let $Z_\psi\subseteq Z$ with $|Z_\psi|\ge \tau|Z|$ for each $\psi\in \sR(Z, k)$. We can find a $Z'\subseteq Z$ such that $|Z'|$ is divisible by $k$ and $|Z\setminus Z'|<k$.
We denote by $\Xi$ the set of $(\psi, W)$ such that $\psi\in \sR(Z, k)$ and $W\subseteq Z_\psi$ with $|W|=\lceil\delta |Z|\rceil$.

For each $\psi\in \sR(Z, k)$, the number of $W\subseteq Z_\psi$ with $|W|=\lceil\delta |Z|\rceil$ is
$$ \binom{|Z_\psi|}{\lceil\delta |Z|\rceil}\ge \binom{\lceil \tau |Z|\rceil}{\lceil\delta |Z|\rceil}.$$
Thus
\begin{align*}
 |\Xi|&\ge |\sR(Z, k)|\cdot \binom{\lceil\tau |Z|\rceil}{\lceil\delta |Z|\rceil}\ge |\sR(Z', k)|\cdot \binom{\lceil\tau |Z|\rceil}{\lceil\delta |Z|\rceil}\\
 &=\frac{|Z'|!}{\left(\left(\frac{1}{k}|Z'|\right)!\right)^k}\cdot \binom{\lceil \tau |Z|\rceil}{\lceil\delta |Z|\rceil} \ge \frac{|Z|!}{|Z|^{k-1}\left(\left(\frac{1}{k}|Z'|\right)!\right)^k}\cdot \binom{\lceil \tau |Z|\rceil}{\lceil\delta |Z|\rceil}.
 \end{align*}
The number of $W\subseteq Z$ with $|W|=\lceil\delta |Z|\rceil$ is $\binom{|Z|}{\lceil\delta |Z|\rceil}$. Thus we can find some $W\subseteq Z$ with $|W|=\lceil\delta |Z|\rceil$ so that
$$|\Xi_W|\ge |\Xi|/\binom{|Z|}{\lceil\delta |Z|\rceil}\ge \frac{|Z|!}{|Z|^{k-1}\left(\left(\frac{1}{k}|Z'|\right)!\right)^k}\cdot \binom{\lceil \tau |Z|\rceil}{\lceil\delta |Z|\rceil}\Big/\binom{|Z|}{\lceil\delta |Z|\rceil},$$
where $\Xi_W$ is the set of $(\psi, W)\in \Xi$.

The number of maps $\phi: Z\setminus W\rightarrow [k]$ is $k^{|Z\setminus W|}$. Thus we can find some $\phi: Z\setminus W\rightarrow [k]$
such that
$$|\Xi_{W, \phi}|\ge |\Xi_W|/k^{|Z\setminus W|},$$
where $\Xi_{W, \phi}$ is the set of $(\psi, W)\in \Xi_W$ satisfying $\psi|_{Z\setminus W}=\phi$.

We denote by $S$ the set of $\psi|_W$ for $(\psi, W)\in \Xi_{W, \phi}$. The map $\Xi_{W, \phi}\rightarrow S$ sending $(\psi, W)$ to $\psi|_W$ is a bijection. Thus
$$ |S|= |\Xi_{W, \phi}|\ge |\Xi_W|/k^{|Z\setminus W|}\ge \frac{|Z|!}{|Z|^{k-1}\left(\left(\frac{1}{k}|Z'|\right)!\right)^k}\cdot \binom{\lceil \tau |Z|\rceil}{\lceil\delta |Z|\rceil}\Big/\left(\binom{|Z|}{\lceil\delta |Z|\rceil}k^{|Z\setminus W|}\right).$$

Since $(\tau-\delta)|Z|\ge (\tau-\delta)N>1$, we have $\lceil \tau |Z|\rceil>\lceil\delta|Z|\rceil$. We can estimate the last term of the above inequalities as follows:
\begin{align*}
 & \frac{|Z|!}{|Z|^{k-1}\left(\left(\frac{1}{k}|Z'|\right)!\right)^k}\cdot \binom{\lceil \tau |Z|\rceil}{\lceil\delta |Z|\rceil}\Big/\left(\binom{|Z|}{\lceil\delta |Z|\rceil}k^{|Z\setminus W|}\right)\\
&=\frac{\left(\lceil \tau |Z|\rceil\right)!\cdot \left(|Z|-\lceil\delta|Z|\rceil\right)!}{|Z|^{k-1}\left(\left(\frac{1}{k}|Z'|\right)!\right)^k\cdot \left(\lceil \tau |Z|\rceil-\lceil\delta|Z|\rceil\right)!\cdot k^{|Z|-\lceil\delta |Z|\rceil}}\\
&\overset{\eqref{E-Stirling}}\ge \frac{e \left(\frac{\lceil \tau |Z|\rceil}{e}\right)^{\lceil \tau |Z|\rceil}e\left(\frac{|Z|-\lceil\delta |Z|\rceil}{e}\right)^{|Z|-\lceil\delta |Z|\rceil}}{|Z|^{k-1}\left(\frac{e}{k}|Z'|\left(\frac{1}{ek}|Z'|\right)^{\frac{1}{k}|Z'|}\right)^k\cdot e\left(\lceil \tau |Z|\rceil-\lceil\delta|Z|\rceil\right)\left(\frac{\lceil \tau |Z|\rceil-\lceil\delta|Z|\rceil}{e}\right)^{\lceil \tau |Z|\rceil-\lceil\delta|Z|\rceil}\cdot k^{|Z|-\lceil\delta |Z|\rceil}}\\
&\ge \frac{e \left(\frac{\lceil \tau |Z|\rceil}{e}\right)^{\lceil \tau |Z|\rceil}e\left(\frac{|Z|-\lceil\delta |Z|\rceil}{e}\right)^{|Z|-\lceil\delta |Z|\rceil}}{|Z|^{k-1}\left(\frac{e}{k}|Z|\left(\frac{1}{ek}|Z|\right)^{\frac{1}{k}|Z|}\right)^k\cdot e\left(\lceil \tau |Z|\rceil-\lceil\delta|Z|\rceil\right)\left(\frac{\lceil \tau |Z|\rceil-\lceil\delta|Z|\rceil}{e}\right)^{\lceil \tau |Z|\rceil-\lceil\delta|Z|\rceil}\cdot k^{|Z|-\lceil\delta |Z|\rceil}}\\
&= \frac{ k^{k+\lceil\delta|Z|\rceil}\lceil \tau |Z|\rceil^{\lceil \tau |Z|\rceil}\left(|Z|-\lceil\delta |Z|\rceil\right)^{|Z|-\lceil\delta |Z|\rceil}}{e^{k-1}|Z|^{2k-1+|Z|}\cdot \left(\lceil \tau |Z|\rceil-\lceil\delta|Z|\rceil\right)^{\lceil \tau |Z|\rceil-\lceil\delta|Z|\rceil+1}}\\
&\ge \frac{ k^{k+\delta|Z|}(\tau |Z|)^{ \tau |Z|}(|Z|-\delta |Z|-1)^{|Z|-\delta |Z|-1}}{e^{k-1}|Z|^{2k-1+|Z|}\cdot (\tau |Z|-\delta|Z|+1)^{\tau |Z|-\delta|Z|+2}}.
\end{align*}
Note that $(1+\frac{1}{t})^t\le e$ for all $t> 0$. Since  $|Z|-\delta |Z|\ge (1-\delta)N>1$, we have
$$\left(|Z|-\delta|Z|-1\right)^{|Z|-\delta|Z|-1}\ge \left(|Z|-\delta|Z|\right)^{|Z|-\delta|Z|-1}/e$$
and
 $$\left(\tau|Z|-\delta|Z|+1\right)^{\tau|Z|-\delta|Z|}\le e\left(\tau|Z|-\delta|Z|\right)^{\tau|Z|-\delta|Z|},$$
 and hence
\begin{align*}
|S|&\ge \frac{ k^{k+\delta|Z|}(\tau |Z|)^{ \tau |Z|}(|Z|-\delta |Z|-1)^{|Z|-\delta |Z|-1}}{e^{k-1}|Z|^{2k-1+|Z|}\cdot (\tau |Z|-\delta|Z|+1)^{\tau |Z|-\delta|Z|+2}}\\
&\ge \frac{ k^{k+\delta|Z|}\left(\tau |Z|\right)^{ \tau |Z|}\left(|Z|-\delta |Z|\right)^{|Z|-\delta |Z|-1}}{e^{k+1}|Z|^{2k-1+|Z|}\cdot \left(\tau |Z|-\delta|Z|+1\right)^2\left(\tau |Z|-\delta|Z|\right)^{\tau |Z|-\delta|Z|}}\\
&= \frac{ k^{k+\delta|Z|}\tau^{ \tau |Z|}(1-\delta)^{|Z|-\delta |Z|-1}}{e^{k+1}|Z|^{2k}\cdot \left(\tau |Z|-\delta|Z|+1\right)^2(\tau-\delta)^{\tau |Z|-\delta|Z|}}\\
&=\frac{ (1-\delta)^{-1}k^k}{e^{k+1}|Z|^{2k} \left(\tau|Z|-\delta|Z|+1\right)^2} e^{|Z|(-g(\tau)-g(1-\delta)+g(\tau-\delta)+\delta \log k)}\\
&\overset{\eqref{E-half to indep1}}\ge \frac{ (1-\delta)^{-1}k^k}{e^{k+1}|Z|^{2k} \left(\tau|Z|-\delta|Z|+1\right)^2} e^{|Z|\delta \theta_1}\\
&\overset{\eqref{E-half to indep2}}\ge e^{\theta_2+|Z|\delta \theta_2}\ge e^{\lceil \delta |Z|\rceil \theta_2}=e^{|W|\theta_2}=((k-1)\lambda)^{|W|}.
\end{align*}
By our choice of $c$, we can find some $J\subseteq W$ such that $|J|\ge c|W|$ and $S|_J=[k]^J$. Then $|J|\ge c|W|\ge c\delta |Z|$, and every map $\sigma: J\rightarrow [k]$ extends to a $\psi\in \sR(Z, k)$ such that $J\subseteq W\subseteq Z_\psi$.
\end{proof}

The condition $\tau>(k-1)/k$ in Lemma~\ref{L-half to indep} cannot be weakened to $\tau\ge (k-1)/k$, as the following example shows.

\begin{example} \label{E-half to indep}
Let $k\ge 2$ be an integer. Let $N\in k^2\Nb$ and let $Z=[N]$. Let $\psi\in \sR(Z, k)$.  We choose a $Z_\psi\subseteq Z$ with $|Z_\psi|=\frac{k-1}{k}|Z|$ as follows. For each $j\in [k]$, noting that $|\psi^{-1}(j)|=|Z|/k$, we may list the element of $\psi^{-1}(j)$ as $$a_{j, 1}<a_{j, 2}<\dots< a_{j, N/k},$$
and set
$$b_{j, i}=a_{j, iN/k^2}$$
for $i\in [k]$. We take the $j_1\in [k]$ such that $b_{j_1, 1}$ is the largest among $b_{j, 1}$ for $j\in [k]$. Next we take the $j_2\in [k]\setminus \{j_1\}$ such that $b_{j_2, 2}$ is the largest among $b_{j, 2}$ for $j\in [k]\setminus \{j_1\}$. Inductively, we take the $j_l\in [k]\setminus \{j_1, \dots, j_{l-1}\}$ such that $b_{j_l, l}$ is the largest among $b_{j, l}$ for $j\in [k]\setminus \{j_1, \dots, j_{l-1}\}$. Then $\{j_1, \dots, j_k\}=[k]$. We set
\begin{align*}
Z_{\psi, j_l}&=\psi^{-1}(j_l)\setminus \{a_{j_l, i}: (l-1)N/k^2<i\le lN/k^2\}\\
&=\{z\in \psi^{-1}(j_l): z\le b_{j_l, l-1} \mbox{ or } z>b_{j_l, l}\}
\end{align*}
for each $l\in [k]$, where we set $b_{j_1, 0}=0$, and
$$Z_\psi=\bigcup_{j\in [k]}Z_{\psi, j}.$$
Then
$$|Z_{\psi, j}|=\frac{k-1}{k}|\psi^{-1}(j)|=\frac{k-1}{k}\frac{|Z|}{k}$$
for each $j\in [k]$, and
$$|Z_\psi|=\frac{k-1}{k}|Z|.$$

We claim that there is no strictly increasing map $\xi: [k]\rightarrow Z_\psi$ such that 
$$\psi \circ \xi(l)=j_l$$ 
for all $l\in [k]$. Assume that there is one such $\xi$. Then $\xi(l)\in Z_\psi\cap \psi^{-1}(j_l)=Z_{\psi, j_l}$ for each $l\in [k]$. Note that $\xi(2)>\xi(1)>b_{j_1, 1}>b_{j_2, 1}$, whence $\xi(2)>b_{j_2, 2}$. Then $\xi(3)>\xi(2)>b_{j_2, 2}>b_{j_3, 2}$, whence $\xi(3)>b_{j_3, 3}$. Going on this way, we get $\xi(l)>b_{j_l, l}$ for all $l\in [k]$. In particular, $\xi(k)>b_{j_k, k}$, which is impossible.  This proves our claim.

Now we claim that there is no $J\subseteq Z$ with $|J|\ge k\cdot k!$ such that every map $\sigma: J\rightarrow [k]$ extends to a $\psi\in \sR(Z, k)$ so that $J\subseteq Z_\psi$. Assume that there is one such  $J$. 
We denote by $\Sym(k)$ the permutation group of $[k]$. Then the map $\phi': [k]\rightarrow [k]$ sending $l$ to $j_l$ lies in $\Sym(k)$. Since $|J|\ge k\cdot k!$, we can find pairwise disjoint sets $J_{\phi}\subseteq J$ for $\phi\in \Sym(k)$ such that $|J_{\phi}|=k$ for all $\phi\in \Sym(k)$. We list the elements of $J_{\phi}$ as $a_{\phi, 1}<a_{\phi, 2}<\dots<a_{\phi, k}$ for each $\phi\in \Sym(k)$.  Then we can find  a map $\sigma: J\rightarrow [k]$ such that
$$\sigma(a_{\phi, l})=\phi(l)$$
for all $\phi\in \Sym(k)$ and $l\in [k]$. By assumption $\sigma$ extends to some $\psi\in \sR(Z, k)$ such that $J\subseteq Z_\psi$. We define  a map $\xi: [k]\rightarrow J_{\phi'}\subseteq J\subseteq Z_\psi$ by $\xi(l)=a_{\phi', l}$ for all $l\in [k]$. Then $\xi: [k]\rightarrow Z_\psi$ is strictly increasing and 
$$\psi \circ \xi(l)=\sigma(\xi(l))=\sigma(a_{\phi', l})=\phi'(l)=j_l$$ 
for all $l\in [k]$, which is impossible as we showed in the above paragraph. This proves our claim. 

For any $c>0$, when $|Z|\in k^2\Nb$ is large enough, one has $c|Z|\ge k\cdot k!$. This shows that the condition $\tau>(k-1)/k$ in Lemma~\ref{L-half to indep} is optimal, even if we require that $|Z_\psi\cap \psi^{-1}(j)|$ is the same for all $j\in [k]$.
\end{example}

We are ready to prove the case $k=2$ of Theorem~\ref{T-IE main}.

\begin{proof}[Proof of the case $k=2$ of Theorem~\ref{T-IE main}] In view of Lemma~\ref{L-Hahn Banach IE}, it suffices to show that given any $(\mu_1, \mu_2)\in \IE_2(\cM(X))$ and $f_1, f_2\in C(X)$ satisfying $\sum_{j\in [2]}\mu_j(f_j)>0$, there is a pair $(x_1, x_2)\in \IE_2(X)$ such that $\sum_{j\in [2]}f_j(x_j)>0$.

We have $\delta, A_1, A_2$ and $V_1\times V_2$ given by Lemma~\ref{L-large density}. We set $k=2$ and $\tau=\frac{1}{2}(1+\frac{\delta}{2})$. Then we have $c$ and $N$ given by Lemma~\ref{L-half to indep}. We take an $N_1\ge \max(N, 2)$ such that $\frac{t-1}{2}\cdot (1+\delta)>t\tau$ for all $t\ge N_1$.

Let $F\in \cF(\Gamma)$ be an independence set for $(V_1, V_2)$ such that $|F|\ge N_1$. Let $\psi\in \sR(F, 2)$. Then the map $\psi: F\rightarrow [2]$ is surjective, whence by our choice of $\delta, A_1, A_2$ and $V_1\times V_2$ there is an $F_\psi\subseteq F$ such that 
$$\sum_{j\in [2]}\frac{|F_\psi\cap \psi^{-1}(j)|}{|\psi^{-1}(j)|}>1+\delta \quad \text{ and } \quad \bigcap_{s\in F_\psi}s^{-1}A_{\psi(s)}\neq \emptyset.$$ 
Note that
\begin{align*}
|F_\psi|&=\sum_{j\in [2]}\left(|\psi^{-1}(j)|\cdot \frac{|F_\psi\cap \psi^{-1}(j)|}{|\psi^{-1}(j)|}\right)\ge \sum_{j\in [2]}\left(\frac{|F|-1}{2}\cdot \frac{|F_\psi\cap \psi^{-1}(j)|}{|\psi^{-1}(j)|}\right)\\
&\ge \frac{|F|-1}{2}\cdot(1+\delta)\ge \tau |F|.
\end{align*}
By our choice of $c$ and $N$, we can find a $J\subseteq F$ with $|J|\ge c|F|$ such that every map $\sigma: J\rightarrow [2]$ extends to a $\psi\in \sR(F, 2)$ so that $J\subseteq F_\psi$. Then
$$\bigcap_{s\in J}s^{-1}A_{\sigma(s)}\supseteq \bigcap_{s\in F_\psi}s^{-1}A_{\psi(s)}$$
is nonempty. This means that $J$ is an independence set for $(A_1, A_2)$.

Since $(\mu_1, \mu_2)\in \IE_2(\cM(X))$, the pair $(V_1, V_2)$ has positive independence density. It follows that  $(A_1, A_2)$ has positive independence density. By Theorem~\ref{T-IE basic}.\eqref{i-IE tuple} there is some $(x_1, x_2)\in \IE_2(X)$ such that $x_j\in A_j$ for each $j\in [2]$. Then 
$$\sum_{j\in [2]}f_j(x_j)\ge \sum_{j\in [2]}\min_{x\in A_j}f_j(x)>0$$ 
by our choice of $A_1$ and $A_2$.
\end{proof}

%%%%%%%%%%%%%%%%%%%%%%%%%%%%%%%%%%%%%%%%%%%%%%%%%%%%%%%%%%%%%%%%%%%%%%%%%%%%%%%%%%%%%%%%%%%%%%%%%%%%%
\subsection{Proof of Lemma~\ref{L-func to indep}} \label{SS-combinatorial}

In this subsection we prove our key combinatorial Lemma~\ref{L-func to indep}.

The following lemma says roughly that given a large enough finite set $Z$, for most $\psi\in [k]^Z$, the partition of $Z$ given by $\psi^{-1}(j)$ for $j\in [k]$ is almost even distributed and furthermore is almost independent to any given partition $\cP$ of $Z$.

\begin{lemma} \label{L-indep to partition}
Let $k\ge 2$ be an integer,  $0<\eta<1$, and $0<\varepsilon<1/k$. Then there are some $\delta_1, N_1>0$ depending only on $k, \eta$ and $\varepsilon$ such that the following holds.  For any finite set $Z$ with $|Z|\ge N_1$, and any partition $\cP$ of $Z$,  setting
$$V_{\cP, \varepsilon, \eta}=\left\{\psi\in [k]^Z: \left|\frac{|\psi^{-1}(j)\cap P|}{|P|}-\frac{1}{k}\right|\ge \varepsilon \mbox{ for some } j\in [k] \mbox{ and }  P\in \cP \mbox{ with } \frac{|P|}{|Z|}\ge \eta\right\},$$
we have
$$ |V_{\cP, \varepsilon, \eta}|\le \left|[k]^Z\right|e^{-\delta_1|Z|}.$$
\end{lemma}
\begin{proof}
Recall the function $g: [0, 1]\rightarrow \Rb$ defined in \eqref{E-Shannon}. We define a 
function $h: [0, 1]\rightarrow \Rb$ by
$$h(t)=g(t)+g(1-t)+(1-t)\log (k-1)$$
for $t\in [0, 1]$.
Note that
$$h'(t)=-\log t+\log (1-t)-\log (k-1)$$
for $t\in (0, 1)$. Clearly $h'$ is strictly decreasing on $(0, 1)$, and $h'(\frac{1}{k})=0$. Thus $h$ is strictly increasing on $[0, \frac{1}{k}]$ and strictly decreasing on $[\frac{1}{k}, 1]$. Note that
$$ h(\frac{1}{k})=\frac{1}{k}\log k+\frac{k-1}{k}\log \frac{k}{k-1}+\frac{k-1}{k}\log (k-1)=\log k,$$
and
$$h(0)=\log (k-1).$$
We set
$$ \delta_1=\frac{\eta}{2}\left(h(\frac{1}{k})-\max_{t\in [0, \frac{1}{k}-\varepsilon]\cup [\frac{1}{k}+\varepsilon, 1]}h(t)\right)>0.$$
Then 
\begin{align} \label{E-indep to partition}
\delta_1\le \frac{\eta}{2}\left(h(\frac{1}{k})-h(0)\right)=\frac{\eta}{2}\left(h(\frac{1}{k})-\log (k-1)\right).
\end{align}
We take a large $N\ge 3$ such that 
$$t(t+1)k/\eta\le e^{t\delta_1/\eta+1}$$ 
for all $t\ge N$. We set $N_1:=N/\eta$. We shall show that $\delta_1$ and $N_1$ have the desired property. 

Let $Z$ be a finite set with $|Z|\ge N_1$ and let $\cP$ be a partition of $Z$.

We denote by $\cP_\eta$ the set of $P\in \cP$ satisfying $\frac{|P|}{|Z|}\ge \eta$. For each $P\in \cP_\eta$ and $j\in [k]$, we set
$$ V_{P, j, \varepsilon}=\left\{\psi\in [k]^Z: \left|\frac{|\psi^{-1}(j)\cap P|}{|P|}-\frac{1}{k}\right|\ge \varepsilon\right\},$$
and
$$ W_{P, j, \varepsilon}=\left\{\psi\in [k]^P: \left|\frac{|\psi^{-1}(j)|}{|P|}-\frac{1}{k}\right|\ge \varepsilon\right\}.$$
We have
\begin{align*}
V_{\cP, \varepsilon, \eta}=\bigcup_{P\in \cP_\eta}\bigcup_{j\in [k]}V_{P, j, \varepsilon},
\end{align*}
whence
\begin{align*}
|V_{\cP, \varepsilon, \eta}|\le \sum_{P\in \cP_\eta, j\in [k]}|V_{P, j, \varepsilon}|=  \sum_{P\in \cP_\eta, j\in [k]}|W_{P, j, \varepsilon}|k^{|Z|-|P|}.
\end{align*}

Let $P\in \cP_\eta$ and $j\in [k]$. We have
$$ |W_{P, j, \varepsilon}|= \sum_{0\le i\le |P|, \left|\frac{i}{|P|}-\frac{1}{k}\right|\ge \varepsilon} (k-1)^{|P|-i}\binom{|P|}{i}.$$
When $\left|\frac{i}{|P|}-\frac{1}{k}\right|\ge \varepsilon$ and $0<i<|P|$, we have
\begin{align*}
(k-1)^{|P|-i}\binom{|P|}{i}&=(k-1)^{|P|-i}\frac{|P|!}{i! (|P|-i)!}\overset{\eqref{E-Stirling}}\le \frac{(k-1)^{|P|-i}e|P|(\frac{|P|}{e})^{|P|}}{e(\frac{i}{e})^i e(\frac{|P|-i}{e})^{|P|-i}}\\
&=(k-1)^{|P|-i}e^{-1}|P|e^{|P|\left(g(\frac{i}{|P|})+g(\frac{|P|-i}{|P|})\right)}\\
&=e^{-1}|P|e^{|P|h(\frac{i}{|P|})}\le e^{-1}|P|e^{|P|\left(h(\frac{1}{k})-2\delta_1/\eta\right)}. 
\end{align*}
Note that $|P|\ge \eta|Z|\ge N\ge 3$. For $i=0$ or $|P|$, clearly
$$ (k-1)^{|P|-i}\binom{|P|}{i}\le (k-1)^{|P|}\le e^{-1}|P|(k-1)^{|P|}\overset{\eqref{E-indep to partition}}\le e^{-1}|P|e^{|P|\left(h(\frac{1}{k})-2\delta_1/\eta\right)}.$$
Thus
$$(k-1)^{|P|-i}\binom{|P|}{i}\le e^{-1}|P|e^{|P|\left(h(\frac{1}{k})-2\delta_1/\eta\right)}$$
for every integer $0\le i\le |P|$ satisfying $\left|\frac{i}{|P|}-\frac{1}{k}\right|\ge \varepsilon$. 
Since $|P|\ge N$, we have
$$ |P|(|P|+1)k/\eta\le e^{|P|\delta_1/\eta+1}.$$
Therefore
$$ |W_{P, j, \varepsilon}|\le (|P|+1)e^{-1}|P|e^{|P|\left(h(\frac{1}{k})-2\delta_1/\eta\right)}\le  \frac{\eta}{k}e^{|P|\left(h(\frac{1}{k})-\delta_1/\eta\right)}=\frac{\eta}{k}e^{|P|(\log k-\delta_1/\eta)}.$$

Note that $|\cP_\eta|\le \frac{1}{\eta}$. 
Thus
\begin{align*}
|V_{\cP, \varepsilon, \eta}|&\le \sum_{P\in \cP_\eta, j\in [k]}|W_{P, j, \varepsilon}|k^{|Z|-|P|}\le  \sum_{P\in \cP_\eta, j\in [k]}\frac{\eta}{k}e^{|P|(\log k-\delta_1/\eta)}k^{|Z|-|P|}\\
&=\sum_{P\in \cP_\eta, j\in [k]}\frac{\eta}{k}e^{|Z|\log k-|P|\delta_1/\eta}\le \sum_{P\in \cP_\eta, j\in [k]}\frac{\eta}{k}e^{|Z|(\log k- \delta_1)}\\
&=|\cP_\eta|\eta e^{|Z|(\log k- \delta_1)}\le e^{|Z|(\log k-\delta_1)}=\left|[k]^Z\right|e^{-\delta_1 |Z|}. \tag*{\qedsymbol}
\end{align*}
     \renewcommand{\qedsymbol}{}
     \vspace{-\baselineskip}
\end{proof}

\begin{notation} \label{N-cover}
Let $k\in \Nb$ and let $Z$ be a finite set. For each  $\phi: Z\rightarrow [k]$, we denote by $U_\phi$ the subset $\prod_{z\in Z}(\{0, 1, \dots, k\}\setminus \{\phi(z)\})$ of $\{0, 1, \dots, k\}^Z$. For any $S\subseteq \{0, 1, \dots, k\}^Z$, we denote by $N_S$ the minimal number of $U_\phi$ needed to cover $S$.
\end{notation}

We need the following lemma of Kerr and Li \cite[Lemma 3.3]{KL07} \cite[Lemma 12.13]{KL16}.

\begin{lemma} \label{L-07}
Let $k\ge 2$ be an integer and $b>0$. There exists a constant $c>0$ depending only on $k$ and $b$ such that for every nonempty finite set $Z$ and $S\subseteq \{0, 1, \dots, k\}^Z$ with $N_S\ge k^{b|Z|}$ there exists a $J\subseteq Z$ with $|J|\ge c|Z|$ and $S|_J\supseteq [k]^J$.
\end{lemma}

For any nonempty finite set $Z$ and any $f: Z\rightarrow \Rb$, we denote by $\|f\|_p$ the $\ell_p$-norm of $f$ for $1\le p\le \infty$, i.e., 
\begin{align}\label{E-p norm1}
 \|f\|_p=\left(\sum_{z\in Z}|f(z)|^p\right)^{1/p} 
 \end{align}
when $1\le p<\infty$ and
\begin{align} \label{E-p norm2}
 \|f\|_\infty=\max_{z\in Z}|f(z)|.
 \end{align}

\begin{lemma} \label{L-func to indep pre}
Let $k\in \Nb$ and $0<r<R\le C$. Also let $1<p\le \infty$ and $1\le q< \infty$ such that $1/p+1/q=1$. Then there exist $c, N, \delta>0$ and a finite set $T\subseteq \Rb^k$ depending only on $k, r, R, C$ and $p$ such that the following hold:
\begin{enumerate}
\item $\frac{1}{k}\sum_{j\in [k]}t_j\ge r$ for every $(t_1, \dots, t_k)\in T$, and
\item for any finite set $Z$ with $|Z|\ge N$ and any $\sR\subseteq [k]^Z$ with $|\sR|\ge \left| [k]^Z\right| e^{-\delta |Z|}$, if for each $\psi\in \sR$  we take an $f_\psi: Z\rightarrow \Rb$  with  $\|f_\psi\|_p\le C$ and $\frac{1}{|Z|^{1/q}}\sum_{z\in Z}f_\psi(z)\ge R$, then there are some $J\subseteq Z$ with $|J|\ge c|Z|$ and $(t_1, \dots, t_k)\in T$  such that every map $\sigma: J\rightarrow [k]$ extends to a $\psi\in \sR$ so that $f_\psi(z)\ge t_j|Z|^{-1/p}$ for all $j\in [k]$ and $z\in \sigma^{-1}(j)$.
\end{enumerate}
\end{lemma}
\begin{proof}
We consider first the case  $k=1$. We may take $N=1$, $T=\{r\}$ and any $c, \delta>0$ such that $R>\max_{t\in [0, c]}\left(Ct^{1/q}+r(1-t)\right)$. Indeed, for any nonempty finite set $Z$, there is only $\psi$ in $[k]^Z$. For any $f_\psi: Z\rightarrow \Rb$ with $\|f_\psi\|_p\le C$ and $\frac{1}{|Z|^{1/q}}\sum_{z\in Z}f_\psi(z)\ge R$, denoting by $J$ the set of $z\in Z$ satisfying $f_\psi(z)\ge r|Z|^{-1/p}$, using Hölder's inequality we have
\begin{align*}
R|Z|^{1/q}&\le \sum_{z\in Z}f_\psi(z)=\sum_{z\in J}f_\psi(z)+\sum_{z\in Z\setminus J}f_\psi(z)\le \|f_\psi\|_p\cdot \|{\bf 1}_J\|_q+\sum_{z\in Z\setminus J}r|Z|^{-1/p}\\
&\le C|J|^{1/q}+r|Z|^{-1/p}|Z\setminus J|,
\end{align*}
whence
$$R\le C\left(\frac{|J|}{|Z|}\right)^{1/q}+r\left(1-\frac{|J|}{|Z|}\right),$$
which implies $|J|>c|Z|$.

Next we consider the case $k\ge 2$.
We take a small $0<\tau<1/2$ such that
$$(k\tau)^{1/q}C\le \frac{R-r}{4}.$$

We take a large $M\in \Nb$ such that
\begin{align} \label{E-func to indep6}
\tau^{-1/p}\frac{C}{M}< \frac{R-r}{2}.
\end{align}
We denote by $I$ the set of $\boldsymbol{i}=(i_1, \dots, i_k)\in [2M]^k$ satisfying
$$\frac{1}{k}\sum_{j\in [k]}\tau^{-1/p}\left(-C+\frac{(i_j-1)C}{M}\right)\ge r.$$
When $i_j=2M$ for all $j\in [k]$, we have
\begin{align*}
 \frac{1}{k}\sum_{j\in [k]}\tau^{-1/p}\left(-C+\frac{(i_j-1)C}{M}\right)&=\tau^{-1/p}\left(-C+\frac{(2M-1)C}{M}\right)\\
 &=\tau^{-1/p}C-\tau^{-1/p}\frac{C}{M}\\
 &> C-\frac{R-r}{2}\ge R-\frac{R-r}{2}>r. 
 \end{align*}
Thus $(2M, \dots, 2M)\in I$, whence $I$ is nonempty. We denote by $T$ the set of
$$\tau^{-1/p}\left(-C+\frac{(i_1-1)C}{M}, \dots, -C+\frac{(i_k-1)C}{M}\right)\in [-\tau^{-1/p}C, \tau^{-1/p}C]^k$$
for $\boldsymbol{i}=(i_1, \dots, i_k)\in I$. Then $\frac{1}{k}\sum_{j\in [k]}t_j\ge r$ for every $(t_1, \dots, t_k)\in T$.

We take $0<\eta<1$ and $0<\varepsilon<1/k$ small enough such that
\begin{align} \label{E-func to indep4}
(k^{|I|}\eta+k\varepsilon+k\tau)^{1/q}C\le \frac{R-r}{2}.
\end{align}
Then we have $\delta_1, N_1>0$ given by Lemma~\ref{L-indep to partition}. We set 
$$\delta=\delta_1/3>0,$$ 
and take a large $N_2>0$ such that
\begin{align} \label{E-func to indep7}
|I|^{|I|}\le e^{\delta N_2}.
\end{align}
We set 
$$N=\max(1/(\varepsilon \eta), N_1, N_2)>0.$$ 
Let $c>0$ be given by Lemma~\ref{L-07} for $k$ and $b=\frac{\delta}{|I|\log k}$. 

We shall show that $c, N, \delta$ and $T$ have the desired property. We have checked that the condition (1) holds. Thus it suffices to check the condition (2). 

Let $Z$ be a finite set with $|Z|\ge N$ and let $\sR\subseteq [k]^Z$ with $|\sR|\ge \left| [k]^Z\right| e^{-\delta |Z|}$.
For each $\psi\in \sR$, we fix an $f_\psi: Z\rightarrow \Rb$ with $\|f_\psi\|_p\le C$ and  $\frac{1}{|Z|^{1/q}}\sum_{z\in Z}f_\psi(z)\ge R$.
For any $\phi, \psi: Z\rightarrow [k]$, we set
$$E_{\phi, \psi}=\{z\in Z: \phi(z)=\psi(z)\}.$$

For each $\psi\in \sR$ and $\boldsymbol{i}=(i_1, \dots, i_k)\in I$, we set
$$Z_{\psi, \boldsymbol{i}}=\bigcup_{j\in [k]}\left(\psi^{-1}(j)\cap f_\psi^{-1}\left((\tau|Z|)^{-1/p}\left[-C+\frac{(i_j-1)C}{M}, C\right]\right)\right)\subseteq Z,$$
and define a map $\psi_{\boldsymbol{i}}: Z\rightarrow \{0, 1, \dots, k\}$ by
\begin{align*}
\psi_{\boldsymbol{i}}(z)=\begin{cases}
\psi(z) & \quad \text{if } z\in Z_{\psi, \boldsymbol{i}}\\
0 & \quad \text{otherwise}.\\
\end{cases}
\end{align*}
We set
$$S_{\boldsymbol{i}}=\{\psi_{\boldsymbol{i}}: \psi\in \sR\}$$
for each $\boldsymbol{i}\in I$, and
$$S=\bigcup_{\boldsymbol{i}\in I}S_{\boldsymbol{i}}.$$

We shall apply Lemma~\ref{L-07} to $S_{\boldsymbol{i}}$ for certain $\boldsymbol{i}\in I$. For this purpose, we shall show that $N_S$ is large, and then conclude that $N_{S_{\boldsymbol{i}}}$ is large for at least one $\boldsymbol{i}\in I$. 
For any $\phi\in [k]^Z$, $\psi\in \sR$ and $\boldsymbol{i}\in I$, note that $\psi_{\boldsymbol{i}}\in U_\phi$ if and only if
$$E_{\phi, \psi}\subseteq  Z\setminus Z_{\psi, \boldsymbol{i}}.$$

Let $W\subseteq [k]^Z$ such that $U_\phi$ for $\phi\in W$ cover $S$. For each $\psi\in \sR$ and $\boldsymbol{i}\in I$, we can find a $\phi_{\psi, \boldsymbol{i}}\in W$ such that $\psi_{\boldsymbol{i}}\in U_{\phi_{\psi, \boldsymbol{i}}}$. This gives us a map $\pi: \sR\rightarrow W^I$ sending $\psi$ to $(\phi_{\psi, \boldsymbol{i}})_{\boldsymbol{i}\in I}$.

Let $w=(w_{\boldsymbol{i}})_{\boldsymbol{i}\in I}\in W^I$. For each $\psi\in \pi^{-1}(w)$, we have $\psi_{\boldsymbol{i}}\in U_{w_{\boldsymbol{i}}}$ and hence
\begin{align} \label{E-func to indep5}
 E_{w_{\boldsymbol{i}}, \psi}\subseteq  Z\setminus Z_{\psi, \boldsymbol{i}}
 \end{align}
for every $\boldsymbol{i}\in I$.
For each $\boldsymbol{i}\in I$, we denote by $\cP_{\boldsymbol{i}}$ the partition of $Z$ consisting of $w_{\boldsymbol{i}}^{-1}(j)$ for $j\in [k]$. Let $\cP$ be the join $\bigvee_{\boldsymbol{i}\in I}\cP_{\boldsymbol{i}}$ of $\cP_{\boldsymbol{i}}$ for $\boldsymbol{i}\in I$, i.e., the partition of $Z$ consisting of sets of the form $\bigcap_{\boldsymbol{i}\in I}P_{\boldsymbol{i}}$ with $P_{\boldsymbol{i}}\in \cP_{\boldsymbol{i}}$ for every $\boldsymbol{i}\in I$. We shall show that
\begin{align} \label{E-func to indep}
\pi^{-1}(w)\subseteq V_{\cP, \varepsilon, \eta},
\end{align}
where $V_{\cP, \varepsilon, \eta}$ is defined in Lemma~\ref{L-indep to partition}.
For each $h\in [k]^I$, let
$$P_h=\bigcap_{\boldsymbol{i}\in I} w_{\boldsymbol{i}}^{-1}(h(\boldsymbol{i}))\subseteq Z.$$
Then $\cP=\{P_h: h\in [k]^I\}$.  We denote by $H_\eta$ the set of $h\in  [k]^I$ satisfying $\frac{|P_h|}{|Z|}\ge \eta$.

We claim that for any $\psi\in \pi^{-1}(w)$ and any $h\in [k]^I$, if $z_j\in \psi^{-1}(j)\cap P_h$ and $|f_\psi(z_j)|\le (\tau|Z|)^{-1/p}C$ for each $j\in [k]$, then
\begin{align} \label{E-func to indep2}
\frac{1}{k}\sum_{j\in [k]}f_\psi(z_j)<\left(r+\tau^{-1/p}\frac{C}{M}\right)|Z|^{-1/p}.
\end{align}
Assume that $\frac{1}{k}\sum_{j\in [k]}f_\psi(z_j)\ge \left(r+\tau^{-1/p}\frac{C}{M}\right)|Z|^{-1/p}$ instead. For each $j\in [k]$, since $|f_\psi(z_j)|\le (\tau|Z|)^{-1/p}C$, 
we can find an $i_j\in [2M]$ such that
\begin{align} \label{E-func to indep3}
|Z|^{1/p}f_\psi(z_j)\in \tau^{-1/p}\left[-C+\frac{(i_j-1)C}{M}, -C+\frac{i_jC}{M}\right].
\end{align}
This gives us an element $\boldsymbol{i}=(i_1, \dots, i_k)$ of $[2M]^k$. Then
$$\frac{1}{k}\sum_{j\in [k]}\tau^{-1/p}\left(-C+\frac{(i_j-1)C}{M}\right)\ge  \frac{1}{k}\sum_{j\in [k]}\left(|Z|^{1/p}f_\psi(z_j)-\tau^{-1/p}\frac{C}{M}\right)\ge r,$$
whence $\boldsymbol{i}\in I$. 
For each $z\in \psi^{-1}(h(\boldsymbol{i}))\cap P_h$, we have 
$$\psi(z)=h(\boldsymbol{i})=w_{\boldsymbol{i}}(z),$$ 
whence $z\in E_{w_{\boldsymbol{i}}, \psi}$. This means
$$\psi^{-1}(h(\boldsymbol{i}))\cap P_h\subseteq E_{w_{\boldsymbol{i}}, \psi}.$$
Thus 
\begin{align*}
z_{h(\boldsymbol{i})}\in \psi^{-1}(h(\boldsymbol{i}))\cap P_h\subseteq E_{w_{\boldsymbol{i}}, \psi}\overset{\eqref{E-func to indep5}}\subseteq Z\setminus Z_{\psi, \boldsymbol{i}},
\end{align*}
and hence $z_{h(\boldsymbol{i})}\not\in \psi^{-1}(h(\boldsymbol{i}))\cap f_\psi^{-1}\left((\tau|Z|)^{-1/p}\left[-C+\frac{(i_{h(\boldsymbol{i})}-1)C}{M}, C\right]\right)$. Since $z_{h(\boldsymbol{i})}\in \psi^{-1}(h(\boldsymbol{i}))$, it follows that
$z_{h(\boldsymbol{i})}\not\in f_\psi^{-1}\left((\tau|Z|)^{-1/p}\left[-C+\frac{(i_{h(\boldsymbol{i})}-1)C}{M}, C\right]\right)$, a contradiction to \eqref{E-func to indep3}. This proves our claim.

To prove \eqref{E-func to indep}, we argue by contradiction. Assume that 
\eqref{E-func to indep} fails.
Then we can find  a $\psi\in \pi^{-1}(w)\setminus V_{\cP, \varepsilon, \eta}$.
We denote by $Z_\psi$ the set of $z\in Z$ such that $|f_\psi(z)|>(\tau|Z|)^{-1/p} C$. When $p=\infty$, $Z_\psi$ is empty. When $p<\infty$, we have
$$ \left((\tau|Z|)^{-1} C^p |Z_\psi|\right)^{1/p} \le \|f_\psi\|_p\le C,$$
whence $|Z_\psi|\le \tau |Z|$. Thus we always have
$$|Z_\psi|\le \tau|Z|.$$
For each $h\in H_\eta$, since $\frac{|P_h|}{|Z|}\ge \eta$ and $\psi\not\in V_{\cP, \varepsilon, \eta}$, we have
 $$ \left|\frac{|\psi^{-1}(j)\cap P_h|}{|P_h|}-\frac{1}{k}\right|< \varepsilon$$
 for each $j\in [k]$, thus we can take a set $Y_h$ with $|Y_h|=\lceil|P_h|(\frac{1}{k}-\varepsilon)\rceil$ and an injective map $g_{h, j}: Y_h \rightarrow \psi^{-1}(j)\cap P_h$ for each $j\in [k]$. For each $h\in H_\eta$, we denote by $Y_h'$ the set of $y\in Y_h$ such that $g_{h, j}(y)\in Z\setminus Z_\psi$ for all $j\in [k]$. The sets $\{g_{h, j}(y): j\in [k]\}$ for $h\in H_\eta$ and $y\in Y_h\setminus Y_h'$ are pairwise disjoint, and each of them intersects with $Z_\psi$. Thus
 \begin{align*}
 \left|\bigcup_{h\in H_\eta}\bigcup_{j\in [k]}g_{h, j}(Y_h\setminus Y_h')\right|=\sum_{h\in H_\eta}\sum_{y\in Y_h\setminus Y_h'}|\{g_{h, j}(y): j\in [k]\}|\le k|Z_\psi|.
 \end{align*}
 Therefore
 \begin{align*}
 &\left|Z\setminus \bigcup_{h\in H_\eta}\bigcup_{j\in [k]}g_{h, j}(Y_h')\right|\\
 &\quad \quad =\left|\bigcup_{h\in [k]^I\setminus H_\eta}P_h\right|+\left|\bigcup_{h\in H_\eta}\left(P_h\setminus \bigcup_{j\in [k]}g_{h, j}(Y_h)\right)\right|+\left|\bigcup_{h\in H_\eta}\bigcup_{j\in [k]}g_{h, j}(Y_h\setminus Y_h')\right|\\
 &\quad \quad \le k^{|I|}\eta |Z|+\sum_{h\in H_\eta}k\varepsilon|P_h|+k|Z_\psi|\\
 &\quad \quad \le k^{|I|}\eta |Z|+k\varepsilon|Z|+k\tau |Z|,
 \end{align*}
 whence by H\"{o}lder's inequality 
 \begin{align*}
 \sum_{z\in Z\setminus \bigcup_{h\in H_\eta}\bigcup_{j\in [k]}g_{h, j}(Y_h')}f_\psi(z)&\le \|f_\psi\|_p\cdot \|{\bf 1}_{Z\setminus \bigcup_{h\in H_\eta}\bigcup_{j\in [k]}g_{h, j}(Y_h')}\|_q\\
 &\le C\left|Z\setminus \bigcup_{h\in H_\eta}\bigcup_{j\in [k]}g_{h, j}(Y_h')\right|^{1/q}\\
 &\le C(k^{|I|}\eta+k\varepsilon+k\tau)^{1/q}|Z|^{1/q}\\
 &\overset{\eqref{E-func to indep4}}\le \frac{R-r}{2}|Z|^{1/q}.
 \end{align*}
 For each $h\in H_\eta$ and $y\in Y'_h$, we have $g_{h, j}(y)\in \psi^{-1}(j)\cap P_h$ and $|f_\psi(g_{h, j}(y))|\le (\tau|Z|)^{-1/p}C$ for every $j\in [k]$, whence
 $$ \frac{1}{k}\sum_{j\in [k]}f_\psi(g_{h, j}(y))<\left(r+\tau^{-1/p}\frac{C}{M}\right)|Z|^{-1/p}$$
 by \eqref{E-func to indep2}. 
 For each $h\in H_\eta$, we have 
 $$|P_h|\varepsilon\ge |Z|\eta\varepsilon\ge N\eta \varepsilon \ge 1,$$  
 whence 
 $$|Y_h|=\lceil|P_h|(\frac{1}{k}-\varepsilon)\rceil< |P_h|(\frac{1}{k}-\varepsilon)+1\le |P_h|/k.$$
 Thus, for each $h\in H_\eta$, we have
 \begin{align*}
 \sum_{y\in Y'_h}\sum_{j\in [k]}f_\psi(g_{h, j}(y))&\le  |Y_h'| k\left(r+\tau^{-1/p}\frac{C}{M}\right)|Z|^{-1/p}\le  |Y_h| k\left(r+\tau^{-1/p}\frac{C}{M}\right)|Z|^{-1/p}\\
 &\le |P_h| \left(r+\tau^{-1/p}\frac{C}{M}\right)|Z|^{-1/p}. 
 \end{align*}
 Therefore
 \begin{align*}
 |Z|^{1/q}R &\le \sum_{z\in Z}f_\psi(z)=\sum_{h\in H_\eta}\sum_{y\in Y'_h}\sum_{j\in [k]}f_\psi(g_{h, j}(y))+\sum_{z\in Z\setminus \bigcup_{h\in H_\eta}\bigcup_{j\in [k]}g_{h, j}(Y_h')}f_\psi(z)\\
 &\le \sum_{h\in H_\eta}|P_h|\left(r+\tau^{-1/p}\frac{C}{M}\right)|Z|^{-1/p}+\frac{R-r}{2}|Z|^{1/q}\\
 &\le |Z|^{1/q}\left(r+\tau^{-1/p}\frac{C}{M}\right)+\frac{R-r}{2}|Z|^{1/q},
 \end{align*}
 a contradiction to \eqref{E-func to indep6}. This proves \eqref{E-func to indep}.

Since $|Z|\ge N\ge N_1$, by our choice of $\delta_1$ and $N_1$
we have
$$|V_{\cP, \varepsilon, \eta}|\le \left|[k]^Z\right|e^{-\delta_1 |Z|}=\left|[k]^Z\right|e^{-3\delta |Z|}.$$
As $|Z|\ge N\ge N_2$, from \eqref{E-func to indep7} we also have
$$ |I|^{|I|}\le e^{\delta N_2}\le e^{\delta |Z|}.$$
We conclude that for each $w\in W^I$ one has 
$$ |\pi^{-1}(w)|\overset{\eqref{E-func to indep}}\le |V_{\cP, \varepsilon, \eta}|\le \left|[k]^Z\right|e^{-3\delta |Z|}\le |\sR|e^{-2\delta |Z|}\le |\sR|e^{-\delta |Z|}|I|^{-|I|}.$$
Thus
$$ |W|^{|I|}=|W^I|\ge e^{\delta |Z|}|I|^{|I|},$$
and hence
$$ |W|\ge e^{\frac{\delta}{|I|}|Z|}|I|.$$
Then
$$ \sum_{\boldsymbol{i}\in I}N_{S_{\boldsymbol{i}}}\ge N_S\ge e^{\frac{\delta}{|I|}|Z|}|I|,$$
thus we can find some $\boldsymbol{i}=(i_1, \dots, i_k)\in I$ so that
$$ N_{S_{\boldsymbol{i}}}\ge e^{\frac{\delta}{|I|}|Z|}=k^{\frac{\delta}{|I|\log k}|Z|}=k^{b|Z|}.$$
We set $t_j=\tau^{-1/p}\left(-C+\frac{(i_j-1)C}{M}\right)$ for each $j\in [k]$. Then $(t_1, \dots, t_k)\in T$.
By our choice of $c$ we can find some $J\subseteq Z$ so that $|J|\ge c|Z|$ and $S_{\boldsymbol{i}}|_J\supseteq [k]^J$.  Then for every map $\sigma: J\rightarrow [k]$,
there is some $\psi\in \sR$ so that $\psi_{\boldsymbol{i}}$ extends  $\sigma$.
In particular, $\psi$ extends $\sigma$ and $f_\psi(z)\ge t_j|Z|^{-1/p}$ for all $j\in [k]$ and $z\in \sigma^{-1}(j)$.
\end{proof}

The following well-known lemma says that the exponential growth rate of  $|\sR(Z, k)|$ is the same as that of $\left|[k]^Z\right|$, so that  $\sR(Z, k)$ satisfies the hypothesis in Lemma~\ref{L-func to indep pre} for $\sR$. For convenience of the reader, we give a proof.

\begin{lemma} \label{L-regular}
Let $k\in \Nb$ and $\delta>0$. Then there is some $N>0$ depending only on $k$ and $\delta$ such that for any finite set $Z$ with $|Z|\ge N$ one has
$$|\sR(Z, k)|\ge \left|[k]^Z\right|e^{-\delta |Z|}.$$
\end{lemma}
\begin{proof}
We take a large $N\ge k$ such that  $e^{k-1}t^k\le e^{\delta t}$ for all $t\ge N$.  Let $Z$ be a finite set with $|Z|\ge N$.
We can find a set $Z'\subseteq Z$ such that $|Z'|$ is divisible by $k$ and $|Z\setminus Z'|<k$. Then
\begin{align*}
|\sR(Z, k)|&\ge |\sR(Z', k)|= \frac{|Z'|!}{\left(\left(\frac{|Z'|}{k}\right)!\right)^k}\overset{\eqref{E-Stirling}}\ge \frac{e\left(\frac{|Z'|}{e}\right)^{|Z'|}}{\left(e\frac{|Z'|}{k}\left(\frac{|Z'|}{ek}\right)^{\frac{|Z'|}{k}}\right)^k}= \frac{k^{|Z'|}}{e^{k-1}\left(\frac{|Z'|}{k}\right)^k}\\
&>\frac{k^{|Z|-k}}{e^{k-1}\left(\frac{|Z'|}{k}\right)^k}
=\frac{ k^{|Z|}}{e^{k-1}|Z'|^k}\ge  \frac{k^{|Z|}}{e^{k-1}|Z|^k}\ge |[k]^Z|e^{-\delta |Z|}. \tag*{\qedsymbol}
\end{align*}
     \renewcommand{\qedsymbol}{}
     \vspace{-\baselineskip}
\end{proof}

Now Lemma~\ref{L-func to indep} follows from Lemmas~\ref{L-func to indep pre} and ~\ref{L-regular}.

Via taking suitable $f_\psi$ in Lemma~\ref{L-func to indep}, one can obtain many consequences. As an example, we use it to obtain the following corollary, which in turn generalizes  Lemma~\ref{L-half to indep}.

\begin{corollary} \label{C-partial indep}
Let $k\in \Nb$, $m\in [k]$ and $\tau\in ((m-1)/k, 1]$. Then there exist $c, N>0$ depending only on $k, m$ and $\tau$ such that the following holds. For any finite set $Z$ with $|Z|\ge N$, if for each $\psi\in \sR(Z, k)$  we take a $Z_\psi\subseteq Z$  with $|Z_\psi|\ge \tau |Z|$, then there are some $J\subseteq Z$ with $|J|\ge c|Z|$ and $A\subseteq [k]$ with $|A|\ge m$ such that  every map $\sigma: J\rightarrow [k]$ extends to a $\psi\in \sR(Z, k)$ with $\sigma^{-1}(A)\subseteq Z_\psi$.
\end{corollary}
\begin{proof} Since $k\tau>m-1$, we can find a small $\theta>0$ such that
$$ \frac{k}{\frac{1}{\tau}+\theta}>m-1.$$
Then
$$ \frac{1}{\tau}-1+\theta\ge \theta\ge \theta \tau.$$
We set $C=\max(1, \frac{1}{\tau}-1+\theta)$, $R=\theta \tau$, $r=R/2$, and $p=\infty$. Then $q=1$ and we have $c, N>0$ and $T\subseteq \Rb^k$ given by Lemma~\ref{L-func to indep}.

Let $Z$ be  a finite set with $|Z|\ge N$, and let   $Z_\psi\subseteq Z$  with $|Z_\psi|\ge \tau |Z|$ for each $\psi\in \sR(Z, k)$.
For each $\psi\in \sR(Z, k)$, we define a function $f_\psi: Z\rightarrow [-C, C]$ by
\begin{align*}
f_{\psi}(z)=\begin{cases}
\frac{1}{\tau}-1+\theta & \quad \text{if } z\in Z_\psi\\
-1 & \quad \text{otherwise }.\\
\end{cases}
\end{align*}
Then
\begin{align*}
\frac{1}{|Z|}\sum_{z\in Z}f_\psi(z)=\frac{|Z_\psi|(\frac{1}{\tau}-1+\theta)-|Z\setminus Z_\psi|}{|Z|}= \frac{|Z_\psi|(\frac{1}{\tau}+\theta)}{|Z|}-1\ge R.
\end{align*}
By our choice of $c, N$ and $T$, we can find a $J\subseteq Z$ with $|J|\ge c|Z|$ and a $(t_1, \dots, t_k)\in T$ such that every map $\sigma: J\rightarrow [k]$ extends to some $\psi\in \sR(Z, k)$ so that $f_\psi(z)\ge t_j$ for all $j\in [k]$ and $z\in \sigma^{-1}(j)$. For each $j\in [k]$, taking a map $\sigma: J\rightarrow [k]$ with $j\in \sigma(J)$, we have
$$ t_j\le f_\psi(z)\le \frac{1}{\tau}-1+\theta$$
for all $z\in \sigma^{-1}(j)$.
We denote by $A$ the set of $j\in [k]$ satisfying $t_j>-1$.
Then by our choice of $T$ we have
\begin{align*}
0<r\le \frac{1}{k}\sum_{j\in [k]}t_j\le \frac{1}{k}\left(|A|(\frac{1}{\tau}-1+\theta)-\left|[k]\setminus A\right|\right)=\frac{1}{k}|A|(\frac{1}{\tau}+\theta)-1.
\end{align*}
Therefore
$$ |A|>\frac{k}{\frac{1}{\tau}+\theta}>m-1,$$
whence $|A|\ge m$.
Given any $\sigma$ and $\psi$ as above, for each $j\in A$, since $f_\psi(z)\ge t_j>-1$ for all $z\in \sigma^{-1}(j)$,
we have $\sigma^{-1}(j)\subseteq Z_\psi$.
\end{proof}

%%%%%%%%%%%%%%%%%%%%%%%%%%%%%%%%%%%%%%%%%%%%%%%%%%%%%%%%%%%%%%%%%%%%%%%%%%%%%%%%%%%%%%%%%%%%%%%%%%%%
\subsection{General case of Theorem~\ref{T-IE main}} \label{SS-general}

In this subsection we  prove the general case of Theorem~\ref{T-IE main}.

\begin{lemma} \label{L-func to func}
Let $k\in \Nb$, $\omu=(\mu_1, \dots, \mu_k)\in \cM(X)^k$ and $\of=(f_1, \dots, f_k)\in C(X)^{\oplus k}$. Let $R=\frac{1}{k}\sum_{j\in [k]}\mu_j(f_j)$ and let $r<R$. Then there is some product neighborhood $U_1\times \cdots \times U_k$ of $\omu$ in $\cM(X)^k$ such that the following holds. For any independence set $F\in \cF(\Gamma)$ of $(U_1, \dots, U_k)$ and any surjective map $\psi: F\rightarrow [k]$,  there is
some $x\in X$ such that
\begin{align} \label{E-func to func}
\frac{1}{k}\sum_{j\in [k]}\frac{1}{|\psi^{-1}(j)|}\sum_{s\in \psi^{-1}(j)}f_j(sx)\ge r.
\end{align}
\end{lemma}
\begin{proof} For each $j\in [k]$, we denote by $U_j$ the set of $\nu\in \cM(X)$ satisfying
$$\nu(f_j)\ge \mu_j(f_j)-(R-r).$$
This is a neighborhood of $\mu_j$ in $\cM(X)$.

Let $F\in \cF(\Gamma)$ be an independence set for $(U_1, \dots, U_k)$ and let  $\psi: F\rightarrow [k]$ be surjective.
Then $\bigcap_{s\in F}s^{-1}U_{\psi(s)}$ is nonempty, whence we can find  a $\nu\in \bigcap_{s\in F}s^{-1}U_{\psi(s)}$. We have
\begin{align*}
\int_X\left(\frac{1}{k}\sum_{j\in [k]}\frac{1}{|\psi^{-1}(j)|}\sum_{s\in \psi^{-1}(j)}f_j(sx)\right)\, d\nu(x)&=\frac{1}{k}\sum_{j\in [k]}\frac{1}{|\psi^{-1}(j)|}\sum_{s\in \psi^{-1}(j)}(s\nu)(f_j)\\
&\ge \frac{1}{k}\sum_{j\in [k]}\frac{1}{|\psi^{-1}(j)|}\sum_{s\in \psi^{-1}(j)}(\mu_j(f_j)-(R-r))\\
&=\frac{1}{k}\sum_{j\in [k]}(\mu_j(f_j)-(R-r))=r.
\end{align*}
Thus we can find some $x\in X$ such that \eqref{E-func to func} holds.
\end{proof}

\begin{notation} \label{N-tuple}
Let $k\in \Nb$, $\of=(f_1, \dots, f_k)\in C(X)^{\oplus k}$ and $\ot=(t_1, \dots, t_k)\in \Rb^k$. 
We denote by $\oA_{\of, \ot}$ the $k$-tuple
$$\left(f_1^{-1}\left([t_1, \infty)\right), \dots, f_k^{-1}\left([t_k, \infty)\right)\right)$$
of closed subsets of $X$.
\end{notation}

\begin{lemma} \label{L-func to density}
Let $k\in \Nb$, $\omu=(\mu_1, \dots, \mu_k)\in \cM(X)^k$ and $\of=(f_1, \dots, f_k)\in C(X)^{\oplus k}$ such that $R:=\frac{1}{k}\sum_{j\in [k]}\mu_j(f_j)>0$.  Let $0<r<R$ and $C\ge \|\of\|$. Then there are some $c, N>0$ and some finite set $T\subseteq \Rb^k$ depending only on $k, r, R$ and $C$,  and some product neighborhood $U_1\times \cdots \times U_k$ of $\omu$ in $\cM(X)^k$ such that the following hold:
\begin{enumerate}
\item $\frac{1}{k}\sum_{j\in [k]}t_j\ge r$ for every $(t_1, \dots, t_k)\in T$, and
\item for any independence set $F\in \cF(\Gamma)$ of $(U_1, \dots, U_k)$ with $|F|\ge N$, there are some $J\subseteq F$ with $|J|\ge c|F|$ and some $\ot\in T$ so that $J$ is an independence set for $\oA_{\of, \ot}$.
    \end{enumerate}
\end{lemma}
\begin{proof} Let $c,N>0$ and $T\subseteq \Rb^k$ be given by Lemma~\ref{L-func to indep} for $k, r, \frac{R+r}{2}, C$ and $p=\infty$.
Then $\frac{1}{k}\sum_{j\in [k]}t_j\ge r$ for every $(t_1, \dots, t_k)\in T$.
By Lemma~\ref{L-func to func} we can find a product neighborhood $U_1\times \cdots \times U_k$ of $\omu$ in $\cM(X)^k$ such that the following holds. For any independence set $F'\in \cF(\Gamma)$ of $\oU=(U_1, \dots, U_k)$ and any surjective map $\psi: F'\rightarrow [k]$,  there is
some $x\in X$ such that
\begin{align*} 
\frac{1}{k}\sum_{j\in [k]}\frac{1}{|\psi^{-1}(j)|}\sum_{s\in \psi^{-1}(j)}f_j(sx)\ge \frac{R+r}{2}.
\end{align*}

Let $F\in \cF(\Gamma)$ be an independence set of $\oU$ with $|F|\ge (k-1)+\max(N, k-1)$. We can find a set $F'\subseteq F$ such that $|F'|$ is divisible by $k$ and $|F\setminus F'|\le k-1$. Then $F'$ is also an independence set of $\oU$, $|F'|\ge N$ and $|F'|\ge |F|/2$. Let $\psi\in \sR(F', k)$. Then $\psi$ is surjective, whence we can find  an $x_\psi\in X$ such that
\begin{align}\label{E-func to density}
 \frac{1}{k}\sum_{j\in [k]}\frac{1}{|\psi^{-1}(j)|}\sum_{s\in \psi^{-1}(j)}f_j(sx_\psi)\ge \frac{R+r}{2}.
 \end{align}
We define a function $f_\psi: F'\rightarrow [-C, C]$ by
$$ f_\psi(s)=f_{\psi(s)}(sx_\psi)$$
for all $s\in F'$. Since $\psi\in \sR(F', k)$ and $|F'|$ is divisible by $k$, we have
$$|\psi^{-1}(j)|=\frac{1}{k}|F'|$$
for each $j\in [k]$.
Thus
 \begin{align*}
 \frac{1}{|F'|}\sum_{s\in F'}f_\psi(s)=\frac{1}{|F'|}\sum_{s\in F'}f_{\psi(s)}(sx_\psi)=\frac{1}{k}\sum_{j\in [k]}\frac{1}{|\psi^{-1}(j)|}\sum_{s\in \psi^{-1}(j)}f_j(sx_\psi)\overset{\eqref{E-func to density}}\ge \frac{R+r}{2}.
 \end{align*}
By our choice of $c, N$ and $T$ we can find some $J\subseteq F'$ with $|J|\ge c|F'|\ge \frac{c}{2}|F|$ and some $\ot=(t_1, \dots, t_k)\in T$  such that every map $\sigma: J\rightarrow [k]$ extends to a $\psi\in \sR(F', k)$ so that $f_\psi(s)\ge t_j$ for all $j\in [k]$ and $s\in \sigma^{-1}(j)$. 
Then
$$ f_{\sigma(s)}(sx_\psi)=f_{\psi(s)}(sx_\psi)=f_\psi(s)\ge t_{\sigma(s)}$$
for every $s\in J$, whence
$$ x_\psi\in \bigcap_{s\in J}s^{-1} f_{\sigma(s)}^{-1}([t_{\sigma(s)}, \infty)).$$
Thus $J$ is an independence set for $\oA_{\of, \ot}$.
\end{proof}

We are ready to prove Theorem~\ref{T-IE main}.

\begin{proof}[Proof of Theorem~\ref{T-IE main}]
In view of Lemma~\ref{L-Hahn Banach IE}, it suffices to show that given any $\omu=(\mu_1, \dots, \mu_k)\in \IE_k(\cM(X))$ and $\of=(f_1, \dots, f_k)\in C(X)^{\oplus k}$ satisfying $\sum_{j\in [k]}\mu_j(f_j)>0$, there is an $(x_1, \dots, x_k)\in \IE_k(X)$ such that $\sum_{j\in [k]}f_j(x_j)>0$.

Let $C=\|\of\|>0$, $R=\frac{1}{k}\sum_{j\in [k]} \mu_j(f_j)>0$ and $r=R/2$. Then we have $c,N>0$, $T\subseteq \Rb^k$ and $U_1\times \cdots \times U_k$ given by Lemma~\ref{L-func to density}. Since $\omu\in \IE_k(\cM(X))$, the tuple $\oU=(U_1, \dots, U_k)$ has positive independence density $\oI(\oU)$ (see Section~\ref{SS-IE}).

We claim that $\oI(\oA_{\of, \ot})\ge c\oI(\oU)$ for some $\ot\in T$. Assume that
$\oI(\oA_{\of, \ot})< c\oI(\oU)$ for every $\ot\in T$ instead.  Then for each $\ot\in T$ we can find some $K_\ot\in \cF(\Gamma)$ and some $\varepsilon_\ot>0$ such that
\begin{align} \label{E-hull for IE}
\varphi_{\oA_{\of, \ot}}(F)< c\oI(\oU)|F|
\end{align}
for every $F\in \cF(\Gamma)$ satisfying $|K_\ot F\Delta F|<\varepsilon_\ot |F|$. We can find an $F\in \cF(\Gamma)$ such that $|F|\ge N/\oI(\oU)$ and $|K_\ot F\Delta F|<\varepsilon_\ot |F|$ for all $\ot\in T$. Then \eqref{E-hull for IE}  holds for every $\ot\in T$. From the definition of $\varphi_{\oU}(F)$ we can find a set $F'\subseteq F$ such that $F'$ is an independence set for $\oU$ and $|F'|\ge \varphi_{\oU}(F)$. Then
$$ |F'|\ge \varphi_{\oU}(F)\ge \oI(\oU)|F|\ge N.$$
By our choice of $c, N, T$ and $U_1\times \cdots \times U_k$, there are some $J\subseteq F'$ with $|J|\ge c|F'|$ and some $\ot\in T$ such that $J$ is an independence set for $\oA_{\of, \ot}$. Then
$$ \varphi_{\oA_{\of, \ot}}(F)\ge |J|\ge c|F'|\ge c\oI(\oU)|F|,$$
a contradiction. This proves our claim.

Let $\ot=(t_1, \dots, t_k)\in T$ such that $\oI(\oA_{\of, \ot})\ge c\oI(\oU)>0$.
By Theorem~\ref{T-IE basic}.\eqref{i-IE tuple} there is some $(x_1, \dots, x_k)\in \IE_k(X)$ such that $x_j\in f_j^{-1}([t_j, \infty])$ for all $j\in [k]$. Then
$$ \sum_{j\in [k]}f_j(x_j)\ge \sum_{j\in [k]}t_j\ge kr>0$$
as desired.
\end{proof}

%%%%%%%%%%%%%%%%%%%%%%%%%%%%%%%%%%%%%%%%%%%%%%%%%%%%%%%%%%%%%%%%%%%%%%%%%%%%%%%%%%%%%
\subsection{Consequences of Theorem~\ref{T-IE main}} \label{SS-IE consequence}

As mentioned in Section~\ref{S-introduction}, 
in view of Theorem~\ref{T-IE basic}.\eqref{i-IE pair}, from the case $k=2$ of Theorem~\ref{T-IE main} we immediately obtain the following result of Glasner and Weiss \cite[Theorem A]{GW95}.

\begin{corollary} \label{C-IE for prob}
If $h_{\rm top}(X)=0$, then $h_{\rm top}(\cM(X))=0$.
\end{corollary}

\begin{lemma} \label{L-finite support}
Let $k\in \Nb$. If $\omu=(\mu_1, \dots, \mu_k)\in \IE_k(\cM(X))$ and $\supp(\mu_j)$ is finite for every $j\in [k]$, then $\omu$ is a convex combination of elements in $(\supp(\mu_1)\times \cdots \times \supp(\mu_k))\cap \IE_k(X)$.
\end{lemma}
\begin{proof} By Theorem~\ref{T-IE main}.(3) there exists a $\mu\in \cM(\IE_k(X))$ such that 
$\mu_j=\mu^{(j)}$ 
for all $j\in [k]$. Then $\mu(\supp(\mu_1)\times \cdots \times \supp(\mu_k))=1$. Thus $\supp(\mu)$ is finite and
$$\supp(\mu)\subseteq (\supp(\mu_1)\times \cdots \times \supp(\mu_k))\cap \IE_k(X).$$
Then $\omu$ is a convex combination of elements in $\supp(\mu)$, whence a convex combination of elements in $(\supp(\mu_1)\times \cdots \times \supp(\mu_k))\cap \IE_k(X)$.
\end{proof}

\begin{lemma} \label{L-support to support}
Let $k\in \Nb$ and let $\mu\in \cM(\IE_k(X))$ with finite support. Then
$$\sum_{\onu\in \supp(\mu)}\mu(\{\onu\})\onu\in \IE_k(\cM_{|\supp(\mu)|}(X)).$$
\end{lemma}
\begin{proof} The case $|\supp(\mu)|=1$ is trivial. The general case follows by induction on $|\supp(\mu)|$ using Lemma~\ref{L-IE convex}.(1).
\end{proof}

Now Corollary~\ref{C-IE finite support} follows from Lemmas~\ref{L-finite support} and \ref{L-support to support}.

We remark that even when $N_1, \dots, N_k$ are all equal, one cannot strengthen  Corollary~\ref{C-IE finite support} to claim
$$ \cM_N(X)^k\cap \IE_k(\cM(X))\subseteq \IE_k(\cM_N(X)).$$
To illustrate this, we give an example in Example~\ref{E-not IE}. For this we need the following lemma.

\begin{lemma} \label{L-not IE}
Let $x_1, x_2, y_1, y_2\in X$ such that $(x_1, x_2), (y_1, x_2), (x_1, y_2)\in \IE_2(X)$ but $(y_1, y_2)\not\in \IE_2(X)$. Let $1/2<\lambda<1$, and set
$$\omu=\lambda(\delta_{x_1}, \delta_{x_2})+(1-\lambda)(\delta_{y_1}, \delta_{y_2})\in \cM_2(X)^2.$$
Then
$\omu\in \IE_2(\cM_3(X))$ but $\omu\not\in \IE_2(\cM_2(X))$.
\end{lemma}
\begin{proof} Note that $\omu$ is a convex combination of $(\delta_{x_1}, \delta_{x_2})$, $(\delta_{y_1}, \delta_{x_2})$ and $(\delta_{x_1}, \delta_{y_2})$:
$$ \omu=(2\lambda-1)(\delta_{x_1}, \delta_{x_2})+(1-\lambda)(\delta_{y_1}, \delta_{x_2})+(1-\lambda)(\delta_{x_1}, \delta_{y_2}).$$
Thus by Lemma~\ref{L-support to support} we have $\omu\in \IE_2(\cM_3(X))$.

Since $(y_1, x_2)\in \IE_2(X)$ and $(y_1, y_2)\not\in \IE_2(X)$, we have $x_2\neq y_2$. Similarly, $x_1\neq y_1$.  Let $V_1$ and $W_1$ be open neighborhoods of $x_1$ and $y_1$ in $X$ respectively such that $V_1\cap W_1=\emptyset$. Also let $V_2$ and $W_2$ be open neighborhoods of $x_2$ and $y_2$ in $X$ respectively such that $V_2\cap W_2=\emptyset$. Let $0<\varepsilon<\min(\lambda-1/2, 1-\lambda)$. Then
$$U_1=\{\mu\in \cM(X): \mu(V_1)>\lambda-\varepsilon, \mu(W_1)>1-\lambda-\varepsilon\}$$
and
$$U_2=\{\mu\in \cM(X): \mu(V_2)>\lambda-\varepsilon, \mu(W_2)>1-\lambda-\varepsilon\}$$
are open neighborhoods of $\lambda\delta_{x_1}+(1-\lambda)\delta_{y_1}$ and $\lambda\delta_{x_2}+(1-\lambda)\delta_{y_2}$ in $\cM(X)$ respectively.

Let $F\in \cF(\Gamma)$ be an independence set of the pair $(U_1\cap \cM_2(X), U_2\cap \cM_2(X))$. For each map $\sigma: F\rightarrow \{1, 2\}$, we can find a $\nu_\sigma\in \cM_2(X)\cap \bigcap_{s\in F}s^{-1}U_{\sigma(s)}$. Say, $\nu_\sigma=t_\sigma \delta_{z_\sigma}+(1-t_\sigma)\delta_{w_\sigma}$ for some $z_\sigma, w_\sigma\in X$ and $t_\sigma\in [1/2, 1]$. For each $s\in F$, we have
$$t_\sigma \delta_{sz_\sigma}+(1-t_\sigma)\delta_{sw_\sigma}=s\nu_\sigma\in U_{\sigma(s)}.$$
Note that $\lambda-\varepsilon>1/2$ and $1-\lambda-\varepsilon>0$. 
If $\sigma(s)=1$, then $sz_\sigma\in V_1$ and $sw_\sigma\in W_1$. If $\sigma(s)=2$, then $sz_\sigma\in V_2$ and $sw_\sigma\in W_2$. Thus $w_\sigma\in \bigcap_{s\in F}s^{-1}W_{\sigma(s)}$.  This shows that $F$ is an independence set of the pair $(W_1, W_2)$.
Thus the independence density of the pair $(W_1, W_2)$ is no less than that of the pair $(U_1\cap \cM_2(X), U_2\cap \cM_2(X))$.

Since $(y_1, y_2)\not\in \IE_2(X)$, if we take $W_1$ and $W_2$ small enough, then the pair $(W_1, W_2)$ has zero independence density, and hence the pair $(U_1\cap \cM_2(X), U_2\cap \cM_2(X))$ has zero independence density. Therefore $\omu\not\in \IE_2(\cM_2(X))$.
\end{proof}

\begin{example} \label{E-not IE}
We consider the right shift action of $\Gamma$ on $\{0, 1,2 \}^\Gamma$:
$$ (sx)_\gamma=x_{\gamma s}$$
for all $x\in \{0, 1, 2\}^\Gamma$ and $s, \gamma\in \Gamma$. 
Let 
$$X=\{0, 1\}^\Gamma\cup \{1, 2\}^\Gamma\subseteq \{0, 1, 2\}^\Gamma.$$ 
Then $X$ is a closed $\Gamma$-invariant subset of $\{0, 1, 2\}^\Gamma$, thus we may consider the restriction of  the shift action to $X$. It is easily checked that
$$\IE_2(X)=(\{0, 1\}^\Gamma\times \{0, 1\}^\Gamma)\cup (\{1, 2\}^\Gamma\times \{1, 2\}^\Gamma).$$
Let $z$ be the constant function $\Gamma\rightarrow \{1\}$ in $X$.
We take $x_2=z$ and any $x_1\in \{1, 2\}^\Gamma$,  $y_1\in \{0, 1\}^\Gamma\setminus \{z\}$, and $y_2\in \{1, 2\}^\Gamma\setminus \{z\}$. Then $x_1, x_2, y_1, y_2$ satisfy the conditions in Lemma~\ref{L-not IE}. Therefore, for any $1/2<\lambda<1$, we have
$$ \lambda(\delta_{x_1}, \delta_{x_2})+(1-\lambda)(\delta_{y_1}, \delta_{y_2})\in \left(\cM_2(X)^2\cap \IE_2(\cM_3(X))\right)\setminus\IE_2(\cM_2(X)).$$
\end{example}

One may decompose the space $\cM_N(X)$ according to the weights at the points, as in the following notation.

\begin{notation} \label{N-finite support2}
Let $N\in \Nb$. We denote by $\orP_N$ the set of probability vectors of length $N$, i.e., 
$$ \orP_N=\{\olambda=(\lambda_1, \dots, \lambda_N)\in [0, 1]^N: \sum_{i\in [N]}\lambda_i=1\}.$$
For each $\olambda=(\lambda_1, \dots, \lambda_N)\in \orP_N$, we denote by $\cM_\olambda(X)$ the set of $\mu$ in $\cM(X)$ of the form $\sum_{i\in [N]}\lambda_i \delta_{x_i}$ for $(x_1, \dots, x_N)\in X^N$. This is a $\Gamma$-invariant closed subset of $\cM_N(X)$. We denote by $\pi_\olambda$ the factor map $X^N\rightarrow \cM_\olambda(X)$ sending $(x_1, \dots, x_N)$ to $\sum_{i\in [N]}\lambda_i \delta_{x_i}$.
\end{notation}

For $k\in \Nb$, we say $\Gamma\curvearrowright X$ has {\it  Uniform Positive Entropy (UPE) of order $k$} if $\IE_k(X)=X^k$. The notion of UPE of order $2$ was introduced by Blanchard in \cite[Definition 1]{Blanchard92} in terms of positivity for topological entropy of standard open covers, as an analogue of the K-property for topological dynamical systems. It was extended to the notion of UPE of order $k$ for $k\ge 2$ by Glasner and Weiss in \cite[page 306]{GW952} in terms of entropy tuples of length $k$. When $k\ge 2$, the definition we use here is equivalent to theirs, as one can see from the coincidence of entropy tuples  and non-diagonal IE-tuples \cite[Theorem 12.20]{KL16}. 

\begin{corollary} \label{C-UPE}
Let $k, N\in \Nb$ and $\olambda\in \orP_N$.
The following are equivalent:
\begin{enumerate}
\item $\Gamma\curvearrowright X$ has UPE of order $k$.
\item $\Gamma\curvearrowright \cM_\olambda(X)$ has UPE of order $k$.
\item $\Gamma\curvearrowright \cM(X)$ has UPE of order $k$.
\end{enumerate}
\end{corollary}
\begin{proof} (1)$\Longrightarrow$(3) follows from Theorem~\ref{T-IE main}.(1).

(1)$\Longrightarrow$(2). Assume that (1) holds. From Theorem~\ref{T-IE basic}.\eqref{i-IE product} we know that $\Gamma\curvearrowright X^N$ has UPE of order $k$. Since $\Gamma \curvearrowright \cM_\olambda(X)$ is a factor of $\Gamma\curvearrowright X^N$, from Theorem~\ref{T-IE basic}.\eqref{i-IE factor} we conclude that $\Gamma\curvearrowright \cM_\olambda(X)$ has UPE of order $k$.

Both (2)$\Longrightarrow$(1) and (3)$\Longrightarrow(1)$ follow from Theorem~\ref{T-IE main}.(2).
\end{proof}

The case $k=2$ and $\olambda=(\underbrace{1/N, \dots, 1/N}_{N})$ of Corollary~\ref{C-UPE} was established by Bernardes, Darji and Vermersch in \cite[Theorem 4]{BDV22} for $\Gamma=\Zb$ and by Liu and Wei in \cite[Theorem 1.1]{LW} for all amenable groups $\Gamma$.

%%%%%%%%%%%%%%%%%%%%%%%%%%%%%%%%%%%%%%%%%%%%%%%%%%%%%%%%%%%%%%%%%%%%%%%%%%%%%%%%%%%%%%%%%%%%%%%%%%%%%%%%%%%%%%%%%%%%%
\section{Measure $\IE$-tuples of $\cM(X)$} \label{S-measure IE for prob}

In this section we study measure IE-tuples and prove Theorem~\ref{T-measure IE main}. Throughout this section, we consider a continuous action of a countably infinite amenable group $\Gamma$ on a compact metrizable space $X$.

The following fact is well known. For convenience of the reader, we give a proof.

\begin{lemma} \label{L-barycenter}
Let $\rmm\in \cM(\cM(X))$ and let $\mu\in \cM(X)$ be the barycenter of $\rmm$. For each Borel subset $D$ of $X$, the function on $\cM(X)$ sending $\nu$ to $\nu(D)$ is Borel, and
\begin{align} \label{E-barycenter}
 \mu(D)=\int_{\cM(X)}\nu(D)\, d\rmm(\nu).
 \end{align}
\end{lemma}
\begin{proof} Consider first the case $D$ is open. We can find an increasing sequence $\{f_n\}_{n\in \Nb}$ in $C(X)$ such that $0\le f_n\le 1$ for every $n\in \Nb$ and $f_n\rightarrow {\bf 1}_D$ pointwisely as $n\to \infty$. Then
$$\nu(D)=\lim_{n\to \infty}\nu(f_n)$$
for every $\nu\in \cM(X)$. Thus the function on $\cM(X)$ sending $\nu$ to $\nu(D)$ is the pointwise limit of a sequence of continuous functions, and hence is Borel. Furthermore, by the bounded convergence theorem we have
$$ \mu(D)=\lim_{n\to \infty}\mu(f_n)=\lim_{n\to \infty}\int_{\cM(X)}\nu(f_n)\, d\rmm(\nu)=\int_{\cM(X)}\nu(D)\, d\rmm(\nu), $$
establishing \eqref{E-barycenter} for $D$.

We denote by $\sB$ the collection of all Borel subsets $D$ of $X$ such that the function on $\cM(X)$ sending $\nu$ to $\nu(D)$ is Borel and that \eqref{E-barycenter} holds for $D$. Then $\sB$ contains the collection of all open subsets of $X$, and is closed under taking complements and countable disjoint unions. For any collection $\sU$ of subsets of $X$ closed under taking finite intersections, the $\sigma$-algebra generated by $\sU$ is the smallest collection of subsets of $X$ which contains $\sU$ and is closed under taking complements and countable disjoint unions \cite[Theorem 10.1.(iii)]{Kechris}. Thus $\sB$ is the Borel $\sigma$-algebra of $X$.
\end{proof}

The following is an analogue of Lemma~\ref{L-func to func}. 

\begin{lemma} \label{L-measure func to func}
Let $\rmm\in \cM_\Gamma(\cM(X))$ and let $\mu\in \cM_\Gamma(X)$ be the barycenter of $\rmm$. Let $k\in \Nb$, $\onu=(\nu_1, \dots, \nu_k)\in \cM(X)^k$ and $\of=(f_1, \dots, f_k)\in C(X)^{\oplus k}$.  Let $R=\frac{1}{k}\sum_{j\in [k]}\nu_j(f_j)$ and let $r<R$.
Then there are some $0<\delta<1$ and some product neighborhood $U_1\times \cdots \times U_k$ of $\onu$ in $\cM(X)^k$ such that the following holds. For any $0<\eta\le 1$ and any $D\in \sB(\mu, \eta \delta)$ (see Section~\ref{SS-measure IE}) there is some $D'\in \sB(\rmm, \eta)$ so that for any independence set $F\in \cF(\Gamma)$ of $(U_1, \dots, U_k)$ relative to $D'$ and any surjective map $\psi: F\rightarrow [k]$,  there is
some $x\in D$ such that
\begin{align} \label{E-measure func to func}
\frac{1}{k}\sum_{j\in [k]}\frac{1}{|\psi^{-1}(j)|}\sum_{s\in \psi^{-1}(j)}f_j(sx)\ge r.
\end{align}
\end{lemma}
\begin{proof} Let $C>0$ with $C\ge \|\of\|$. Then $R\le C$. We take $0<\theta, \delta<1$ small such that
$$\frac{R-\theta-C\delta}{1-\delta}\ge r. $$
For each $j\in [k]$, we denote by $U_j$ the set of $\omega\in \cM(X)$ satisfying $\omega(f_j)\ge \nu_j(f_j)-\theta$. This is a neighborhood of $\nu_j$ in $\cM(X)$. We shall show that $\delta$ and $U_1\times \cdots \times U_k$ have the desired property. 

Let $0<\eta\le 1$
and $D\in \sB(\mu, \eta \delta)$. We denote by $D'$ the set of $\omega\in \cM(X)$ satisfying $\omega(D)\ge 1-\delta$. From Lemma~\ref{L-barycenter} we know that $D'$ is Borel and
\begin{align*}
\eta \delta &\ge \mu(X\setminus D)=\int_{\cM(X)}\omega(X\setminus D)\, d\rmm(\omega)\ge \int_{\cM(X)\setminus D'}\omega(X\setminus D)\, d\rmm(\omega)\\
&\ge \delta\cdot \rmm(\cM(X)\setminus D')=\delta(1-\rmm(D')),
\end{align*}
whence
$$\rmm(D')\ge 1-\eta.$$
That is, $D'\in \sB(\rmm, \eta)$.

Let $F\in \cF(\Gamma)$ be an independence set for $(U_1, \dots, U_k)$ relative to $D'$ and let $\psi: F\rightarrow [k]$ be surjective. Then $D'\cap \bigcap_{s\in F}s^{-1}U_{\psi(s)}$ is nonempty, whence we can find some $\omega\in D'\cap \bigcap_{s\in F}s^{-1}U_{\psi(s)}$.
We have
\begin{align*}
\int_X \left(\frac{1}{k}\sum_{j\in [k]}\frac{1}{|\psi^{-1}(j)|}\sum_{s\in \psi^{-1}(j)}f_j(sx)\right)\, d\omega(x)&=\frac{1}{k}\sum_{j\in [k]}\frac{1}{|\psi^{-1}(j)|}\sum_{s\in \psi^{-1}(j)}(s\omega)(f_j)\\
&\ge \frac{1}{k}\sum_{j\in [k]}\frac{1}{|\psi^{-1}(j)|}\sum_{s\in \psi^{-1}(j)}(\nu_j(f_j)-\theta)\\
&=\frac{1}{k}\sum_{j\in [k]}(\nu_j(f_j)-\theta)=R-\theta,
\end{align*}
and hence
\begin{align*}
\frac{1}{\omega(D)}&\int_D \left(\frac{1}{k}\sum_{j\in [k]}\frac{1}{|\psi^{-1}(j)|}\sum_{s\in \psi^{-1}(j)}f_j(sx)\right)\, d\omega(x)\\
&\ge \frac{1}{\omega(D)}\left(\int_X \left(\frac{1}{k}\sum_{j\in [k]}\frac{1}{|\psi^{-1}(j)|}\sum_{s\in \psi^{-1}(j)}f_j(sx)\right)\, d\omega(x)-C\omega(X\setminus D)\right)\\
&\ge \frac{R-\theta-C(1-\omega(D))}{\omega(D)}= \frac{R-\theta-C}{\omega(D)}+C\\
&\ge  \frac{R-\theta-C}{1-\delta}+C=\frac{R-\theta-C\delta}{1-\delta}\ge r,
\end{align*}
where in the 3rd inequality we use that $R-\theta-C<0$ and $\omega(D)\ge 1-\delta$.
Thus we can find some $x\in D$ such that \eqref{E-measure func to func} holds.
\end{proof}

The following lemma is an analogue of Lemma~\ref{L-Hahn Banach IE}, and can be proved in the same way.

\begin{lemma} \label{L-Hahn Banach measure IE}
Let $\rmm\in \cM_\Gamma(\cM(X))$ and let $\mu\in \cM_\Gamma(X)$ be the barycenter of $\rmm$.
If $ \IE^\rmm_k(\cM(X))\nsubseteq \overline{\rm co}(\IE^\mu_k(X))$ for some $k\in \Nb$, then there are some $(\nu_1, \dots, \nu_k)\in \IE^\rmm_k(\cM(X))$ and some $(f_1, \dots, f_k)\in C(X)^{\oplus k}$ such that $\sum_{j\in [k]}f_j(x_j)\le 0$
for all $(x_1, \dots, x_k)\in \IE^\mu_k(X)$ and
$\sum_{j\in [k]}\nu_j(f_j)>0$.
\end{lemma}

We are ready to prove Theorem~\ref{T-measure IE main}. 

\begin{proof}[Proof of Theorem~\ref{T-measure IE main}]
(1). In view of Lemma~\ref{L-Hahn Banach measure IE}, it suffices to show that given any $\onu=(\nu_1, \dots, \nu_k)\in \IE^\rmm_k(\cM(X))$ and  $\of=(f_1, \dots, f_k)\in C(X)^{\oplus k}$ satisfying $\sum_{j\in [k]}\nu_j(f_j)>0$, there is an $(x_1, \dots, x_k)\in \IE^\mu_k(X)$ such that $\sum_{j\in [k]}f_j(x_j)>0$.

Let $C>0$ with $C\ge \|\of\|$. Let $R=\frac{1}{k}\sum_{j\in [k]} \nu_j(f_j)>0$ and $r=R/4$. Then we have $c, N>0$ and $T\subseteq \Rb^k$ given by Lemma~\ref{L-func to indep} for $k, r, \frac{R}{2}, C$ and $p=\infty$.
By Lemma~\ref{L-measure func to func} we can find a
$0<\delta<1$ and a product neighborhood $U_1\times \cdots \times U_k$ of $\onu$ in $\cM(X)^k$ such that the following holds. For any $0<\eta\le 1$ and any $D\in \sB(\mu, \eta \delta)$ there is some $D'\in \sB(\rmm, \eta)$ so that for any  independence set $F\in \cF(\Gamma)$ of $\oU=(U_1, \dots, U_k)$ relative to $D'$ and any surjective map $\psi: F\rightarrow [k]$,  there is
some $x\in D$ such that
$$\frac{1}{k}\sum_{j\in [k]}\frac{1}{|\psi^{-1}(j)|}\sum_{s\in \psi^{-1}(j)}f_j(sx)\ge \frac{R}{2}.$$

Since $\onu\in \IE^\rmm_k(\cM(X))$, we have $\upind_\rmm(\oU)>0$. Then we can find a $0<\eta\le 1$ such that $\upind_\rmm(\oU, \eta)>0$.

We claim that
\begin{align} \label{E-measure IE for prob}
 \max_{\ot\in T}\upind_\mu(\oA_{\of, \ot})\ge \frac{c}{3}\upind_\rmm(\oU, \eta),
 \end{align}
where $\oA_{\of, \ot}$ is given in Notation~\ref{N-tuple}.
Let $\varepsilon>0$ and $K\in \cF(\Gamma)$. We can find an $F'\in \cF(\Gamma)$ with $|KF'\Delta F'|<\varepsilon |F'|$ and $\frac{1}{|F'|}\varphi_{\oU, \rmm, \eta}(F')>\frac{1}{2}\upind_\rmm(\oU, \eta)$.
Let $D\in \sB(\mu, \eta \delta)$. By our choice of $\delta$ and $U_1\times \cdots \times U_k$, we can find some $D'\in \sB(\rmm, \eta)$ so that
for any independence set $F\in \cF(\Gamma)$ of $\oU$ relative to $D'$ and any surjective map $\psi: F\rightarrow [k]$,  there is
some $x\in D$ such that
$$\frac{1}{k}\sum_{j\in [k]}\frac{1}{|\psi^{-1}(j)|}\sum_{s\in \psi^{-1}(j)}f_j(sx)\ge \frac{R}{2}.$$
When $\varepsilon$ is small enough and $|K|$ is large enough, we have $\frac{1}{6}\upind_\rmm(\oU, \eta)|F'|\ge \max(k-1, N/2)$, whence we can find an independence set $F\subseteq F'$ for $\oU$ relative to $D'$ such that $|F|$ is divisible by $k$ and
$$|F|\ge \varphi_{\oU, \rmm, \eta}(F')-(k-1)>\frac{1}{2}\upind_\rmm(\oU, \eta)|F'|-(k-1)\ge \frac{1}{3}\upind_\rmm(\oU, \eta)|F'|\ge N.$$
For each $\psi\in \sR(F, k)$, since $|F|$ is divisible by $k$, we have $|\psi^{-1}(j)|=|F|/k$ for every $j\in [k]$.  In particular, $\psi$ is surjective. Thus we can find an $x_\psi\in D$ such that 
$$\frac{1}{k}\sum_{j\in [k]}\frac{1}{|\psi^{-1}(j)|}\sum_{s\in \psi^{-1}(j)}f_j(sx_\psi)\ge \frac{R}{2}.$$
We define a function $f_\psi: F\rightarrow [-C, C]$ by
$$ f_\psi(s)=f_{\psi(s)}(sx_\psi).$$
Then
$$ \frac{1}{|F|}\sum_{s\in F}f_\psi(s)=\frac{1}{|F|}\sum_{s\in F}f_{\psi(s)}(sx_\psi)=\frac{1}{k}\sum_{j\in [k]}\frac{1}{|\psi^{-1}(j)|}\sum_{s\in \psi^{-1}(j)}f_j(sx_\psi)\ge \frac{R}{2}.$$
By our choice of $c, N$ and $T$ we can find some $J\subseteq F$ with $|J|\ge c|F|$ and some $\ot=(t_1, \dots, t_k)\in T$  such that every map $\sigma: J\rightarrow [k]$ extends to a $\psi\in \sR(F, k)$ 
so that $f_\psi(s)\ge t_j$ for all $j\in [k]$ and $s\in \sigma^{-1}(j)$. 
Then
$$f_{\sigma(s)}(sx_\psi)=f_{\psi(s)}(sx_\psi)=f_\psi(s)\ge t_{\sigma(s)}$$
for every $s\in J$, whence
$$x_\psi\in D\cap \bigcap_{s\in J}s^{-1}f_{\sigma(s)}^{-1}([t_{\sigma(s)}, \infty)).$$
Thus $J$ is an independence set for  $\oA_{\of, \ot}$ relative to $D$.
Note that
$$|J|\ge c|F|\ge \frac{c}{3}\upind_\rmm(\oU, \eta)|F'|.$$
If for each $\ot\in T$ there is some $D_\ot\in \sB(\mu, \eta\delta/|T|)$ such that $\oA_{\of, \ot}$ has no independence set in $F'$ relative to $D_\ot$ with cardinality at least $  \frac{c}{3}\upind_\rmm(\oU, \eta)|F'|$, then setting $D=\bigcap_{\ot\in T}D_\ot\in \sB(\mu, \eta\delta)$, from the above we know that $\oA_{\of, \ot}$ for some $\ot\in T$ has an independence set $J$ in $F'$ relative to $D$ with cardinality at least $\frac{c}{3}\upind_\rmm(\oU, \eta)|F'|$, but then $J$ is also an independence set for $\oA_{\of, \ot}$ in $F'$ relative to $D_\ot$, a contradiction. This means that
$$\max_{\ot \in T}\varphi_{\oA_{\of, \ot}, \mu, \eta \delta/|T|}(F')\ge \frac{c}{3}\upind_\rmm(\oU, \eta)|F'|.$$
It follows that
$$ \max_{\ot\in T}\upind_\mu(\oA_{\of, \ot})\ge \max_{\ot \in T}\upind_\mu(\oA_{\of, \ot}, \eta \delta/|T|)\ge \frac{c}{3}\upind_\rmm(\oU, \eta),$$
establishing \eqref{E-measure IE for prob}.

Now we can take a $\ot=(t_1, \dots, t_k)\in T$ such that $\upind_\mu(\oA_{\of, \ot})\ge \frac{c}{3}\upind_\rmm(\oU, \eta)>0$. By Theorem~\ref{T-measure IE basic}.\eqref{i-measure IE tuple} there is an $(x_1, \dots, x_k)\in \IE^\mu_k(X)$ with $x_j\in f_j^{-1}\left([t_j, \infty)\right)$ for every $j\in [k]$.
Then
$$ \sum_{j\in [k]}f_j(x_j)\ge \sum_{j\in [k]}t_j\ge kr>0$$
as desired.

(2). Note that $\ext\, \cM(X)=X$, whence $\ext(\cM(X)^k)=(\ext\, \cM(X))^k=X^k$. By Theorem~\ref{T-measure IE basic}.\eqref{i-measure IE invariant} we know that $\IE^\mu_k(X)$ is closed in $X^k$. From part (1) and Lemma~\ref{L-ext} we have
\begin{align*}
\IE^\rmm_k(\cM(X))\cap X^k &\subseteq \overline{\rm co}(\IE^\mu_k(X))\cap X^k=\overline{\rm co}(\IE^\mu_k(X))\cap \ext(\cM(X)^k)\\
&\subseteq \ext\, \overline{\rm co}(\IE^\mu_k(X))\subseteq \IE^\mu_k(X). \tag*{\qedsymbol}
\end{align*}
     \renewcommand{\qedsymbol}{}
     \vspace{-\baselineskip}
\end{proof}

\begin{remark} \label{R-measure IE not convex}
Despite the product formula for measure IE-tuples in Theorem~\ref{T-measure IE basic}.\eqref{i-measure IE product}, $\IE_k^\rmm(\cM(X))$ may fail to be convex for $\rmm\in \cM_\Gamma(\cM(X))$. For example, let $\mu_1, \mu_2\in \cM_\Gamma(X)$ be distinct, and let $0<\lambda<1$. Then $\rmm=\lambda \delta_{\mu_1}+(1-\lambda)\delta_{\mu_2}$ is in $\cM_\Gamma(\cM(X))$. It is easily checked that
$$\IE^\rmm_2(\cM(X))=\{(\mu_1, \mu_1), (\mu_2, \mu_2)\}$$
is not convex.
\end{remark}

Next we give a few consequences of Theorem~\ref{T-measure IE main}. As mentioned in Section~\ref{S-introduction},
in view of Theorem~\ref{T-measure IE basic}.\eqref{i-measure IE pair}, from the case $k=2$ of Theorem~\ref{T-measure IE main}.(1) we immediately obtain the following result of Glasner, Thouvenot and Weiss \cite[Theorem 1.1]{GW95} \cite[Theorem 7]{GTW}.

\begin{corollary} \label{C-measure IE for prob}
Let $\rmm\in \cM_\Gamma(\cM(X))$ and let $\mu\in \cM_\Gamma(X)$ be the barycenter of $\rmm$. If $h_\mu(X)=0$, then $h_\rmm(\cM(X))=0$.
\end{corollary}

For $\mu\in \cM_\Gamma(X)$ and $k\in \Nb$, following Vermersch \cite{Vermersch},
we say that $\Gamma\curvearrowright X$ has {\it $\mu$-Uniform Positive Entropy ($\mu$-UPE) of order $k$} if $\IE^\mu_k(X)=X^k$. From Theorem~\ref{T-measure IE main}.(2) we have the following corollary.

\begin{corollary} \label{C-measure UPE}
Let $\rmm\in \cM_\Gamma(\cM(X))$ and let $\mu\in \cM_\Gamma(X)$ be the barycenter of $\rmm$. Let $k\in \Nb$. If $\Gamma\curvearrowright \cM(X)$ has $\rmm$-UPE of order $k$, then $\Gamma\curvearrowright X$ has $\mu$-UPE of order $k$.
\end{corollary}

The case $k=2$ of Corollary~\ref{C-measure UPE} was established by Vermersch for $\Gamma=\Zb$ and ergodic $\rmm$ in \cite[Theorem 6]{Vermersch}.

\begin{corollary} \label{C-average}
Let $N\in \Nb$ and $\olambda\in \orP_N$. Let $\mu\in \cM_\Gamma(X)$, and set $\rmm=(\pi_\olambda)_*\mu^N\in \cM_\Gamma(\cM_\olambda(X))$ (see Notation~\ref{N-finite support2}).
For any $k\in \Nb$, $\Gamma\curvearrowright X$ has $\mu$-UPE of order $k$ if and only if $\Gamma\curvearrowright \cM_\olambda(X)$ has $\rmm$-UPE of order $k$.
\end{corollary}
\begin{proof} Assume that $\Gamma\curvearrowright X$ has $\mu$-UPE of order $k$. From Theorem~\ref{T-measure IE basic}.\eqref{i-measure IE product} we know that $\Gamma\curvearrowright X^N$ has $\mu^N$-UPE of order $k$,
then from Theorem~\ref{T-measure IE basic}.\eqref{i-measure IE factor} we get that $\Gamma\curvearrowright \cM_\olambda(X)$ has $\rmm$-UPE of order $k$. This proves the
``only if'' part.

We may think of $\rmm$ as an element of $\cM_\Gamma(\cM(X))$. It is easily checked that $\IE^\rmm_k(\cM_\olambda(X))\subseteq \IE^\rmm_k(\cM(X))$. Say, $\olambda=(\lambda_1, \dots, \lambda_N)$. For any $f\in C(X)$, we have
$$ \int_{\cM(X)}\nu(f)\, d\rmm(\nu)=\int_{\cM_\olambda(X)}\nu(f)\, d\rmm(\nu)=\int_{X^N}\sum_{j\in [N]}\lambda_j f(x_j)\, d\mu^N(x_1, \dots, x_N)=\mu(f).$$
Thus the barycenter of $\rmm$ is $\mu$.

Assume that $\Gamma\curvearrowright \cM_\olambda(X)$ has $\rmm$-UPE of order $k$. By Theorem~\ref{T-measure IE main}.(2)  we have
$$ X^k=\cM_\olambda(X)^k\cap X^k=\IE^\rmm_k(\cM_\olambda(X))\cap X^k\subseteq \IE^\rmm_k(\cM(X))\cap X^k\subseteq \IE^\mu_k(X)\subseteq X^k,$$
whence $X^k=\IE^\mu_k(X)$. This proves the ``if'' part.
\end{proof}

The case $k=2$ and $\olambda=(\underbrace{1/N, \dots, 1/N}_{N})$ of Corollary~\ref{C-average} was established by Vermersch for $\Gamma=\Zb$ and ergodic $\mu$ in \cite[Theorem 3]{Vermersch}.

%%%%%%%%%%%%%%%%%%%%%%%%%%%%%%%%%%%%%%%%%%%%%%%%%%%%%%%%%%%%%%%%%%%%%%%%%%%%%%%%%%%%%%%%%%%%%%%%%%%%%%%%%%%%%%%%%%%%%
\section{$\IN$-tuples of $\cM(X)$} \label{S-IN for prob}

In this section we prove Theorem~\ref{T-IN main}. Throughout this section, we consider a continuous action of a countably infinite (not necessarily amenable) group $\Gamma$ on a compact metrizable space $X$.

As we shall see later in Section~\ref{S-example}, the set $\IN_k(\cM(X))$ may fail to be convex. In order to obtain an analogue of Theorem~\ref{T-IE main}.(3) for IN-tuples, we need to replace $\cM(\IN_k(X))$ by a suitable subset. 
The following definition provides such a subset, and it involves not only $\IN_k(X)$ but also $\IN_k(X^n)$ for all $n\in \Nb$. 

\begin{definition} \label{D-IN measure}
Let $k\in \Nb$. We say $\mu\in \cM(X^k)$ is an {\it $\IN$-measure} if for any $n\in \Nb$ and   any $(x_1^{(i)}, \dots, x_k^{(i)})\in \supp(\mu)$ for $i\in [n]$, one has
$$\left((x_1^{(1)}, \dots, x_1^{(n)}), (x_2^{(1)}, \dots, x_2^{(n)}), \dots, (x_k^{(1)}, \dots, x_k^{(n)})\right)\in \IN_k(X^n).$$
We denote by $\cM_\IN(X^k)$ the set of all $\IN$-measures in  $\cM(X^k)$.
\end{definition}

\begin{remark} \label{R-IN measure}
For each $n\in \Nb$, we have the $\Gamma$-equivariant embedding $X^n\rightarrow X^{n+1}$ sending $(x_1, \dots, x_n)$ to $(x_1, \dots, x_n, x_n)$ repeating the last coordinate. From Theorem~\ref{T-IN basic}.\eqref{i-IN subset} we see that for any $k, n\in \Nb$ and 
$$\left((x_1^{(1)}, \dots, x_1^{(n)}), (x_2^{(1)}, \dots, x_2^{(n)}), \dots, (x_k^{(1)}, \dots, x_k^{(n)})\right)\in \IN_k(X^n),$$
one has 
$$\left((x_1^{(1)}, \dots, x_1^{(n)}, x_1^{(n)}), (x_2^{(1)}, \dots, x_2^{(n)}, x_2^{(n)}), \dots, (x_k^{(1)}, \dots, x_k^{(n)}, x_k^{(n)})\right)\in \IN_k(X^{n+1}).$$
It follows that in Definition~\ref{D-IN measure} we obtain the same definition no matter we require the points $(x_1^{(i)}, \dots, x_k^{(i)})\in \supp(\mu)$ for $i\in [n]$ to be distinct or not. 
\end{remark}

\begin{proposition} \label{P-IN measure basic}
For each $k\in \Nb$, $\cM_\IN(X^k)$ is a $\Gamma$-invariant closed subset of $\cM(\IN_k(X))$.
\end{proposition}
\begin{proof} For each $\mu\in \cM_\IN(X^k)$, taking $n=1$ in Definition~\ref{D-IN measure} we see that $\supp(\mu)\subseteq \IN_k(X)$, whence $\mu$ is actually a measure on $\IN_k(X)$. 

From Theorem~\ref{T-IN basic}.\eqref{i-IN invariant} we know that  $\IN_k(X^n)$ is $\Gamma$-invariant for each $n\in \Nb$. It follows that $\cM_\IN(X^k)$ is $\Gamma$-invariant.

Let $\mu$ be in the closure of $\cM_\IN(X^k)$. Let $n\in \Nb$ and $\ox^{(i)}=(x_1^{(i)}, \dots, x_k^{(i)})\in \supp(\mu)$ for each $i\in  [n]$.
We have the natural homeomorphism $\Phi: (X^k)^n\rightarrow (X^n)^k$ sending
$$\left((z_1^{(1)}, \dots, z_k^{(1)}), (z_1^{(2)}, \dots, z_k^{(2)}), \dots, (z_1^{(n)}, \dots, z_k^{(n)})\right)$$
to
$$\left((z_1^{(1)}, \dots, z_1^{(n)}), (z_2^{(1)}, \dots, z_2^{(n)}), \dots, (z_k^{(1)}, \dots, z_k^{(n)})\right).$$
Let $V$ be a neighborhood of $\Phi(\ox^{(1)}, \dots, \ox^{(n)})$ in $(X^n)^k$.
Then we can find an open neighborhood $U_i$ of $\ox^{(i)}$ in $X^k$ for each $i\in [n]$ such that $\Phi(U_1\times \cdots \times U_n)\subseteq V$. Note that $\mu(U_i)>0$ for each $i\in [n]$. Since the function $\cM(X^k)\rightarrow \Rb$ sending $\nu$ to $\nu(U)$ is lower semicontinuous for each open subset $U$ of $X^k$, we can find a $\nu\in \cM_\IN(X^k)$ close enough to $\mu$ such that $\nu(U_i)>0$ for all $i\in [n]$. Then there exists $\oy^{(i)}=(y_1^{(i)}, \dots, y_k^{(i)})\in \supp(\nu)\cap U_i$ for each $i\in [n]$. Since $\nu\in \cM_\IN(X^k)$, we have
$$\left((y_1^{(1)}, \dots, y_1^{(n)}), (y_2^{(1)}, \dots, y_2^{(n)}), \dots, (y_k^{(1)}, \dots, y_k^{(n)})\right)\in \IN_k(X^n).$$
This point is $\Phi(\oy^{(1)}, \cdots, \oy^{(n)})\in \Phi(U_1\times \cdots \times U_n)\subseteq V$. Thus $V\cap \IN_k(X^n)\neq \emptyset$. From Theorem~\ref{T-IN basic}.(1) we know that $\IN_k(X^n)$ is a closed subset of $(X^n)^k$. We conclude that $\Phi(\ox^{(1)}, \dots, \ox^{(n)})\in \IN_k(X^n)$. Thus $\mu\in \cM_\IN(X^k)$. Therefore $\cM_\IN(X^k)$ is closed. 
\end{proof}

For $k, n\in \Nb$,
we set 
$$\cM_{\IN, n}(X^k)=\cM_\IN(X^k)\cap \cM_n(X^k)$$ 
to be the set of $\mu\in \cM_\IN(X^k)$ with $|\supp(\mu)|\le n$, and set (see Notation~\ref{N-finite support2})
$$\cA_{k, n}=\bigcup_{\olambda\in \orP_n}(\pi_\olambda\times \cdots \times \pi_\olambda)(\IN_k(X^n))\subseteq  \cM_n(X)^k.$$
In order to prove part (3) of Theorem~\ref{T-IN main}, we shall show that both $\{(\mu^{(1)}, \dots, \mu^{(k)}): \mu\in \cM_\IN(X^k)\}$ and $\IN_k(\cM(X))$ are equal to 
$\overline{\bigcup_{n\in \Nb}\cA_{k, n}}$ in Lemmas~\ref{L-IN closure vs push} and \ref{L-IN dense} respectively. 

\begin{lemma} \label{L-IN finite support}
For any $k, n\in \Nb$, we have
$$\cA_{k, n}=\{(\mu^{(1)}, \dots, \mu^{(k)})\in \cM(X)^k: \mu\in \cM_{\IN, n}(X^k)\}.$$
\end{lemma}
\begin{proof} From Remark~\ref{R-IN measure} we see that the set $\cM_{\IN, n}(X^k)$ consists of 
$$ \mu=\sum_{i\in [n]}\lambda_i \delta_{(x_1^{(i)}, \dots, x_k^{(i)})}$$
for $\olambda=(\lambda_1, \dots, \lambda_n)\in \orP_n$ and 
$$\ox=\left((x_1^{(1)}, \dots, x_1^{(n)}), (x_2^{(1)}, \dots, x_2^{(n)}), \dots, (x_k^{(1)}, \dots, x_k^{(n)})\right)\in \IN_k(X^n).$$
Then
$$(\mu^{(1)}, \dots, \mu^{(k)})=\left(\sum_{i\in [n]}\lambda_i \delta_{x_1^{(i)}}, \dots, \sum_{i\in [n]}\lambda_i \delta_{x_k^{(i)}}\right)$$
is exactly $(\pi_\olambda\times \cdots \times \pi_\olambda)(\ox)$.
\end{proof}

\begin{lemma} \label{L-IN convex}
We have $\cA_{k, n}\subseteq  \IN_k(\cM_n(X))$ for all $k, n\in \Nb$.
\end{lemma}
\begin{proof}
From parts \eqref{i-IN factor} and \eqref{i-IN subset} of Theorem~\ref{T-IN basic} we have
\begin{displaymath}
 \cA_{k, n}=\bigcup_{\olambda\in \orP_n}\IN_k(\cM_\olambda(X))\subseteq \IN_k(\cM_n(X)). \tag*{\qedsymbol}
\end{displaymath}
     \renewcommand{\qedsymbol}{}
     \vspace{-\baselineskip}
\end{proof}

\begin{lemma} \label{L-IN closure vs push}
Let $k\in \Nb$. Then
$$\overline{\bigcup_{n\in \Nb}\cA_{k, n}}=\{(\mu^{(1)}, \dots, \mu^{(k)})\in \cM(X)^k: \mu\in \cM_\IN(X^k)\}.$$
\end{lemma}
\begin{proof} The map $\pi: \cM(X^k)\rightarrow \cM(X)^k$ sending $\nu$ to $(\nu^{(1)}, \dots, \nu^{(k)})$ is continuous. 
Note that $\cM_\IN(X^k)':=\bigcup_{n\in \Nb}\cM_{\IN, n}(X^k)$ is dense in $\cM_\IN(X^k)$. 
We have
$$ \pi(\cM_\IN(X^k)')\subseteq \pi(\cM_\IN(X^k))=\pi(\overline{\cM_\IN(X^k)'})\subseteq \overline{\pi(\cM_\IN(X^k)')}.$$
By Proposition~\ref{P-IN measure basic} the space $\cM_\IN(X^k)$ is compact. Thus $\pi(\cM_\IN(X^k))$ is compact, and hence closed in $\cM(X)^k$.
Then from the first inclusion above we have 
$$\overline{\pi(\cM_\IN(X^k)')}\subseteq \pi(\cM_\IN(X^k)).$$
It follows that 
\begin{align*}
\pi(\cM_\IN(X^k))=\overline{\pi(\cM_\IN(X^k)')}=\overline{\bigcup_{n\in \Nb}\cA_{k, n}},
\end{align*}
where the last equality is from Lemma~\ref{L-IN finite support}. 
\end{proof}

Let $m\in \Nb$. For any $\ot=(t_1, \dots, t_m)\in \Rb^m$, we set
$$ \|\ot\|_\infty=\max_{j\in [m]}|t_j|,  \mbox{ and } \|\ot\|_1=\sum_{j\in [m]}|t_j|.$$
We also set
$$D_m=\{(t_1, \dots, t_m)\in \Rb^m: t_j\ge 0 \mbox{ for all } j\in [m]\}.$$
For $\boldsymbol{w}, \boldsymbol{t}\in \Rb^m$, we write $\left< \boldsymbol{w}, \boldsymbol{t}\right>$ for their inner product.

\begin{lemma} \label{L-convex intersection}
Let $m\in \Nb$ and $C>1$. There are an $\varepsilon\in (0, 1)$ and a finite set $\Lambda\subseteq \orP_m$ depending only on $m$ and $C$ such that the following holds. If we take a point $\ow^{(\olambda)}\in \Rb^m$ satisfying $\|\ow^{(\olambda)}\|_\infty\le C$ and $\left<\ow^{(\olambda)}, \olambda\right>\ge 1-\varepsilon$ for each $\olambda\in \Lambda$, then the convex hull of $\{\ow^{(\olambda)}: \olambda\in \Lambda\}$ intersects $\frac{1}{2}\ov+D_m$, where $\ov=(\underbrace{1, \dots, 1}_{m})\in \Rb^m$.
\end{lemma}
\begin{proof} Let $\varepsilon=1/4$. Since $\orP_m$ is compact, we can find a finite set $\Lambda\subseteq \orP_m$ such that for each $\olambda\in \orP_m$ there is some $\olambda'\in \Lambda$ so that $\|\olambda-\olambda'\|_1<1/(4C)$. We shall show that $\varepsilon$ and $\Lambda$ have the desired property.

Let $\ow^{(\olambda)}\in \Rb^m$ for each $\olambda\in \Lambda$ such that $\|\ow^{(\olambda)}\|_\infty\le C$ and $\left<\ow^{(\olambda)}, \olambda\right>\ge 1-\varepsilon$.
We denote by $\cK$ the convex hull of $\{\ow^{(\olambda)}: \olambda\in \Lambda\}$. We argue by contradiction.  Assume that $\cK\cap (\frac{1}{2}\ov+D_m)=\emptyset$.
Since $\Lambda$ is finite, $\cK$ is compact.
By the geometric consequence of the Hahn-Banach theorem \cite[Proposition IV.3.6.(b) and Theorem IV.3.9]{Conway}, there is some $\boldsymbol{\lambda}\in \Rb^m$ such that
\begin{align} \label{E-convex intersection}
\max_{\ot\in \cK}\left< \ot, \olambda\right>< \inf_{\ot\in \frac{1}{2}\ov+D_m}\left< \ot, \olambda\right>.
\end{align}
Then
$$\frac{1}{2}\left< \ov, \olambda\right>+\inf_{\ot\in D_m}\left< \ot, \olambda\right>=\inf_{\ot\in \frac{1}{2}\ov+D_m}\left< \ot, \olambda\right>$$
is finite, and hence $\olambda\in D_m$. Multiplying $\olambda$ by a scalar if necessary, we may assume that $\|\olambda\|_1=1$. Then $\olambda\in \orP_m$. By our choice of $\Lambda$ we can find a $\olambda'\in \Lambda$ such that $\|\olambda-\olambda'\|_1<1/(4C)$. We have
\begin{align*}
\max_{\ot\in \cK}\left< \ot, \olambda\right>&\ge
\left<\ow^{(\olambda')}, \olambda\right>\ge
\left<\ow^{(\olambda')}, \olambda'\right>-\|\ow^{(\olambda')}\|_\infty\cdot \|\olambda-\olambda'\|_1\\
&> \left<\ow^{(\olambda')}, \olambda'\right>-1/4\ge 1-\varepsilon-1/4=1/2,
\end{align*}
and
\begin{align*}
\inf_{\ot\in \frac{1}{2}\ov+D_m}\left< \ot, \olambda\right>
=\left< \frac{1}{2}\ov, \olambda\right>=1/2,
\end{align*}
a contradiction to \eqref{E-convex intersection}. Therefore $\cK\cap (\frac{1}{2}\ov+D_m)\neq \emptyset$.
\end{proof}

\begin{lemma} \label{L-IN simul}
Let $k, n\in \Nb$, $\omu=(\mu_1, \dots, \mu_k)\in \cM(X)^k$  and $\of^{(l)}=(f_1^{(l)}, \dots, f_k^{(l)})\in C(X)^{\oplus k}$ for $l\in [n]$ such that $R^{(l)}:=\frac{1}{k}\sum_{j\in [k]}\mu_j(f_j^{(l)})>0$ for each $l\in [n]$. 
%Let $C>0$ with $C\ge \max_{l\in [n]}\|\of^{(l)}\|$ and  
Let $0<r^{(l)}<R^{(l)}$ for each $l\in [n]$.  Then there are some finite set $T^{(l)}\subseteq \Rb^k$ for each $l\in [n]$,  some $c, N>0$, 
 and some product neighborhood $U_1\times \cdots \times U_k$ of $\omu$ in $\cM(X)^k$ such that the following hold:
 \begin{enumerate}
 \item $\frac{1}{k}\sum_{j\in [k]}t_j\ge r^{(l)}$ for every $l\in [n]$ and $(t_1, \dots, t_k)\in T^{(l)}$, and
 \item for any independence set $F\in \cF(\Gamma)$ of $(U_1, \dots, U_k)$ with $|F|\ge N$,
 there are  some $J\subseteq F$ with $|J|\ge c|F|$ and some $\ot^{(l)}=(t_1^{(l)}, \dots, t_k^{(l)})\in T^{(l)}$ for each $l\in [n]$ so that
 $J$ is an independence set for the $k$-tuple
 \begin{align} \label{E-IN simul}
 \Big(\prod_{l\in [n]} (f_1^{(l)})^{-1}([t_1^{(l)}, \infty)), \dots, \prod_{l\in [n]} (f_k^{(l)})^{-1}([t_k^{(l)}, \infty))\Big)
  \end{align}
  of subsets of $X^n$.
  \end{enumerate}
 \end{lemma}
\begin{proof} We take a $C\ge \max_{l\in [n]}\|\of^{(l)}\|$. For each $l\in [n]$, we have $c^{(l)}, N^{(l)}$, $T^{(l)}$ and $U_1^{(l)}\times \cdots \times U_k^{(l)}$ given by Lemma~\ref{L-func to density} for $k, r^{(l)}, R^{(l)}, C, \omu$ and $\of^{(l)}$. We set $c=\prod_{l\in [n]}c^{(l)}>0$. We take a large $N>0$ such that $Nc^{(1)}\cdots c^{(l-1)}\ge N^{(l)}$ for all $l\in [n]$. We also set $U_j=\bigcap_{l\in [n]}U_j^{(l)}$ for each $j\in [k]$. Then $\frac{1}{k}\sum_{j\in [k]}t_j\ge r^{(l)}$ for every $l\in [n]$ and $(t_1, \dots, t_k)\in T^{(l)}$, and $U_1\times \cdots \times U_k$ is a product neighborhood  of $\omu$ in $\cM(X)^k$.

Let $F\in \cF(\Gamma)$ be  an  independence set of $(U_1, \dots, U_k)$ with $|F|\ge N$. We set $J_0=F$.
Via induction on $l=1, \dots, n$, from our choice of $c^{(l)}, N^{(l)}$, $T^{(l)}$ and $U_1^{(l)}\times \cdots \times U_k^{(l)}$ we can find some $J_l\subseteq J_{l-1}$ and some $\ot^{(l)}\in T^{(l)}$ such that $|J_l|\ge c^{(l)}|J_{l-1}|$ and $J_l$ is an independence set for $\oA_{\of^{(l)}, \ot^{(l)}}$ (see Notation~\ref{N-tuple}). Then $|J_n|\ge c|F|$ and $J_n$ is an independence set for $\oA_{\of^{(l)}, \ot^{(l)}}$ for every $l\in [n]$. It follows that $J_n$ is an independence set for
the $k$-tuple in \eqref{E-IN simul}.
\end{proof}

\begin{lemma} \label{L-IN dense}
Let $k\in \Nb$. Then $\IN_k(\cM(X))=\overline{\bigcup_{n\in \Nb}\cA_{k, n}}$.
\end{lemma}
\begin{proof} By Theorem~\ref{T-IN basic}.\eqref{i-IN subset} we have $\IN_k(\cM(X))\supseteq \IN_k(\cM_n(X))$ for every $n\in \Nb$. 
Thus from Lemma~\ref{L-IN convex} we have $\IN_k(\cM(X))\supseteq \bigcup_{n\in \Nb}\cA_{k, n}$. From Theorem~\ref{T-IN basic}.\eqref{i-IN invariant} we know that $\IN_k(\cM(X))$ is a closed subset of $\cM(X)^k$. Therefore $\IN_k(\cM(X))\supseteq \overline{\bigcup_{n\in \Nb}\cA_{k, n}}$.

Let $\omu=(\mu_1, \dots, \mu_k)\in \IN_k(\cM(X))$. We shall show that $\omu\in \overline{\bigcup_{n\in \Nb}\cA_{k, n}}$.

Let $U$ be a neighborhood of $\omu$ in $\cM(X)^k$. It suffices to show $U\cap \bigcup_{n\in \Nb}\cA_{k, n}\neq \emptyset$. Shrinking $U$ if necessary, we may assume that there are $m\in \Nb$ and
$$\of^{(i)}=(f^{(i)}_1, \dots, f^{(i)}_k)\in C(X)^{\oplus k}$$
for $i\in [m]$ such that
$$\sum_{j\in [k]}\mu_j(f^{(i)}_j)=1$$
for each $i\in [m]$ and
$$ U=\{(\nu_1, \dots, \nu_k)\in \cM(X)^k: \sum_{j\in [k]}\nu_j(f^{(i)}_j)\ge 0 \mbox{ for all } i\in  [m]\}.$$

We set $C=2+k\max_{i\in  [m]}\|\of^{(i)}\|>1$.  Then we have $\varepsilon>0$ and a finite set $\Lambda\subseteq \orP_m$ given by Lemma~\ref{L-convex intersection}.  For each $\olambda=(\lambda_1, \dots, \lambda_m)\in \Lambda$, we set
$$f^{(\olambda)}_j=\sum_{i\in [m]} \lambda_i f^{(i)}_j\in C(X)$$
for $j\in [k]$, and
$$\of^{(\olambda)}=\sum_{i\in [m]} \lambda_i\of^{(i)}=\left(\sum_{i\in [m]} \lambda_i f^{(i)}_1, \dots, \sum_{i\in [m]} \lambda_i f^{(i)}_k\right)=(f^{(\olambda)}_1, \dots, f^{(\olambda)}_k)\in C(X)^{\oplus k}.$$
Then 
$$\|\of^{(\olambda)}\|=\max_{j\in [k]}\|f^{(\olambda)}_j\|\le \max_{j\in [k]}\max_{i\in [m]}\|f^{(i)}_j\|=\max_{i\in [m]}\|\of^{(i)}\|\le C,$$ 
and
$$\sum_{j\in [k]}\mu_j(f^{(\olambda)}_j)=\sum_{j\in [k]}\sum_{i\in [m]} \lambda_i\mu_j(f^{(i)}_j)=\sum_{i\in [m]}\sum_{j\in [k]} \lambda_i\mu_j(f^{(i)}_j)=\sum_{i\in [m]}\lambda_i=1$$
for every $\olambda\in \Lambda$.

We set $R=1/k$ and $r=(1-\varepsilon)/k$. By Lemma~\ref{L-IN simul}  there are some finite set $T^{(\olambda)}\subseteq \Rb^k$ for each $\olambda\in \Lambda$, some $c, N>0$,
 and some product neighborhood $U_1\times \cdots \times U_k$ of $\omu$ in $\cM(X)^k$ such that the following hold:
  \begin{enumerate}
  \item $\frac{1}{k}\sum_{j\in [k]}t_j^{(\olambda)}\ge r$ for all $\olambda\in \Lambda$ and $(t_1^{(\olambda)}, \dots, t_k^{(\olambda)})\in T^{(\olambda)}$, and
  \item for any independence set $F\in \cF(\Gamma)$ of $\oU=(U_1, \dots, U_k)$ with $|F|\ge N$,
 there are  some $J\subseteq F$ with $|J|\ge c|F|$ and some $\ot^{(\olambda)}=(t_1^{(\olambda)}, \dots, t_k^{(\olambda)})\in T^{(\olambda)}$ for each $\olambda\in \Lambda$ so that
 $J$ is an independence set for the $k$-tuple
 \begin{align} \label{E-IN dense}
 \left(\prod_{\olambda\in \Lambda} (f_1^{(\olambda)})^{-1}([t_1^{(\olambda)}, \infty)), \dots, \prod_{\olambda\in \Lambda} (f_k^{(\olambda)})^{-1}([t_k^{(\olambda)}, \infty))\right)
 \end{align}
 of closed subsets of $X^\Lambda$.
 \end{enumerate}

We claim that there is some $\ot^{(\olambda)}=(t_1^{(\olambda)}, \dots, t_k^{(\olambda)})\in T^{(\olambda)}$ for each $\olambda\in \Lambda$ so that the tuple \eqref{E-IN dense} has arbitrarily large finite independence sets. We argue by contradiction. Assume that the tuple \eqref{E-IN dense} does not have arbitrarily large finite independence sets
for any choice of $\ot^{(\olambda)}=(t_1^{(\olambda)}, \dots, t_k^{(\olambda)})\in T^{(\olambda)}$ for each $\olambda\in \Lambda$. Since $\Lambda$ is finite and $T^{(\olambda)}$ for each $\olambda\in \Lambda$ is finite, it follows that there is some $L\in \Nb$ such that for any choice of $\ot^{(\olambda)}=(t_1^{(\olambda)}, \dots, t_k^{(\olambda)})\in T^{(\olambda)}$ for each $\olambda\in \Lambda$, every independence set for the tuple \eqref{E-IN dense} has cardinality $<L$.
Since $\omu\in \IN_k(\cM(X))$, we can find an independence set $F\in \cF(\Gamma)$ for $\oU$ with $|F|\ge \max(N, L/c)$.
Then there are  some $J\subseteq F$ with $|J|\ge c|F|\ge L$ and some $\ot^{(\olambda)}=(t_1^{(\olambda)}, \dots, t_k^{(\olambda)})\in T^{(\olambda)}$ for each $\olambda\in \Lambda$ so that
 $J$ is an independence set for the tuple \eqref{E-IN dense}, a contradiction. This proves our claim.

 Let $\ot^{(\olambda)}=(t_1^{(\olambda)}, \dots, t_k^{(\olambda)})\in T^{(\olambda)}$ for each $\olambda\in \Lambda$ so that the tuple \eqref{E-IN dense} has arbitrarily large finite independence sets. By Theorem~\ref{T-IN basic}.\eqref{i-IN tuple} there is some
 $$ ((x^{(\olambda)}_1)_{\olambda\in \Lambda}, \dots, (x^{(\olambda)}_k)_{\olambda\in \Lambda})$$
in the intersection of $\big(\prod_{\olambda\in \Lambda} (f_1^{(\olambda)})^{-1}([t_1^{(\olambda)}, \infty))\big)\times \cdots \times \big(\prod_{\olambda\in \Lambda}(f_k^{(\olambda)})^{-1}([t_k^{(\olambda)}, \infty))\big)$ and
$\IN_k(X^\Lambda)$.

Let $\olambda=(\lambda_1, \dots, \lambda_m)\in \Lambda$. We have
$$(x^{(\olambda)}_1, \dots, x^{(\olambda)}_k)\in (f_1^{(\olambda)})^{-1}([t_1^{(\olambda)}, \infty))\times \cdots \times (f_k^{(\olambda)})^{-1}( [t_k^{(\olambda)}, \infty)). $$
Since  $\ot^{(\olambda)}=(t_1^{(\olambda)}, \dots, t_k^{(\olambda)})\in T^{(\olambda)}$ ,  we have
$$ \frac{1}{k}\sum_{j\in [k]} f^{(\olambda)}_j(x^{(\olambda)}_j)\ge \frac{1}{k}\sum_{j\in [k]} t_j^{(\olambda)}\ge r=\frac{1}{k}(1-\varepsilon). $$
 We set
$$ \ow^{(\olambda)}=\left(\sum_{j\in [k]} f^{(1)}_j(x^{(\olambda)}_j), \dots, \sum_{j\in [k]} f^{(m)}_j(x^{(\olambda)}_j)\right)\in \Rb^m.$$
Then  $\|\ow^{(\olambda)}\|_\infty\le k\max_{i\in [m]}\|\of^{(i)}\|\le C$, and
$$\left<\ow^{(\olambda)}, \olambda\right>=\sum_{i\in [m]} \lambda_i\sum_{j\in [k]} f^{(i)}_j(x^{(\olambda)}_j)=\sum_{j\in [k]}\sum_{i\in [m]} \lambda_i f^{(i)}_j(x^{(\olambda)}_j)=\sum_{j\in [k]} f^{(\olambda)}_j(x^{(\olambda)}_j)\ge 1-\varepsilon.$$

By our choice of $\varepsilon$ and $\Lambda$, the convex hull of $\{\ow^{(\olambda)}: \olambda\in \Lambda\}$ intersects $\frac{1}{2}\ov+D_m$, where $\ov=(\underbrace{1, \dots, 1}_{m})\in \Rb^m$. Thus we can find $t_{\olambda}\in [0, 1]$ for each $\olambda\in \Lambda$ such that $\sum_{\olambda\in \Lambda}t_{\olambda}=1$ and $\sum_{\olambda\in \Lambda}t_{\olambda}\ow^{(\olambda)}\in \frac{1}{2}\ov+D_m$. Now we have
$$ \onu:=\left(\sum_{\olambda\in \Lambda}t_{\olambda}\delta_{x_1^{(\olambda)}}, \dots, \sum_{\olambda\in \Lambda}t_{\olambda}\delta_{x_k^{(\olambda)}}\right)\in \cA_{k, |\Lambda|}.$$
For each $i\in [m]$,
\begin{align*}
\sum_{j\in [k]}\left(\sum_{\olambda\in \Lambda}t_{\olambda}\delta_{x_j^{(\olambda)}}\right)(f_j^{(i)})
= \sum_{\olambda\in \Lambda}\sum_{j\in [k]} t_{\olambda}f_j^{(i)}(x_j^{(\olambda)})
\end{align*}
is the $i$-th coordinate of $\sum_{\olambda\in \Lambda}t_{\olambda}\ow^{(\olambda)}$, and hence is $\ge 1/2$. Thus $\onu\in U\cap \cA_{k, |\Lambda|}$. This finishes the proof.
\end{proof}

We are ready to prove Theorem~\ref{T-IN main}.

\begin{proof}[Proof of Theorem~\ref{T-IN main}] 
For each $n\in \Nb$ and 
$$\left((x_1^{(1)}, \dots, x_1^{(n)}), (x_2^{(1)}, \dots, x_2^{(n)}), \dots, (x_k^{(1)}, \dots, x_k^{(n)})\right)\in \IN_k(X^n),$$
from Theorem~\ref{T-IN basic}.\eqref{i-IN factor} we know that 
$$ (x_1^{(i)}, \dots, x_k^{(i)})\in \IN_k(X)$$
for every $i\in [n]$. 
Thus $\cA_{k, n}$ is included in the convex hull of $\IN_k(X)$ in $\cM(X)^k$ for every $n\in \Nb$. Then part (1) follows from Lemma~\ref{L-IN dense}.

Part (3) follows from Lemmas~\ref{L-IN closure vs push} and \ref{L-IN dense}. 

From Theorem~\ref{T-IN basic}.\eqref{i-IN subset} we have $\IN_k(X)\subseteq \IN_k(\cM(X))\cap X^k$. From part (3) we have
$\IN_k(\cM(X))\cap X^k\subseteq \IN_k(X)$. This proves part (2). 
\end{proof}

Next we prove Corollary~\ref{C-IN finite support}. 

\begin{proof}[Proof of Corollary~\ref{C-IN finite support}] Let $\omu\in (\cM_{N_1}(X)\times \cdots \times \cM_{N_k}(X))\cap \IN_k(\cM(X))$. By Theorem~\ref{T-IN main}.(3) we have $\omu=(\mu^{(1)}, \dots, \mu^{(k)})$ for some $\mu\in \cM_\IN(X^k)$. 
%there exists a $\mu\in \cM_\IN(X^k)$ such that $\mu_j$ is the push-forward of $\mu$ under the $j$-th coordinate map $X^k\rightarrow %X$ for each $j\in [k]$. 
Then $$\mu\left(\supp(\mu^{(1)})\times \cdots \times \supp(\mu^{(k)})\right)=1,$$
and $|\supp(\mu^{(j)})|\le N_j$ for every $j\in [k]$. 
Thus $|\supp(\mu)|\le N$, and hence $\mu\in \cM_{\IN, N}(X^k)$. By Lemmas~\ref{L-IN finite support} and \ref{L-IN convex}
% and Theorem~\ref{T-IN basic}.\eqref{i-IN subset} 
we have $\omu\in \cA_{k, N}\subseteq \IN_k(\cM_N(X))$. 
\end{proof}

In view of Theorem~\ref{T-IN basic}.\eqref{i-IN pair}, from the case $k=2$  of parts (1) and (2) of Theorem~\ref{T-IN main} we obtain immediately the following result of
Kerr and Li \cite[Theorem 5.10]{KL05}.

\begin{corollary} \label{C-IN for prob}
$\Gamma\curvearrowright X$ is null if and only if $\Gamma\curvearrowright \cM(X)$ is null.
\end{corollary}

In \cite[Theorem 1]{QZ} Qiao and Zhou strengthened Corollary~\ref{C-IN for prob}. They showed that for any fixed sequence $\mathfrak{s}$ in $\Gamma$, $h_{\rm top}(X; \mathfrak{s})=0$ if and only if $h_{\rm top}(\cM(X); \mathfrak{s})=0$.

For $k\in \Nb$, we say $\Gamma\curvearrowright X$ is {\it uniformly nonnull of order $k$} if $\IN_k(X)=X^k$ \cite[page 894]{KL07}. In view of Theorem~\ref{T-IN basic}.\eqref{i-IN subset}, from Theorem~\ref{T-IN main}.(2) we have the following corollary.

\begin{corollary} \label{C-IN order k}
Let  $k, N\in \Nb$ and $\olambda\in \orP_N$. If $\Gamma\curvearrowright \cM_\olambda(X)$ or $\Gamma\curvearrowright \cM(X)$ is uniformly nonnull of order $k$, then so is $\Gamma\curvearrowright X$.
\end{corollary}

%%%%%%%%%%%%%%%%%%%%%%%%%%%%%%%%%%%%%%%%%%%%%%%%%%%%%%%%%%%%%%%%%%%%%%%%%%%%%%%%%%%%%%%%%%%%%%%%%%%%%%%%%%%%%%%%%%%%%
\section{Measure $\IN$-tuples of $\cM(X)$} \label{S-measure IN for prob}

In this section we prove Theorem~\ref{T-measure IN main}. Throughout this section we consider a continuous action of a countably infinite group $\Gamma$ on a compact metrizable space $X$.

We start with the following  generalization of  Lemma~\ref{L-measure func to func}.

\begin{lemma} \label{L-measure IN func to func}
Let $\rmm\in \cM_\Gamma(\cM(X))$ and let $\mu\in \cM_\Gamma(X)$ be the barycenter of $\rmm$. Let $k\in \Nb$, $\onu=(\nu_1, \dots, \nu_k)\in \cM(X)^k$ and $\of=(f_1, \dots, f_k)\in C(X)^{\oplus k}$.  Let $R=\frac{1}{k}\sum_{j\in [k]}\nu_j(f_j)$ and let $r<R$.
Let $C\in \Rb$.
Then there are some $0<\delta<1$ and some product neighborhood $U_1\times \cdots \times U_k$ of $\onu$ in $\cM(X)^k$ such that the following holds. For any $0<\eta\le 1$ and any $D\in \sB'(\mu, \eta \delta)$ (see Section~\ref{SS-measure IN}) there is some $D'\in \sB'(\rmm, \eta)$ so that for any independence set $F\in \cF(\Gamma)$ of $(U_1, \dots, U_k)$ relative to $D'$ and any surjective map $\psi: F\rightarrow [k]$,  there is
some $x\in X$ such that
\begin{align} \label{E-measure IN func to func}
\frac{1}{k}\sum_{j\in [k]}\frac{1}{|\psi^{-1}(j)|}\sum_{s\in \psi^{-1}(j)}\left(f_j(sx){\bf 1}_{D_s}(x)-(k-1)C{\bf 1}_{X\setminus D_s}(x)\right)\ge r.
\end{align}
\end{lemma}
\begin{proof} We take a $C_1>0$ such that $C_1\ge \|\of\|$. We also take $0<\theta, \delta<1$ small such that
$$R-\theta-(C_1+(k-1)|C|)\delta  \ge r. $$
For each $j\in [k]$, we denote by $U_j$ the set of $\omega\in \cM(X)$ satisfying $\omega(f_j)\ge \nu_j(f_j)-\theta$. This is a neighborhood of $\nu_j$ in $\cM(X)$. We shall show that $\delta$ and $U_1\times \cdots \times U_k$  have the desired property. 

Let $0<\eta\le 1$
and $D\in \sB'(\mu, \eta \delta)$. For each $s\in \Gamma$, we denote by $D'_s$ the set of $\omega\in \cM(X)$ satisfying $\omega(D_s)\ge 1-\delta$. By Lemma~\ref{L-barycenter} we know that $D_s'$ is Borel and
\begin{align*}
\eta \delta &\ge \mu(X\setminus D_s)=\int_{\cM(X)}\omega(X\setminus D_s)\, d\rmm(\omega)\ge \int_{\cM(X)\setminus D'_s}\omega(X\setminus D_s)\, d\rmm(\omega)\\
&\ge \delta\cdot \rmm(\cM(X)\setminus D'_s)=\delta(1-\rmm(D'_s)),
\end{align*}
whence
$$\rmm(D'_s)\ge 1-\eta.$$
That is, $D'\in \sB'(\rmm, \eta)$.

Let $F\in \cF(\Gamma)$ be an independence set for $(U_1, \dots, U_k)$ relative to $D'$ and let $\psi: F\rightarrow [k]$ be surjective. Then $\bigcap_{s\in F}(D'_s\cap s^{-1}U_{\psi(s)})$ is nonempty, whence 
we can find some $\omega\in \bigcap_{s\in F}(D'_s\cap s^{-1}U_{\psi(s)})$.
We have
\begin{align*}
\int_X \left(\frac{1}{k}\sum_{j\in [k]}\frac{1}{|\psi^{-1}(j)|}\sum_{s\in \psi^{-1}(j)}f_j(sx)\right)\, d\omega(x)&=\frac{1}{k}\sum_{j\in [k]}\frac{1}{|\psi^{-1}(j)|}\sum_{s\in \psi^{-1}(j)}(s\omega)(f_j)\\
&\ge \frac{1}{k}\sum_{j\in [k]}\frac{1}{|\psi^{-1}(j)|}\sum_{s\in \psi^{-1}(j)}(\nu_j(f_j)-\theta)\\
&=\frac{1}{k}\sum_{j\in [k]}(\nu_j(f_j)-\theta)=R-\theta,
\end{align*}
where in the inequality we use $s\omega\in U_{\psi(s)}$ for all $s\in F$, and hence
\begin{align*}
& \int_X \left(\frac{1}{k}\sum_{j\in [k]}\frac{1}{|\psi^{-1}(j)|}\sum_{s\in \psi^{-1}(j)}\left(f_j(sx){\bf 1}_{D_s}(x)-(k-1)C {\bf 1}_{X\setminus D_s}(x)\right)\right)\, d\omega(x)\\
&\quad \quad = \int_X \left(\frac{1}{k}\sum_{j\in [k]}\frac{1}{|\psi^{-1}(j)|}\sum_{s\in \psi^{-1}(j)}f_j(sx)\right)\, d\omega(x)\\
&\quad \quad \quad \quad -\int_X \left(\frac{1}{k}\sum_{j\in [k]}\frac{1}{|\psi^{-1}(j)|}\sum_{s\in \psi^{-1}(j)}(f_j(sx)+(k-1)C){\bf 1}_{X\setminus D_s}(x)\right)\, d\omega(x)\\
&\quad \quad\ge R-\theta-\frac{1}{k}\sum_{j\in [k]}\frac{1}{|\psi^{-1}(j)|}\sum_{s\in \psi^{-1}(j)}(C_1+(k-1)|C|)\omega(X\setminus D_s)\\
&\quad \quad\ge R-\theta-(C_1+(k-1)|C|)\delta \ge r,
\end{align*}
where in the second inequality we use $\omega\in D_s'$ for all $s\in F$.
Thus we can find some $x\in X$ such that \eqref{E-measure IN func to func} holds.
\end{proof}

The following lemma is an analogue of Lemmas~\ref{L-Hahn Banach IE} and \ref{L-Hahn Banach measure IE}, and can be proved in the same way.

\begin{lemma} \label{L-Hahn Banach measure IN}
Let $\rmm\in \cM_\Gamma(\cM(X))$ and let $\mu\in \cM_\Gamma(X)$ be the barycenter of $\rmm$.
If $ \IN^\rmm_k(\cM(X))\nsubseteq \overline{\rm co}(\IN^\mu_k(X))$ for some $k\in \Nb$, then there are some $(\nu_1, \dots, \nu_k)\in \IN^\rmm_k(\cM(X))$ and some $(f_1, \dots, f_k)\in C(X)^{\oplus k}$ such that $\sum_{j\in [k]}f_j(x_j)\le 0$
for all $(x_1, \dots, x_k)\in \IN^\mu_k(X)$ and
$\sum_{j\in [k]}\nu_j(f_j)>0$.
\end{lemma}

We are ready to prove Theorem~\ref{T-measure IN main}.

\begin{proof}[Proof of Theorem~\ref{T-measure IN main}]
(1). In view of Lemma~\ref{L-Hahn Banach measure IN}, it suffices to show that given any $\onu=(\nu_1, \dots, \nu_k)\in \IN^\rmm_k(\cM(X))$ and  $\of=(f_1, \dots, f_k)\in C(X)^{\oplus k}$ satisfying $\sum_{j\in [k]}\nu_j(f_j)>0$, there is an $(x_1, \dots, x_k)\in \IN^\mu_k(X)$ such that $\sum_{j\in [k]}f_j(x_j)>0$.

We take a  $C>0$ with $C\ge \|\of\|$. We set $R=\frac{1}{k}\sum_{j\in [k]} \nu_j(f_j)>0$ and $r=R/4$. Then we have $c, N>0$ and $T\subseteq \Rb^k$ given by Lemma~\ref{L-func to indep} for $k, r, \frac{R}{2}, kC$ and $p=\infty$.
By Lemma~\ref{L-measure IN func to func} we can find a
$0<\delta<1$ and a product neighborhood $U_1\times \cdots \times U_k$ of $\onu$ in $\cM(X)^k$ such that the following holds. For any $0<\eta\le 1$ and any $D\in \sB'(\mu, \eta \delta)$ there is some $D'\in \sB'(\rmm, \eta)$ so that for any independence set $F\in \cF(\Gamma)$ of $\oU=(U_1, \dots, U_k)$ relative to $D'$ and any surjective map $\psi: F\rightarrow [k]$,  there is
some $x\in X$ such that
$$\frac{1}{k}\sum_{j\in [k]}\frac{1}{|\psi^{-1}(j)|}\sum_{s\in \psi^{-1}(j)}\left(f_j(sx){\bf 1}_{D_s}(x)-(k-1)C{\bf 1}_{X\setminus D_s}(x)\right)\ge \frac{R}{2}.$$

Since $\onu\in \IN^\rmm_k(\cM(X))$, $\oU$ has positive sequential $\rmm$-independence density. Then there are $0< \eta, c_1<1$ such that for every $M>0$ there is an $F'\in \cF(\Gamma)$ with $|F'|\ge M$ so that for every $D'\in \sB'(\rmm, \eta)$ there is an independence set $F\subseteq F'$ for $\oU$ relative to $D'$ with $|F|\ge c_1|F'|$.

Let $D\in \sB'(\mu, \eta \delta)$ and let $F'\in \cF(\Gamma)$ with $|F'|\ge \frac{2\max(N, k)}{c_1}$ such that for every $D'\in \sB'(\rmm, \eta)$ there is an independence set $F\subseteq F'$ for $\oU$ relative to $D'$ with $|F|\ge c_1|F'|$. We claim that there are some $J\subseteq F'$ with $|J|\ge \frac{cc_1}{2}|F'|$ and  some $\ot\in T$ such that $J$ is an independence set for $\oA_{\of, \ot}$ (see Notation~\ref{N-tuple}) relative to $D$.
By our choice of $\delta$ and $U_1\times \cdots \times U_k$, we can find some $D'\in \sB'(\rmm, \eta)$ such that
for any independence set $F\in \cF(\Gamma)$ of $\oU$ relative to $D'$ and any surjective map $\psi: F\rightarrow [k]$,  there is
some $x\in X$ such that
$$\frac{1}{k}\sum_{j\in [k]}\frac{1}{|\psi^{-1}(j)|}\sum_{s\in \psi^{-1}(j)}\left(f_j(sx){\bf 1}_{D_s}(x)-(k-1)C{\bf 1}_{X\setminus D_s}(x)\right)\ge \frac{R}{2}.$$
By our assumption on $F'$ we can find an independence set $F\subseteq F'$ for $\oU$ relative to $D'$ such that $|F|$ is divisible by $k$ and
$$|F|\ge c_1|F'|-(k-1)\ge \frac{c_1}{2}|F'|\ge N.$$
For each $\psi\in \sR(F, k)$, since $|F|$ is divisible by $k$, we have $|\psi^{-1}(j)|=|F|/k$ for every $j\in [k]$. In particular, $\psi$ is surjective. Thus, 
by our choice of $D'$ we can find an $x_\psi\in X$ so that
\begin{align*}
 \frac{1}{k}\sum_{j\in [k]}\frac{1}{|\psi^{-1}(j)|}\sum_{s\in \psi^{-1}(j)}\left(f_j(sx_\psi){\bf 1}_{D_s}(x_\psi)-(k-1)C{\bf 1}_{X\setminus D_s}(x_\psi)\right)\ge \frac{R}{2}.
 \end{align*}
We define a function $f_\psi: F\rightarrow [-kC, kC]$ by
$$ f_\psi(s)=f_{\psi(s)}(sx_\psi){\bf 1}_{D_s}(x_\psi)-(k-1)C{\bf 1}_{X\setminus D_s}(x_\psi).$$
Then
\begin{align*}
 \frac{1}{|F|}\sum_{s\in F}f_\psi(s)&=\frac{1}{|F|}\sum_{s\in F}\left(f_{\psi(s)}(sx_\psi){\bf 1}_{D_s}(x_\psi)-(k-1)C{\bf 1}_{X\setminus D_s}(x_\psi)\right)\\
 &=\frac{1}{k}\sum_{j\in [k]}\frac{1}{|\psi^{-1}(j)|}\sum_{s\in \psi^{-1}(j)}\left(f_j(sx_\psi){\bf 1}_{D_s}(x_\psi)-(k-1)C{\bf 1}_{X\setminus D_s}(x_\psi)\right)\ge \frac{R}{2}.
 \end{align*}
By our choice of $c, N$ and $T$ we can find some $J\subseteq F$ with $|J|\ge c|F|$ and some $\ot=(t_1, \dots, t_k)\in T$  such that every map $\sigma: J\rightarrow [k]$ extends to a $\psi\in \sR(F, k)$ so that $f_\psi(s)\ge t_j$ for all $j\in [k]$ and $s\in \sigma^{-1}(j)$.
For each $j\in [k]$, taking a $\sigma: J\rightarrow [k]$ with $j\in \sigma(J)$, we have
$$t_j\le f_\psi(s)\le C$$
for all $s\in \sigma^{-1}(j)$. Let $\sigma$ and $\psi$ be as above.
If $x_\psi\in X\setminus D_\gamma$ for some $\gamma\in J$, then
\begin{align*}
r\le \frac{1}{k}\sum_{j\in [k]}t_j=\frac{1}{k}\left(t_{\sigma(\gamma)}+\sum_{j\in [k]\setminus \{\sigma(\gamma)\}}t_j\right)\le
\frac{1}{k}(f_\psi(\gamma)+(k-1)C)=0,
\end{align*}
which is impossible. Thus, $x_\psi\in \bigcap_{s\in J}D_s$. Then
$$ f_{\sigma(s)}(sx_\psi)=f_{\psi(s)}(sx_\psi)=f_\psi(s)\ge t_{\sigma(s)}$$
for all $s\in J$, and whence
$$x_\psi\in \bigcap_{s\in J}(D_s\cap s^{-1}f_{\sigma(s)}^{-1}([t_{\sigma(s)}, \infty))).$$
Therefore $J$ is an independence set for  $\oA_{\of, \ot}$ relative to $D$.
Note that
$$|J|\ge c|F|\ge \frac{cc_1}{2}|F'|.$$
This proves our claim.

Next we claim that there is some $\ot\in T$ such that $\oA_{\of, \ot}$ has positive sequential $\mu$-independence density.
We argue by contradiction. Assume that for each $\ot\in T$ the tuple $\oA_{\of, \ot}$ does not have positive sequential $\mu$-independence density.
Then there is some $M_\ot>0$ such that for any $F'\in \cF(\Gamma)$ with $|F'|\ge M_\ot$, there is some $D\in \sB'(\mu, \eta \delta/|T|)$ so that $\oA_{\of, \ot}$ has no independence sets $J\subseteq F'$ relative to $D$ with $|J|\ge \frac{cc_1}{2}|F'|$. 
We set
$$M=\max\left(\max_{\ot\in T}M_\ot, \frac{2\max(N, k)}{c_1}\right)>0.$$
By our choice of $\eta$ and $c_1$ we can find an $F'\in \cF(\Gamma)$ with $|F'|\ge M$ such that for every $D'\in \sB'(\rmm, \eta)$ there is an independence set $F\subseteq F'$ for $\oU$ relative to $D'$ with $|F|\ge c_1|F'|$. For each $\ot\in T$, since $|F'|\ge M\ge M_\ot$, there is some $D_\ot\in \sB'(\mu, \eta \delta/|T|)$ so that $\oA_{\of, \ot}$ has no independence sets $J\subseteq F'$ relative to $D_\ot$ with $|J|\ge \frac{cc_1}{2}|F'|$.
We set
$$D_s=\bigcap_{\ot\in T}(D_\ot)_s\in \sB(\mu, \eta\delta)$$
for each $s\in \Gamma$. Then $D\in \sB'(\mu, \eta\delta)$. From the above paragraph we can find some $J\subseteq F'$ with $|J|\ge \frac{cc_1}{2}|F'|$ and  some $\ot\in T$ such that $J$ is an independence set for $\oA_{\of, \ot}$ relative to $D$. Then $J$ is also an independence set for $\oA_{\of, \ot}$ relative to $D_\ot$, a contradiction to our choice of $D_\ot$. This proves our claim.

Now we can take a $\ot=(t_1, \dots, t_k)\in T$ such that $\oA_{\of, \ot}$ has positive sequential $\mu$-independence density.
By Theorem~\ref{T-measure IN basic}.\eqref{i-measure IN tuple} there is an $(x_1, \dots, x_k)\in \IN^\mu_k(X)$ with $x_j\in f_j^{-1}([t_j, \infty))$ for every $j\in [k]$.
Then
$$ \sum_{j\in [k]}f_j(x_j)\ge \sum_{j\in [k]}t_j\ge kr>0$$
as desired.

(2) is proved in the same way as Theorem~\ref{T-measure IE main}.(2), using part (1), Theorem~\ref{T-measure IN basic} and Lemma~\ref{L-ext}.
\end{proof}

\begin{remark} \label{R-measure IN not convex}
Despite the product formula for measure IN-tuples in Theorem~\ref{T-measure IN basic}.\eqref{i-measure IN product}, $\IN_k^\rmm(\cM(X))$ may fail to be convex for $\rmm\in \cM_\Gamma(\cM(X))$. Similar to the example in Remark~\ref{R-measure IE not convex}, for distinct $\mu_1, \mu_2\in \cM_\Gamma(X)$ and $0<\lambda<1$, if we set $\rmm=\lambda \delta_{\mu_1}+(1-\lambda)\delta_{\mu_2}\in \cM_\Gamma(\cM(X))$, then  it is easily checked that
$$\IN^\rmm_2(\cM(X))=\{(\mu_1, \mu_1), (\mu_2, \mu_2)\}$$
is not convex.
\end{remark}

%$\IN^\mu_2(X)\subseteq \Delta_2(X)$ 
The action $\Gamma\curvearrowright (X, \mu)$ is null if and only if 
%$\Gamma\curvearrowright (X, \mu)$ 
it has discrete spectrum, i.e., every $f\in L^2(X, \mu)$ has compact orbit closure \cite[Theorem 5.5]{KL09} \cite{Kushnirenko}. In view of Theorem~\ref{T-measure IN basic}.\eqref{i-IN pair}, from Theorem~\ref{T-measure IN main}.(1) we have the following consequence.

\begin{corollary}  \label{C-measure IN for prob}
Let $\rmm\in \cM_\Gamma(\cM(X))$ and let $\mu\in \cM_\Gamma(X)$ be the barycenter of $\rmm$. If $\Gamma\curvearrowright (X, \mu)$ has discrete spectrum, then so does $\Gamma\curvearrowright (\cM(X), \rmm)$.
\end{corollary}

From Theorem~\ref{T-measure IN main}.(2) we have the following corollary.

\begin{corollary} \label{C-measure uniform IN}
Let $\rmm\in \cM_\Gamma(\cM(X))$ and let $\mu\in \cM_\Gamma(X)$ be the barycenter of $\rmm$. Let $k\in \Nb$. If $\IN^\rmm_k(\cM(X))=\cM(X)^k$, then $\IN^\mu_k(X)=X^k$.
\end{corollary}

The following is an analogue of Corollary~\ref{C-average}, and can be proved in the same way, using Theorems~\ref{T-measure IN main} and 
\ref{T-measure IN basic}.

\begin{corollary} \label{C-average IN}
Let $N\in \Nb$ and $\olambda\in \orP_N$. Let $\mu\in \cM_\Gamma(X)$, and set $\rmm=(\pi_\olambda)_*\mu^N\in \cM_\Gamma(\cM_\olambda(X))$ (see Notation~\ref{N-finite support2}).
For any $k\in \Nb$, $\IN^\mu_k(X)=X^k$ if and only if $\IN^\rmm_k(\cM_\olambda(X))=\cM_\olambda(X)^k$.
\end{corollary}

%%%%%%%%%%%%%%%%%%%%%%%%%%%%%%%%%%%%%%%%%%%%%%%%%%%%%%%%%%%%%%%%%%%%%%%%%%%%%%%%%%%%%%%%%%%%%%%%%%%%%%%%%%%%%%%%%%%%%
\section{$\IT$-tuples of $\cM(X)$} \label{S-IT for prob}

In this section we prove Theorem~\ref{T-IT main}. Throughout this section we consider a continuous action of a countably infinite group $\Gamma$ on a compact metrizable space $X$.

Similar to the case of IN-tuples, $\IT_k(\cM(X))$ may fail to be convex (see Section~\ref{S-example} for an example). In order to obtain an explicit description of elements of $\IT_k(\cM(X))$ like that for $\IE_k(\cM(X))$ in Theorem~\ref{T-IE main}.(3), we need to find a suitable subset of $\cM(\IT_k(X))$.  
The following is the IT version of Definition~\ref{D-IN measure}.

\begin{definition} \label{D-IT measure}
Let $k\in \Nb$. We say $\mu\in \cM(X^k)$ is an {\it $\IT$-measure} if for any $n\in \Nb$ and   any $(x_1^{(i)}, \dots, x_k^{(i)})\in \supp(\mu)$ for $i\in [n]$, one has
$$\left((x_1^{(1)}, \dots, x_1^{(n)}), (x_2^{(1)}, \dots, x_2^{(n)}), \dots, (x_k^{(1)}, \dots, x_k^{(n)})\right)\in \IT_k(X^n).$$
We denote by $\cM_\IT(X^k)$ the set of all $\IT$-measures in  $\cM(X^k)$.
\end{definition}

\begin{remark} \label{R-IT measure}
Similar to Remark~\ref{R-IN measure}, using Theorem~\ref{T-IT basic}.\eqref{i-IT subset} one sees that in Definition~\ref{D-IT measure} we obtain the same definition no matter we require the points $(x_1^{(i)}, \dots, x_k^{(i)})\in \supp(\mu)$ for $i\in [n]$ to be distinct or not. 
\end{remark}

The following is an analogue of Proposition~\ref{P-IN measure basic}, and can be proved in the same way using Theorem~\ref{T-IT basic}.\eqref{i-IT invariant}.

\begin{proposition} \label{P-IT measure basic}
For each $k\in \Nb$, $\cM_\IT(X^k)$ is a $\Gamma$-invariant closed subset of $\cM(\IT_k(X))$.
\end{proposition}

For $k, n\in \Nb$, we set
$$\cM_{\IT, n}(X^k)=\cM_\IT(X^k)\cap \cM_n(X^k)$$
to be the set of $\mu\in \cM_\IN(X^k)$ with $|\supp(\mu)|\le n$, and  set (see Notation~\ref{N-finite support2})
$$\cB_{k, n}=\bigcup_{\olambda\in \orP_n}(\pi_\olambda\times \cdots \times \pi_\olambda)(\IT_k(X^n))\subseteq  \cM_n(X)^k.$$
In order to prove part (3) of Theorem~\ref{T-IT main}, as in Section~\ref{S-IN for prob}, we shall show that both $\{(\mu^{(1)}, \dots, \mu^{(k)}): \mu\in \cM_\IT(X^k)\}$
and $\IT_k(\cM(X))$ are equal to $\overline{\bigcup_{n\in \Nb}\cB_{k, n}}$ in Lemmas~\ref{L-IT closure vs push} and \ref{L-IT dense} respectively. 

The following is an analogue of Lemma~\ref{L-IN finite support}, and can be proved in the same way using 
Remark~\ref{R-IT measure}.

\begin{lemma} \label{L-IT finite support}
For any $k, n\in \Nb$, we have
$$\cB_{k, n}=\{(\mu^{(1)}, \dots, \mu^{(k)})\in \cM(X)^k: \mu\in \cM_{\IT, n}(X^k)\}.$$
\end{lemma}

The following is an analogue of Lemma~\ref{L-IN convex}, and can be proved in the same way using 
Theorem~\ref{T-IT basic}.

\begin{lemma} \label{L-IT convex}
We have $\cB_{k, n}\subseteq \IT_k(\cM_n(X))$ for all $k, n\in \Nb$.
\end{lemma}

The following is an analogue of Lemma~\ref{L-IN closure vs push}, and can be proved in the same way using Proposition~\ref{P-IT measure basic} and Lemma~\ref{L-IT finite support}. 

\begin{lemma} \label{L-IT closure vs push}
Let $k\in \Nb$. Then
$$\overline{\bigcup_{n\in \Nb}\cB_{k, n}}=\{(\mu^{(1)}, \dots, \mu^{(k)})\in \cM(X)^k: \mu\in \cM_\IT(X^k)\}.$$
\end{lemma}

To prove Lemma~\ref{L-IT dense}, we need Lemma~\ref{L-func to infinite}, which is an analogue of Lemma~\ref{L-func to density}. While Lemma~\ref{L-func to density} is based on Lemma~\ref{L-func to indep}, the combinatorial fact behind Lemma~\ref{L-func to infinite} is the following 
dichotomy of Rosenthal \cite[Theorem 8.4]{KL16}, which is the  $k$-tuple version  of   \cite[Theorem 2]{Rosenthal74}.
For $k\in \Nb$ and a set $Y$, we say that a family $\{(A_{i, 1}, \dots, A_{i, k}): i\in F\}$ of $k$-tuples of subsets of $Y$ is {\it independent} if 
$$\bigcap_{i\in J} A_{i, \sigma(i)}\neq \emptyset$$
for every nonempty finite set $J\subseteq F$ and every map $\sigma: J\rightarrow [k]$. 

\begin{lemma} \label{L-IT combinatorial}
Let $k\in \Nb$ and let $\{(A_{n, 1}, \dots, A_{n, k})\}_{n\in \Nb}$ be a sequence of $k$-tuples of subsets of a set $Y$. Then there is an infinite set $F\subseteq \Nb$ such that either $\{(A_{n, 1}, \dots, A_{n, k}): n\in F\}$ is independent  or for every $y\in Y$ there exists a $j\in [k]$ such that the set $\{n\in F: y\in A_{n, j}\}$ is finite.
\end{lemma}

We also need the following result of Bergelson \cite[Theorem 1.1]{Bergelson85}. In \cite[Theorem 1.1]{Bergelson85} it is assumed that $\nu(A_n)=\delta$ for all $n\in \Nb$, but the proof works when $\nu(A_n)\ge \delta$ for all $n\in \Nb$.

\begin{lemma} \label{L-finite intersection}
Let $(Y, \sB, \nu)$ be a probability measure space, $\delta>0$, and $A_n\in \sB$ with $\nu(A_n)\ge \delta$ for each $n\in \Nb$. Then there is an infinite set $F\subseteq \Nb$ such that $\nu\left(\bigcap_{n\in J}A_n\right)>0$ for every nonempty finite set $J\subseteq F$.
\end{lemma}

\begin{lemma} \label{L-func to infinite}
Let $k\in \Nb$, $\omu=(\mu_1, \dots, \mu_k)\in \cM(X)^k$ and $\of=(f_1, \dots, f_k)\in C(X)^{\oplus k}$ such that $R:=\frac{1}{k}\sum_{j\in [k]}\mu_j(f_j)>0$.  Let $0<r<R$ and $C\ge \|\of\|$. Then there are some  finite set $T\subseteq \Rb^k$ depending only on $k, r, R$ and $C$,  and some product neighborhood $U_1\times \cdots \times U_k$ of $\omu$ in $\cM(X)^k$ such that the following hold:
\begin{enumerate}
\item $\frac{1}{k}\sum_{j\in [k]}t_j\ge r$ for every $(t_1, \dots, t_k)\in T$, and
\item for any infinite independence set $F\subseteq \Gamma$ of $(U_1, \dots, U_k)$, there are some infinite set $J\subseteq F$ and some $\ot\in T$ so that $J$ is an independence set for $\oA_{\of, \ot}$ (see Notation~\ref{N-tuple}).
    \end{enumerate}
\end{lemma}
\begin{proof} We take an $M\in \Nb$ such that
$$\frac{C}{M}< \frac{R-r}{2}.$$
We denote by $I$ the set of $\boldsymbol{i}=(i_1, \dots, i_k)\in [2M]^k$ satisfying
$$\frac{1}{k}\sum_{j\in [k]}\left(-C+\frac{(i_j-1)C}{M}\right)\ge r.$$
When $i_j=2M$ for all $j\in [k]$, we have
$$ \frac{1}{k}\sum_{j\in [k]}\left(-C+\frac{(i_j-1)C}{M}\right)=-C+\frac{(2M-1)C}{M}=C-\frac{C}{M}>R-\frac{R-r}{2}>r.$$
Thus $(2M, \dots, 2M)\in I$, whence $I$ is nonempty. We denote by $T$ the set of
$$\left(-C+\frac{(i_1-1)C}{M}, \dots, -C+\frac{(i_k-1)C}{M}\right)\in [-C, C]^k$$
for $\boldsymbol{i}=(i_1, \dots, i_k)\in I$. Then $\frac{1}{k}\sum_{j\in [k]}t_j\ge r$ for every $(t_1, \dots, t_k)\in T$.

For each $j\in [k]$, we denote by $U_j$ the set of $\nu\in \cM(X)$ satisfying
$$ \nu(f_j)\ge \mu_j(f_j)-(R-r)/4.$$
This is a closed neighborhood of $\mu_j$ in $\cM(X)$.

We shall show that $T$ and $U_1\times \cdots \times U_k$ have the desired property. We have checked that the condition (1) holds. 
Thus it suffices to check the condition (2). 

Let $F\subseteq \Gamma$ be an infinite independence set for $\oU=(U_1, \dots, U_k)$. We argue by contradiction. Assume that for each $\ot\in T$ the tuple $\oA_{\of, \ot}$ has no infinite independence set $J\subseteq F$.
We list the elements of $T$ as
$$\ot^{(1)}=(t_1^{(1)}, \dots, t_k^{(1)}), \dots, \ot^{(|T|)}=(t_1^{(|T|)}, \dots, t_k^{(|T|)}).$$
We set $F_0=F$. Via induction on $l=1, \dots, |T|$,  by Lemma~\ref{L-IT combinatorial} we can find an infinite set $F_l\subseteq F_{l-1}$ such that for every $x\in X$ there exists a $j_{l, x}\in [k]$ such that the set $\{s\in F_l: x\in s^{-1}f_{j_{l, x}}^{-1}([t_{j_{l, x}}^{(l)}, \infty))\}$ is finite.

Since $F_{|T|}$ is infinite, we can find an injective map $\psi_j: \Nb\rightarrow F_{|T|}$ for each $j\in [k]$ such that the sets $\psi_j(\Nb)$ for $j\in [k]$ are pairwise disjoint. Since $F$ is an independence set for $\oU$ and $\oU$ is a tuple of compact sets, we can find a
$$\nu\in \bigcap_{n\in \Nb}\bigcap_{j\in [k]}\psi_j(n)^{-1}U_j.$$

Let $n\in \Nb$. We set
$$ X_{n, \ot}=\bigcap_{j\in [k]}\psi_j(n)^{-1}f_j^{-1}([t_j, \infty))$$
for every $\ot=(t_1, \dots, t_k)\in T$, and set
$$ X_n=\bigcup_{\ot\in T}X_{n, \ot}, \mbox{ and }  X_n'=\left\{x\in X: \frac{1}{k}\sum_{j\in [k]}f_j(\psi_j(n)x)\ge \frac{R+r}{2}\right\}.$$
Note that
\begin{align*}
\int_X \frac{1}{k}\sum_{j\in [k]}f_j(\psi_j(n)x)\, d\nu(x)&=\frac{1}{k}\sum_{j\in [k]} (\psi_j(n) \nu)(f_j)\\
&\ge \frac{1}{k}\sum_{j\in [k]}(\mu_j(f_j)-(R-r)/4)=\frac{3R+r}{4},
\end{align*}
and
\begin{align*}
\int_X \frac{1}{k}\sum_{j\in [k]}f_j(\psi_j(n)x)\, d\nu(x)&\le \nu(X_n')C+\nu(X\setminus X_n')\frac{R+r}{2}\\
&=\nu(X_n')(C-\frac{R+r}{2})+\frac{R+r}{2}.
\end{align*}
Thus
$$\frac{3R+r}{4}\le \nu(X_n')(C-\frac{R+r}{2})+\frac{R+r}{2},$$
whence
$$ \nu(X_n')\ge \frac{R-r}{4C-2R-2r}.$$
Let $x\in X_n'$. For each $j\in [k]$
we can find an $i_j\in [2M]$ so that
$$-C+\frac{(i_j-1)C}{M}\le f_j(\psi_j(n)x)\le -C+\frac{i_j C}{M}.$$
We set $t_j=-C+\frac{(i_j-1)C}{M}$ for each $j\in [k]$.
Note that
\begin{align*}
\frac{R+r}{2}\le \frac{1}{k}\sum_{j\in [k]}f_j(\psi_j(n)x)\le \frac{1}{k}\sum_{j\in [k]}\left(t_j+\frac{C}{M}\right),
\end{align*}
whence
$$ \frac{1}{k}\sum_{j\in [k]}t_j\ge \frac{R+r}{2}-\frac{C}{M}\ge r,$$
which means that $(i_1, \dots, i_j)$ lies in $I$ and $\ot:=(t_1, \dots, t_k)$ lies in $T$. We have $x\in X_{n, \ot}\subseteq X_n$, and hence $X_n'\subseteq X_n$.  Thus
$$\nu(X_n)\ge \nu(X_n')\ge \frac{R-r}{4C-2R-2r}.$$
Then there is some $\ot^{\{n\}}\in T$ such that
$$ \nu(X_{n, \ot^{\{n\}}})\ge \frac{R-r}{(4C-2R-2r)|T|}=:\delta.$$

We can find some $l\in [|T|]$ and some infinite set $W\subseteq \Nb$ such that $\ot^{\{n\}}=\ot^{(l)}$ for all $n\in W$. By Lemma~\ref{L-finite intersection} there is some infinite set $W'\subseteq W$ such that 
$$\nu\left(\bigcap_{n\in W''}X_{n, \ot^{(l)}}\right)=\nu\left(\bigcap_{n\in W''}X_{n, \ot^{\{n\}}}\right)>0$$ 
for every nonempty finite set $W''\subseteq W'$. In particular, $\bigcap_{n\in W''}X_{n, \ot^{(l)}}\neq \emptyset$ for every nonempty finite set $W''\subseteq W'$. Since $X_{n, \ot^{(l)}}$ is compact for each $n\in \Nb$, it follows that $\bigcap_{n\in W'}X_{n, \ot^{(l)}}\neq \emptyset$. Then we can find an $x\in \bigcap_{n\in W'}X_{n, \ot^{(l)}}$. For each $j\in [k]$,
$$\{s\in F_l: x\in s^{-1}f_j^{-1}([t_j^{(l)}, \infty))\}\supseteq \psi_j(W')$$
is infinite, a contradiction to our choice of $F_l$. Therefore, $\oA_{\of, \ot}$ for some $\ot\in T$ has an infinite independence set $J\subseteq F$.
\end{proof}

The following is an analogue of Lemma~\ref{L-IN simul}, and can be proved in the same way using Lemma~\ref{L-func to infinite}.

\begin{lemma} \label{L-IT simul}
Let $k, n\in \Nb$, $\omu=(\mu_1, \dots, \mu_k)\in \cM(X)^k$  and $\of^{(l)}=(f_1^{(l)}, \dots, f_k^{(l)})\in C(X)^{\oplus k}$ for $l\in [n]$ such that $R^{(l)}:=\frac{1}{k}\sum_{j\in [k]}\mu_j(f_j^{(l)})>0$ for each $l\in [n]$. 
%Let $C>0$ with $C\ge \max_{l\in [n]}\|\of^{(l)}\|_\infty$ and  
Let $0<r^{(l)}<R^{(l)}$ for each $l\in [n]$.  Then there are some finite set $T^{(l)}\subseteq \Rb^k$ for each $l\in [n]$,
 and some product neighborhood $U_1\times \cdots \times U_k$ of $\omu$ in $\cM(X)^k$ such that the following hold:
  \begin{enumerate}
  \item $\frac{1}{k}\sum_{j\in [k]}t_j\ge r^{(l)}$ for every $l\in [n]$ and $(t_1, \dots, t_k)\in T^{(l)}$, and
  \item for any infinite independence set $F\subseteq \Gamma$ of $(U_1, \dots, U_k)$,
 there are  some infinite set $J\subseteq F$ and some $\ot^{(l)}=(t_1^{(l)}, \dots, t_k^{(l)})\in T^{(l)}$ for each $l\in [n]$ so that
 $J$ is an independence set for the $k$-tuple
 \begin{align*}
 \Big(\prod_{l\in [n]} (f_1^{(l)})^{-1}([t_1^{(l)}, \infty)), \dots, \prod_{l\in [n]} (f_k^{(l)})^{-1}([t_k^{(l)}, \infty))\Big)
  \end{align*}
  of subsets of $X^n$.
  \end{enumerate}
 \end{lemma}

The following lemma is proved in the same way as Lemma~\ref{L-IN dense}, using Theorem~\ref{T-IT basic} and Lemmas~\ref{L-IT convex},
\ref{L-convex intersection} and \ref{L-IT simul}.

\begin{lemma} \label{L-IT dense}
Let $k\in \Nb$. Then $\IT_k(\cM(X))=\overline{\bigcup_{n\in \Nb}\cB_{k, n}}$.
\end{lemma}

Now Theorem~\ref{T-IT main} can be proved in the same way as Theorem~\ref{T-IN main}, using Theorem~\ref{T-IT basic} and Lemmas~\ref{L-IT dense} and \ref{L-IT closure vs push}.

The following is an analogue of Corollary~\ref{C-IN finite support} with essentially the same proof, using Theorem~\ref{T-IT main} and Lemmas~\ref{L-IT finite support} and \ref{L-IT convex}.

\begin{corollary} \label{C-IT finite support}
Let $k\in \Nb$.
Then for any $N_1, \dots, N_k\in \Nb$, setting $N=\prod_{j\in [k]}N_j$, one has
$$ (\cM_{N_1}(X)\times \cdots \times \cM_{N_k}(X))\cap \IT_k(\cM(X))\subseteq \IT_k(\cM_N(X)).$$
\end{corollary}

In view of parts (3) and (5) of Theorem~\ref{T-IT basic}, from the case $k=2$ of parts (1) and (2) of Theorem~\ref{T-IT main} we obtain immediately the following result, which also follows from \cite[Proposition 6.6]{KL07}.

\begin{corollary} \label{C-IT for prob}
$\Gamma\curvearrowright X$ is tame if and only if $\Gamma\curvearrowright \cM(X)$ is tame.
\end{corollary}

For $k\in \Nb$, we say $\Gamma\curvearrowright X$ is {\it uniformly untame of order $k$} if $\IT_k(X)=X^k$ \cite[page 897]{KL07}. In view of Theorem~\ref{T-IT basic}.\eqref{i-IT subset}, from Theorem~\ref{T-IT main}.(2) we have the following corollary.

\begin{corollary} \label{C-IT order k}
Let  $k, N\in \Nb$ and $\olambda\in \orP_N$. If $\Gamma\curvearrowright \cM_\olambda(X)$ or $\Gamma\curvearrowright \cM(X)$ is uniformly untame of order $k$, then so is $\Gamma\curvearrowright X$.
\end{corollary}

%%%%%%%%%%%%%%%%%%%%%%%%%%%%%%%%%%%%%%%%%%%%%%%%%%%%%%%%%%%%%%%%%%%%%%%%%%%%%%%%%%%%%%%%%%%%%%%%%%%%%%%%%%%%%%%%
\section{Non-convexity of $\IN$ and $\IT$-tuples} \label{S-example}

In this section we construct a minimal subshift of $\Gamma=\Zb$ such that neither $\IN_2(\cM(X))$ nor $\IT_2(\cM(X))$ is convex.

We first use Theorem~\ref{T-IN main} to deduce a sufficient condition for certain pairs in $\cM_2(X)$ to fail to be in $\IN_2(\cM(X))$.

\begin{lemma} \label{L-not IN}
Let a countably infinite group $\Gamma$ act on a compact metrizable space $X$ continuously. Let $(x_1, x_2)$ and $(x_3, x_4)$ be distinct elements in $X^2$ such that $(x_1, x_4), (x_2, x_3)\not\in \IN_2(X)$ and $((x_1, x_3), (x_2, x_4))\not\in \IN_2(X^2)$. Then for each $0<\lambda<1$,
$$(\lambda \delta_{x_1}+(1-\lambda)\delta_{x_3},  \lambda \delta_{x_2}+(1-\lambda)\delta_{x_4})\not\in \IN_2(\cM(X)).$$
\end{lemma}
\begin{proof} We argue by contradiction. Assume that $(\lambda \delta_{x_1}+(1-\lambda)\delta_{x_3},  \lambda \delta_{x_2}+(1-\lambda)\delta_{x_4})\in \IN_2(\cM(X))$ for some $0<\lambda<1$. By Theorem~\ref{T-IN main}.(3) there is some $\mu\in \cM_\IN(X^2)$ such that
\begin{align} \label{E-not IN}
 \mu^{(1)}=\lambda \delta_{x_1}+(1-\lambda)\delta_{x_3} \quad \text{ and } \quad \mu^{(2)}=\lambda \delta_{x_2}+(1-\lambda)\delta_{x_4}.
\end{align}
Then
\begin{align*}
\supp(\mu)\subseteq \{x_1, x_3\}\times \{x_2, x_4\}=\{(x_1, x_2), (x_1, x_4), (x_3, x_2), (x_3, x_4)\}.
\end{align*}
Since $\mu\in \cM_\IN(X^2)$ and $(x_1, x_4), (x_3, x_2)\not\in \IN_2(X)$, we must have
$$ \supp(\mu)\subseteq \{(x_1, x_2),  (x_3, x_4)\}. $$
Using $\mu\in \cM_\IN(X^2)$ again and $((x_1, x_3), (x_2, x_4))\not\in \IN_2(X^2)$, we see that $\supp(\mu)$ cannot be equal to $\{(x_1, x_2),  (x_3, x_4)\}$. Therefore $\supp(\mu)$ is equal to either $\{(x_1, x_2)\}$ or $\{(x_3, x_4)\}$, i.e., $\mu$ is equal to either $\delta_{(x_1, x_2)}$ or
$\delta_{(x_3, x_4)}$. From \eqref{E-not IN} we conclude that $(x_1, x_2)=(x_3, x_4)$, a contradiction to the assumption that $(x_1, x_2)$ and $(x_3, x_4)$ are distinct.
\end{proof}

Let $m\in \Nb$. An element $x\in \Omega_m:=\{0, 1, \dots, m-1\}^\Zb$ is called a {\it Toeplitz sequence} if for each $j\in \Zb$ there is an $n\in \Nb$ such that $x(j+tn)=x(j)$ for all $t\in \Zb$. We endow $\Omega_m$ with the shift $\Zb$-action given by $(Tx)(s)=x(s+1)$ for all $s\in \Zb$. The subshift generated by $x$, i.e., the restriction of the shift action to the orbit closure of $x$,  is called a {\it Toeplitz subshift} \cite{JK}.  Note that every Toeplitz subshift is minimal \cite[Theorem 4]{JK}.

We set $m=5$. We choose an increasing sequence $n_1<n_2<\cdots$ in $\Nb$ with $n_p|n_{p+1}$ for each $p\in \Nb$ and $2^{p+2}$ distinct elements $y_{p, 1}, y_{p, 2}, \dots, y_{p, 2^{p+1}}, y'_{p, 1}, \dots, y'_{p, 2^{p+1}}$ in $\Zb/n_p\Zb$ for each $p\in \Nb$ with the following properties.
\begin{enumerate}[label=(\roman*)]
\item $y_{p+1, 2q-1}\equiv y_{p+1, 2q}\equiv y'_{p+1, 2q-1}\equiv y'_{p+1, 2q}\equiv y'_{p, q} \mod n_p$ for all $p\in \Nb$ and all $q\in [2^{p+1}]$.
\item For each $p\in \Nb$ there are $z_p, u_p\in \Zb/n_p\Zb$ such that $y'_{p, q}=y_{p, q}+z_p$ for all $q\in [2^p]$ and $y'_{p, q}=y_{p, q}+u_p$ for all $q\in [2^{p+1}]\setminus [2^p]$.
%\item There does not exist any $z\in \Zb$ with $z+n_p\Zb \in \{y'_{p, q}: 1\le q\le 2^{p+1}\}$ for all $p\in \Nb$.
\item For each $p\in \Nb$ and any $v_1, v_2, v_3\in V_p:=\bigcup_{q\in [2^p]}\{y_{p, q}, y'_{p, q}\}$ and any $w\in W_p:=\bigcup_{q\in [2^{p+1}]\setminus [2^p]}\{y_{p, q}, y'_{p, q}\}$ we have $v_1-v_2\neq v_3-w$.
\item For each $p\in \Nb$ and any distinct $v_1, v_2\in V_p$ and any $w_1, w_2\in W_p$, we have $v_1-v_2\neq w_1-w_2$.
\item There does not exist any $z\in \Zb$ such that $z+n_p\Zb \in \{y'_{p, q}: q\in [2^{p+1}]\}$ for all $p\in \Nb$.
\end{enumerate}
To see that such $n_p$ and $y_{p, q}, y'_{p, q}$ exist, we shall construct them inductively on $p$ to satisfy the above properties (i)-(iv) and (v') below which is stronger than (v):
\begin{enumerate}
\item[(v')] For each $p\in \Nb$, none of $t+n_p\Zb$ for $t=-p, -p+1, \dots, 0, \dots, p-1, p$ lies in $\{y'_{p, q}: q\in [2^{p+1}]\}$.
\end{enumerate}

To start, for $p=1$, we take integers $1<\tilde{y}_{1, 1}<\tilde{y}_{1, 2}$ first, then take an integer $\tilde{z}_1>\tilde{y}_{1, 2}$ and set $\tilde{y}'_{1, q}=\tilde{y}_{1, q}+\tilde{z}_1$ for $q\in [2]$. Next we take integers $2\tilde{y}'_{1, 2}<\tilde{y}_{1, 3}<\tilde{y}_{1, 4}$ such that
$\tilde{y}_{1, 4}-\tilde{y}_{1, 3}>\tilde{y}'_{1, 2}$, and take an integer $\tilde{u}_1>\tilde{y}_{1, 4}$ and set $\tilde{y}'_{1, q}=\tilde{y}_{1, q}+\tilde{u}_1$ for $q\in [4]\setminus [2]$. Finally, we take an integer $n_1>\tilde{y}'_{1, 4}+\tilde{y}'_{1, 2}+1$, and set
$$y_{1, q}=\tilde{y}_{1, q}+n_1\Zb \quad \mbox{ and }  \quad y'_{1, q}=\tilde{y}'_{1, q}+n_1\Zb \quad \mbox{ for all } q\in [2^2],$$
and 
$$z_1=\tilde{z}_1+n_1\Zb \quad \mbox{ and } \quad u_1=\tilde{u}_1+n_1\Zb. $$
Then (ii)-(iv) and (v') hold for $p=1$. 

Assume that we have constructed $n_1<n_2<\dots <n_j$ satisfying $n_p|n_{p+1}$ for all $1\le p<j$ and distinct $y_{p, q}, y'_{p, q}\in \Zb/n_p\Zb$ for $p\in [j]$ and $q\in [2^{p+1}]$ satisfying (i) for $1\le p<j$ and (ii)-(iv) and (v') for $p\in [j]$. We take integers $j+1<\tilde{y}_{j+1, 1}<\dots <\tilde{y}_{j+1, 2^{j+1}}$ first such that
$$\tilde{y}_{j+1, 2q-1}\equiv \tilde{y}_{j+1, 2q}\equiv y'_{j, q} \mod n_j\Zb$$
for all $q\in [2^j]$, then take an integer $\tilde{z}_{j+1}\in n_j\Zb$ satisfying $\tilde{z}_{j+1}>\tilde{y}_{j+1, 2^{j+1}}$ and set $\tilde{y}'_{j+1, q}=\tilde{y}_{j+1, q}+\tilde{z}_{j+1}$ for all $q\in [2^{j+1}]$. Next we take integers $2\tilde{y}'_{j+1, 2^{j+1}}<\tilde{y}_{j+1, 2^{j+1}+1}<\dots<\tilde{y}_{j+1, 2^{j+2}}$ such that
$\tilde{y}_{j+1, q+1}-\tilde{y}_{j+1, q}>\tilde{y}'_{j+1, 2^{j+1}}$ for all $2^{j+1}<q< 2^{j+2}$ and
$$\tilde{y}_{j+1, 2q-1}\equiv \tilde{y}_{j+1, 2q}\equiv y'_{j, q} \mod n_j\Zb$$
for all $q\in [2^{j+1}]\setminus [2^{j}]$, and take an integer $\tilde{u}_{j+1}\in n_j\Zb$ with $\tilde{u}_{j+1}>\tilde{y}_{j+1, 2^{j+2}}$ and set $\tilde{y}'_{j+1, q}=\tilde{y}_{j+1, q}+\tilde{u}_{j+1}$ for $q\in [2^{j+2}]\setminus [2^{j+1}]$. Finally, we take an integer $n_{j+1}\in n_j\Zb$ such that $n_{j+1}>\tilde{y}'_{j+1, 2^{j+2}}+\tilde{y}'_{j+1, 2^{j+1}}+j+1+n_j$, and set
\begin{align*}
y_{j+1, q}=\tilde{y}_{j+1, q}+n_{j+1}\Zb \quad \mbox{ and }  \quad y'_{j+1, q}=\tilde{y}'_{j+1, q}+n_{j+1}\Zb \quad \mbox{ for all } q\in [2^{j+2}],
\end{align*}
and
\begin{align*}
z_{j+1}=\tilde{z}_{j+1}+n_{j+1}\Zb \quad \mbox{ and } \quad u_{j+1}=\tilde{u}_{j+1}+n_{j+1}\Zb.
\end{align*}
Then (i) holds for all $p\in [j]$ and (ii)-(iv) and (v') hold for all $p\in [j+1]$. This finishes the induction step and gives us the $n_p$ and $y_{p, q}, y'_{p, q}$ as desired.

We define a map $h: \bigcup_{p\in \Nb}\{y_{p, q}: q\in [2^{p+1}]\}\rightarrow [4]$ by
\begin{align*}
h(y_{p, q})=\begin{cases}
1, & \quad \text{if } q\le 2^p \text{ and is odd,}\\
2, & \quad \text{if } q\le 2^p \text{ and is even,}\\
3, & \quad \text{if } q> 2^p \text{ and is odd,}\\
4, & \quad \text{if } q> 2^p \text{ and is even.}
\end{cases}
\end{align*}
Now we define an $x\in \Omega_5$ by
\begin{align*}
x(s)=\begin{cases}
h(y_{p, q}),  & \quad \text{if } s\equiv y_{p, q} \mod n_p \text{ for some } p\in \Nb \text{ and some } q\in  [2^{p+1}],\\
0 & \quad \text{ otherwise.}
\end{cases}
\end{align*}

\begin{lemma} \label{L-Toplitz}
The element $x$ is well defined and Toeplitz.
\end{lemma}
\begin{proof} From Property (i) it is clear that if $s\in \Zb$, $p_i\in \Nb$ and $q_i\in [2^{p_i+1}]$ for $i=1, 2$ such that
$s\equiv y_{p_i, q_i} \mod n_{p_i}$ for $i=1, 2$, then $p_1=p_2$ and $q_1=q_2$. Thus $x$ is well defined.

Let $s\in \Zb$ such that $x(s)\in [4]$. Then
$$s \equiv y_{p, q} \mod n_p$$
for some $p\in \Nb$ and $q\in [2^{p+1}]$. For any $t\in \Zb$, we have
$$s+tn_p \equiv s \equiv y_{p, q} \mod n_p,$$
whence $x(s+tn_p)=h(y_{p, q})=x(s)$.

Let $s\in \Zb$ such that $x(s)=0$. By Property (v) we can find some $p\in \Nb$ such that $s+n_p\Zb\not\in \{y'_{p, q}: q\in  [2^{p+1}]\}$. Let $t\in \Zb$. We claim that $s+tn_p \not\equiv y_{p', q'} \mod n_{p'}$ for all $p'\in \Nb$ and $q'\in [2^{p'+1}]$. Indeed, for any $p'>p$ and any $q'\in [2^{p'+1}]$, by Property (i) we have $y_{p', q'}\equiv y'_{p, q} \mod n_p$ for some $q\in [2^{p+1}]$, whence
$$s+tn_p\equiv s\not\equiv y_{p', q'} \mod n_p,$$
which implies that $s+tn_p \not\equiv y_{p', q'} \mod n_{p'}$. If $s+tn_p\equiv y_{p', q'} \mod n_{p'}$ for some $p'\le p$ and $q'\in [2^{p'+1}]$, then
$$ s\equiv s+t n_p\equiv y_{p', q'} \mod n_{p'},$$
and hence $x(s)=h(y_{p', q'})\in [4]$, which is impossible. This proves our claim. Then $x(s+tn_p)=0=x(s)$.

Now we conclude that $x$ is Toeplitz.
\end{proof}

We denote by $X$ the orbit closure of $x$ under the shift action. We shall show that neither $\IN_2(\cM(X))$ nor $\IT_2(\cM(X))$ is convex.

\begin{lemma} \label{L-difference}
For any $p\in \Nb$ and $q_1, q_2\in [2^p]$ or $q_1, q_2\in [2^{p+1}]\setminus [2^p]$, we have
$$y_{p+1, 2q_1-1}-y_{p+1, 2q_2-1}\equiv y_{p, q_1}-y_{p, q_2} \mod n_p.$$
\end{lemma}
\begin{proof}  Using Properties (i) and (ii) we have
\begin{displaymath}
y_{p+1, 2q_1-1}-y_{p+1, 2q_2-1}\equiv y'_{p, q_1}-y'_{p, q_2}\equiv y_{p, q_1}-y_{p, q_2} \mod n_p. \tag*{\qedsymbol}
\end{displaymath}
     \renewcommand{\qedsymbol}{}
     \vspace{-\baselineskip}
\end{proof}

\begin{lemma} \label{L-realize map}
The following hold:
\begin{enumerate}
\item For any $p\in \Nb$ and any map $\sigma: [p]\rightarrow [2]$ there are some $q_i\in [2^i]$ for each $i\in [p]$ such that $q_{i+1}\in \{2q_i-1, 2q_i\}$ for each $1\le i<p$ and $h(y_{i, q_i})=\sigma(i)$ for all $i\in [p]$.
\item For any $p\in \Nb$ and any map $\sigma: [p]\rightarrow \{3, 4\}$ there are some $q_i\in [2^{i+1}]\setminus [2^i]$ for each $i\in [p]$ such that $q_{i+1}\in \{2q_i-1, 2q_i\}$ for each $1\le i<p$ and $h(y_{i, q_i})=\sigma(i)$ for all $i\in [p]$.
\end{enumerate}
\end{lemma}
\begin{proof} We prove (1) by induction on $p$. The case $p=1$ is trivial. Assume that the assertion holds for $p=l$. Let $\sigma: [l+1]\rightarrow [2]$. By the inductive assumption there are some $q_i\in [2^i]$ for each $i\in [l]$ such that $q_{i+1}\in \{2q_i-1, 2q_i\}$ for each $1\le i<l$ and $h(y_{i, q_i})=\sigma(i)$ for all $i\in [l]$.
If $\sigma(l+1)=1$, setting $q_{l+1}=2q_l-1$ we have $q_{l+1}\in [2^{l+1}]$,
and
$$h(y_{l+1, q_{l+1}})=1=\sigma(l+1).$$
If $\sigma(l+1)=2$, setting $q_{l+1}=2q_l$ we have $q_{l+1}\in [2^{l+1}]$,
and
$$h(y_{l+1, q_{l+1}})=2=\sigma(l+1).$$
This finishes the induction step and proves (1). The proof for (2) is similar.
\end{proof}

For each $j\in [4]$ we denote by $A_j$ the set of elements in $X$ taking value $j$ at $0$. This is a clopen subset of $X$. For each $p\in \Nb$ we take $a_p, b_p\in \Zb$ such that
$$y_{p, 1}=a_p+ n_p\Zb \quad  \mbox{ and } \quad y_{p, 2^p+1}=b_p+n_p\Zb.$$

\begin{lemma} \label{L-example IT}
The set $\{a_p: p\in \Nb\}$ is an independence set for the pair $(A_1, A_2)$, and the set $\{b_p: p\in \Nb\}$ is an independence set for the pair $(A_3, A_4)$.
\end{lemma}
\begin{proof} We prove the first assertion. The proof for the second assertion is similar.

It suffices to show that $\{a_1, \dots, a_p\}$ is an independence set for the pair $(A_1, A_2)$ for every $p\in \Nb$. Let $\sigma: [p]\rightarrow [2]$. By Lemma~\ref{L-realize map} we can find some $q_i\in [2^i]$ for each $i\in [p]$ such that $q_{i+1}\in \{2q_i-1, 2q_i\}$ for each $1\le i<p$ and $h(y_{i, q_i})=\sigma(i)$ for all $i\in [p]$. We take an $a\in \Zb$ such that 
$$a\equiv a_p-y_{p, q_p}\mod n_p.$$ 
If $1\le i<p$, by Property (i) and Lemma~\ref{L-difference} we have
$$  a_{i+1}-y_{i+1, q_{i+1}}\equiv  y_{i+1, 1}-y_{i+1, 2q_i-1}\equiv y_{i, 1}-y_{i, q_i}\equiv a_i-y_{i, q_i}\mod n_i.$$
Taking $i=p-1$ we get
$$ a\equiv a_{p}-y_{p, q_{p}} \equiv a_{p-1}-y_{p-1, p_{p-1}} \mod n_{p-1}.$$
Taking $i=p-2$ we get
$$ a \equiv a_{p-1}-y_{p-1, q_{p-1}} \equiv a_{p-2}-y_{p-2, p_{p-2}} \mod n_{p-2}.$$
Going on this way, we get
$$ a \equiv a_i-y_{i, q_i} \mod n_i$$
for all $i\in [p]$. Equivalently,
$$ a_i-a \equiv y_{i, q_i} \mod n_i$$
for all $i\in [p]$. 
Then
$$(T^{a_i-a}x)(0)=x(a_i-a)=h(y_{i, q_i})=\sigma(i)$$
for all $i\in [p]$, i.e., $T^{-a} x\in \bigcap_{i\in [p]}T^{-a_i}A_{\sigma(i)}$. Therefore $\{a_1, \dots, a_p\}$ is an independence set for the pair $(A_1, A_2)$.
\end{proof}

For each $s\in \Zb$ with $x(s)\in [4]$, let $P(s)$ and $Q(s)$ denote the unique positive integers such that $1\le Q(s)\le 2^{P(s)+1}$ and $s\equiv y_{P(s), Q(s)} \mod n_{P(s)}$. The uniqueness of such $P(s)$ and $Q(s)$ is shown in the proof of Lemma~\ref{L-Toplitz}. 

\begin{lemma} \label{L-not IN1}
Let $i\in \{1, 2\}$ and $j\in \{3, 4\}$. Then each independence set for the pair $(A_i, A_j)$ has cardinality at most $1$.
\end{lemma}
\begin{proof} We argue by contradiction. Assume that $\{s_1, s_2\}$ is an independence set for $(A_i, A_j)$ for some distinct $s_1, s_2\in \Zb$. Then we can find some $z, z'\in X$ such that
$$z\in T^{-s_1}A_i\cap T^{-s_2}A_i \quad \text{ and } \quad z'\in T^{-s_1}A_i \cap T^{-s_2}A_j.$$ 
We may find $a, b\in \Zb$ so that $T^{a}x$ and $T^{b}x$ are close enough to $z$ and $z'$ respectively such that
$$T^ax\in T^{-s_1}A_i\cap T^{-s_2}A_i \quad \text{ and } \quad T^b x\in T^{-s_1}A_i \cap T^{-s_2}A_j.$$ 
This means
$$x(s_1+a)=x(s_2+a)=x(s_1+b)=i \quad \text{ and } \quad x(s_2+b)=j.$$ 
Let
$$p=\min(P(s_1+a), P(s_2+a), P(s_1+b), P(s_2+b)).$$
By Property (i) we have 
$$s_1+a +n_p \Zb, \, s_2+a+n_p\Zb, \, s_1+b+n_p\Zb \in V_p \quad \text{ and } \quad s_2+b+n_p\Zb \in W_p.$$ 
From Property (iii) we have
$$ (s_1+a)-(s_2+a)\not\equiv (s_1+b)-(s_2+b) \mod n_p,$$
which is impossible. Therefore each independence set for the pair $(A_i, A_j)$ has cardinality at most $1$.
\end{proof}

\begin{lemma} \label{L-not IN2}
Each independence set for the pair $(A_1\times A_3, A_2\times A_4)$ has cardinality at most $1$.
\end{lemma}
\begin{proof} We argue by contradiction. Assume that $\{s_1, s_2\}$ is an independence set for $(A_1\times A_3, A_2\times A_4)$ for some distinct $s_1, s_2\in \Zb$. Then 
we can find some $(z, z')\in X^2$ such that
$$(z, z')\in T^{-s_1}(A_1\times A_3)\cap T^{-s_2}(A_2\times A_4).$$ 
We may find $a, b\in \Zb$ so that $(T^{a}x, T^{b}x)$ is close enough to $(z, z')$ such that
$$(T^a x, T^b x)\in T^{-s_1}(A_1\times A_3)\cap T^{-s_2}(A_2\times A_4).$$ 
This means
$$x(s_1+a)=1, \quad x(s_2+a)=2, \quad x(s_1+b)=3 \quad \text{ and } \quad x(s_2+b)=4.$$ 
%we can find some $a, b\in \Zb$ such that $x(s_1+a)=1$, $x(s_2+a)=2$, $x(s_1+b)=3$, and $x(s_2+b)=4$. Set
Let
$$p=\min(P(s_1+a), P(s_2+a), P(s_1+b), P(s_2+b)).$$
By Property (i) we have 
$$s_1+a +n_p \Zb, \, s_2+a+n_p\Zb\in V_p \quad \text{ and } \quad s_1+b+n_p\Zb, \, s_2+b+n_p\Zb \in W_p.$$
Since
$$ (s_1+a)-(s_2+a)\equiv (s_1+b)-(s_2+b) \mod n_p,$$
by Property (iv) we have 
$$ s_1+a\equiv s_2+a \mod n_p \quad \text{ and } \quad  s_1+b\equiv s_2+b\mod n_p.$$
If $p=\min(P(s_1+a), P(s_2+a))$, then $x(s_1+a)=x(s_2+a)$, contrary to the fact that $x(s_1+a)=1$ and $x(s_2+a)=2$.
If $p=\min(P(s_1+b), P(s_2+b))$, then $x(s_1+b)=x(s_2+b)$, contrary to the fact that $x(s_1+b)=3$ and $x(s_2+b)=4$.
Therefore each independence set for the pair $(A_1\times A_3, A_2\times A_4)$ has cardinality at most $1$.
\end{proof}

Since $A_1, A_2, A_3$ and $A_4$ are all closed, by Lemma~\ref{L-example IT} and Theorem~\ref{T-IT basic}.\eqref{i-IT tuple}
we can find $x_j\in A_j$ for $j\in [4]$ such that 
$$(x_1, x_2), (x_3, x_4)\in \IT_2(X).$$ 
For any $i\in \{1, 2\}$ and $j\in\{3, 4\}$, since $A_i\times A_j$ is a product neighborhood of $(x_i, x_j)$ in $X^2$, by Lemma~\ref{L-not IN1} we have $(x_i, x_j)\not\in \IN_2(X)$. Because $A_1\times A_3$ and $A_2\times A_4$ are neighborhoods of $(x_1, x_3)$ and $(x_2, x_4)$ in $X^2$ respectively, from Lemma~\ref{L-not IN2} we know that $((x_1, x_3), (x_2, x_4))\not\in \IN_2(X^2)$. By Lemma~\ref{L-not IN} we have
$$(\lambda \delta_{x_1}+(1-\lambda)\delta_{x_3},  \lambda \delta_{x_2}+(1-\lambda)\delta_{x_4})\not\in \IN_2(\cM(X))$$
for all $0<\lambda<1$. As IT-tuples are always IN-tuples, we also have
$$(x_1, x_2), (x_3, x_4)\in \IN_2(X),$$
and 
$$(\lambda \delta_{x_1}+(1-\lambda)\delta_{x_3},  \lambda \delta_{x_2}+(1-\lambda)\delta_{x_4})\not\in \IT_2(\cM(X))$$
for all $0<\lambda<1$. 
Therefore neither $\IN_2(\cM(X))$ nor $\IT_2(\cM(X))$ is convex.

%%%%%%%%%%%%%%%%%%%%%%%%%%%%%%%%%%%%%%%%%%%%%%%%%%%%%%%%%%%%%%%%%%%%%%%%%%%%%%%%%%%%%%%%%%%%%%%%%%%%%%%%%%%%%%%%%%%%%%
\section{Application to local theory of Banach spaces} \label{S-Banach}

For a set $E$ in a normed vector space $V$, we denote by $\widetilde{\rm co}(E)$ the convex hull of $E\cup \{0\}$ in $V$. For $r>0$, we denote by $B_r(V)$ the closed $r$-ball of $V$, i.e., 
$$B_r(V)=\{v\in V: \|v\|\le r\}.$$
For $n\in \Nb$ and $1\le p\le \infty$, we write $\ell_p^n$ for $\Rb^n$ equipped with the $\ell_p$-norm given in \eqref{E-p norm1} and \eqref{E-p norm2}. 

\begin{lemma} \label{L-FLM}
Let $0<\delta\le 1$ and $1<p\le \infty$. There exist a finite set $T\subseteq \Rb^2$ and $c'>0$, depending only on $\delta$ and $p$, such that the following hold:
\begin{enumerate}
\item $t_1\ge t_2+\delta$ for every $(t_1, t_2)\in T$, and
\item for any $n\in \Nb$, if $E\subseteq B_1(\ell_p^n)$  and $\widetilde{\rm co}(E)\supseteq B_\delta(\ell_p^n)$, then there are some $J\subseteq [n]$ with $|J|\ge c'n$
and $(t_1, t_2)\in T$ such that for any map $\sigma: J\rightarrow [2]$ there is some $v\in E$ so that $v(j)\ge t_1n^{-1/p}$ for all $j\in \sigma^{-1}(1)$ and $v(j)\le t_2n^{-1/p}$ for all $j\in \sigma^{-1}(2)$.
\end{enumerate}
\end{lemma}
\begin{proof} We take $1\le q<\infty$ such that $1/p+1/q=1$. Let $k=2$, $C=1$, $R=\delta$, and $r=\delta/2$. Then we have $c, N$ and $T$ given by Lemma~\ref{L-func to indep}. We may assume that $N\ge 2$. We set
$$T'=\{(t_1, t_2): (t_1, -t_2)\in T\}.$$
We shall show that $T''=T'\cup \{(\delta, -\delta)\}$ and $c'=\min(c, \frac{1}{N})$ have the desired property. 

For each $(t_1, t_2)\in T'$, we have $\frac{1}{2}(t_1-t_2)\ge r=\delta/2$, whence $t_1\ge t_2+\delta$. It follows that (1) holds. 

Let $n\in \Nb$ and $E\subseteq B_1(\ell_p^n)$  such that $\widetilde{\rm co}(E)\supseteq B_\delta(\ell_p^n)$.

 We consider first the case $n\ge N$.
 Let $\psi: [n]\rightarrow [2]$. We denote by $g_\psi$ the element in $B_\delta(\ell_p^n)$ given by
 $$g_\psi(i)=(-1)^{\psi(i)+1}\delta n^{-1/p}$$
 for all $i\in [n]$.
Then $g_\psi \in \widetilde{\rm co}(E)$, and hence we can write $g_\psi$ as $\sum_{v\in E}t_v v$ for some $t_v\in [0, 1]$ for each $v\in E$ such that $\sum_{v\in E}t_v\le 1$. Note that
\begin{align*}
\sum_{v\in E}t_v\frac{1}{n^{1/q}}\sum_{i\in [n]}(-1)^{\psi(i)+1}v(i)&=\frac{1}{n^{1/q}}\sum_{i\in [n]}(-1)^{\psi(i)+1}\sum_{v\in E}t_v v(i)\\
&=\frac{1}{n^{1/q}}\sum_{i\in [n]}(-1)^{\psi(i)+1}g_\psi(i)\\
&=\frac{1}{n^{1/q}}\sum_{i\in [n]}\delta n^{-1/p}=\delta,
\end{align*}
whence we can find some $v_\psi\in E$ such that
$$ \frac{1}{n^{1/q}}\sum_{i\in [n]}(-1)^{\psi(i)+1}v_\psi(i)\ge \delta.$$
We define a function $f_\psi: [n]\rightarrow \Rb$ by
$$ f_\psi(i)=(-1)^{\psi(i)+1}v_\psi(i)$$
for all $i\in [n]$. 
Then
$$ \|f_\psi\|_p=\|v_\psi\|_p\le 1=C,$$
and 
$$ \frac{1}{\left|[n]\right|^{1/q}}\sum_{i\in [n]}f_\psi(i)=\frac{1}{n^{1/q}}\sum_{i\in [n]}(-1)^{\psi(i)+1}v_\psi(i)\ge \delta=R.$$
By our choice of $c, N$ and $T'$, we can find some $J\subseteq [n]$ with $|J|\ge cn\ge c'n$ and $(t_1, t_2)\in T'$ such that  every map $\sigma: J\rightarrow [2]$ extends to a
$\psi\in \sR([n], 2)$
so that $f_\psi(i)\ge t_1\left| [n]\right|^{-1/p}$ for all $i\in \sigma^{-1}(1)$ and $f_\psi(i)\ge (-t_2)\left| [n]\right|^{-1/p}$ for all $i\in \sigma^{-1}(2)$. That is, 
$$v_\psi(i)=(-1)^{\sigma(i)+1}v_\psi(i)=(-1)^{\psi(i)+1}v_\psi(i)=f_\psi(i)\ge t_1\left| [n]\right|^{-1/p}=t_1n^{-1/p}$$
for all $i\in \sigma^{-1}(1)$ and
 $$-v_\psi(i)=(-1)^{\sigma(i)+1}v_\psi(i)=(-1)^{\psi(i)+1}v_\psi(i)=f_\psi(i)\ge (-t_2)\left| [n]\right|^{-1/p}=-t_2n^{-1/p}$$
 for all $i\in \sigma^{-1}(2)$. 

Next we consider the case $n<N$. We set $J=\{1\}$ and $(t_1, t_2)=(\delta, -\delta)\in T''$. Then $|J|\ge n/N\ge c'n$.
Since
$$ (\delta, \underbrace{0, \dots, 0}_{n-1})\in B_\delta(\ell_p^n)\subseteq \widetilde{\rm co}(E),$$
we can find some $v\in E$ such that $v(1)\ge \delta\ge t_1 n^{-1/p}$. Similarly, using 
$$ (-\delta, \underbrace{0, \dots, 0}_{n-1})\in B_\delta(\ell_p^n)\subseteq \widetilde{\rm co}(E),$$
we can find some $w\in E$ such that $w(1)\le -\delta\le t_2n^{-1/p}$. Then for a map $\sigma: J\rightarrow [2]$, we can use $v$ when $\sigma(1)=1$ and $w$ when $\sigma(1)=2$ to fulfill the requirement of (2). 
\end{proof}

We are ready to prove Theorem~\ref{T-FLM}. 

\begin{proof}[Proof of Theorem~\ref{T-FLM}] We set $\delta=1/C\in (0, 1)$ and take $1<p\le \infty$ such that $1/p+1/q=1$. Then we have $T$ and $c'$ given by Lemma~\ref{L-FLM}. Set $c=\frac{2}{c'\log 2}>0$. 

Let $n, m\in \Nb$ with $m\ge 2$ such that $\ell_q^n$ is $C$-isomorphic to a linear subspace of $\ell_\infty^m$. Then we have 
an injective linear map $\Phi: \ell_q^n\rightarrow \ell_\infty^m$ such that $\|\Phi\|\cdot \|\Phi^{-1}\|\le C$, where $\Phi^{-1}$ denotes the inverse map from $\Phi(\ell_q^n)$ to $\ell_q^n$. Multiplying $\Phi$ by a scalar if necessary, we may assume that $\|\Phi\|=1$. Then $\|\Phi^{-1}\|\le C$.

For each Banach space $V$, we denote by $V^*$ the dual space of $V$. For each $\phi\in (\ell_q^n)^*=\ell_p^n$, $\phi\circ \Phi^{-1}$ is in $(\Phi(\ell_q^n))^*$ with
$$\| \phi\circ \Phi^{-1}\|\le \|\phi\|\cdot \|\Phi^{-1}\|\le \|\phi\|C,$$
and hence by the Hahn-Banach theorem $\phi \circ \Phi^{-1}$ extends to a $\omega\in (\ell_\infty^m)^*=\ell_1^m$ with
$$\|\omega\|=\|\phi\circ \Phi^{-1}\|\le \|\phi\|C.$$
Thus, denoting by $\Phi^*$ the dual map $\ell_1^m=(\ell_\infty^m)^*\rightarrow (\ell_q^n)^*=\ell_p^n$ of $\Phi$, we have
$$ \|\Phi^*\|=\|\Phi\|=1 \mbox{ and } \Phi^*(B_1(\ell_1^m))\supseteq B_{1/C}(\ell_p^n)=B_{\delta}(\ell_p^n).$$

We denote by $E$ the standard basis of $\ell_1^m$. Clearly $\widetilde{\rm co}(E\cup (-E))=B_1( \ell_1^m)$. We set $E'=\Phi^*(E\cup (-E))\subseteq B_1(\ell_p^n)$. Then
$$\widetilde{\rm co}(E')=\Phi^*(\widetilde{\rm co}(E\cup (-E)))=\Phi^*(B_1( \ell_1^m))\supseteq B_{\delta}(\ell_p^n).$$
By our choice of $T$ and $c'$, we can find some $J\subseteq [n]$ with $|J|\ge c'n$ and some $(t_1, t_2)\in T$ such that for any map $\sigma: J\rightarrow [2]$ there is some $v\in E'$ so that $v(j)\ge t_1n^{-1/p}$ for all $j\in \sigma^{-1}(1)$ and $v(j)\le t_2n^{-1/p}$ for all $j\in \sigma^{-1}(2)$. Since $t_1\ge t_2+\delta>t_2$, different $\sigma$ must yield different $v$. Thus
$$ 2^{c'n}\le 2^{|J|}\le |E'|\le 2m.$$
Therefore, using $m\ge 2$, we get
\begin{displaymath}
n \le \frac{1}{c'\log 2}\log (2m)\le \frac{2}{c'\log 2}\log m=c\log m. \tag*{\qedsymbol}
\end{displaymath}
     \renewcommand{\qedsymbol}{}
     \vspace{-\baselineskip}
\end{proof}

The bound in Theorem~\ref{T-FLM} is also optimal, as one can see from the following well-known fact (cf. \cite[page 65]{FLM} or \cite[page 56]{Pisier}).

\begin{lemma} \label{L-embedding exist}
Let $C>1$. There is a constant $\lambda>1$ depending only on $C$ such that  every finite-dimensional normed space $V$ is $C$-isomorphic to a linear subspace of $\ell_\infty^m$ for some $m\le \lambda^{\dim V}$.
\end{lemma}

%%%%%%%%%%%%%%%%%%%%%%%%%%%%%%%%%%%%%%%%%%%%%%%%%%%%%%%%%%%%%%%%%%%%%%%%%%%%%%%%%%%%%%%%%%%%%%%%%%%%%%%%%%%%%%%%%%%%%%%%%%%

\end{document}